\def\qkpierisz{10pt}
\IfFileExists{qkpieri.conf}{\input{qkpieri.conf}}{}
\documentclass[\qkpierisz]{amsart}
\usepackage{amsmath, amsfonts, amssymb}
\usepackage{stmaryrd}
\usepackage{accents}
\usepackage[all]{xy}
\usepackage{graphicx}
\usepackage{hhline}
\usepackage{tableau}
\usepackage{tikz}
\usetikzlibrary{patterns,calc}
\usepackage[colorlinks=true, linkcolor=blue,
  citecolor=blue, urlcolor=blue]{hyperref}
\ifcsname mypreamble\endcsname\mypreamble\fi

\def\newthm#1#2{\newtheorem{#1}[dummy]{#2}%
  \expandafter\def\csname#2\endcsname##1{\hyperref[#1:##1]{#2~\ref*{#1:##1}}}}
\newthm{thm}{Theorem}
\newthm{lemma}{Lemma}
\newthm{prop}{Proposition}
\newthm{cor}{Corollary}
\newthm{conj}{Conjecture}
\newthm{cons}{Consequence}
\newthm{claim}{Claim}
\newthm{fact}{Fact}
\newthm{obs}{Observation}
\newthm{guess}{Guess}
\theoremstyle{definition}
\newthm{defn}{Definition}
\newthm{remark}{Remark}
\newthm{example}{Example}
\newthm{question}{Question}
\newthm{notation}{Notation}
\newthm{problem}{Problem}

\newcommand{\Section}[1]{\hyperref[sec:#1]{Section~\ref*{sec:#1}}}
\newcommand{\Sec}[1]{\hyperref[sec:#1]{\S\ref*{sec:#1}}}
\newcommand{\Table}[1]{\hyperref[tab:#1]{Table~\ref*{tab:#1}}}
\newcommand{\Figure}[1]{\hyperref[fig:#1]{Figure~\ref*{fig:#1}}}
\newcommand{\eqn}[1]{\hyperref[eqn:#1]{(\ref*{eqn:#1})}}

\newcounter{symbolpage}
\def\spage#1{\ifnum\value{symbolpage}=\getpagerefnumber{#1}\else\hyperref[#1]{\textcolor{blue}{[\pageref*{#1}]}}\, \fi\setcounter{symbolpage}{\getpagerefnumber{#1}}}
\newif\ifsnext
\def\scomma{\ifsnext,\, \fi\snexttrue}

\newcommand{\targetsec}[2]{\hypertarget{link:#1}{\phantomsection\label{page:#1}#2}}
\newcommand{\indexsec}[2]{\snextfalse\noin\Section{#1}: {\hypersetup{hidelinks}#2}.}
\newcommand{\slink}[2]{\scomma\spage{page:#1}\hyperlink{link:#1}{#2}}
\newcommand{\sref}[2]{\scomma\spage{#1}\hyperref[#1]{#2}}

\DeclareMathOperator{\Sp}{Sp}

\DeclareMathOperator{\Gr}{Gr}
\DeclareMathOperator{\Fl}{Fl}
\DeclareMathOperator{\LG}{LG}

\DeclareMathOperator{\SG}{SG}
\DeclareMathOperator{\SF}{SF}
\DeclareMathOperator{\OG}{OG}

\DeclareMathOperator{\Span}{Span}

\DeclareMathOperator{\QH}{QH}
\DeclareMathOperator{\QK}{QK}

\DeclareMathOperator{\codim}{codim}

\DeclareMathOperator{\dist}{dist}

\newcommand{\comin}{\mathrm{comin}}

\newcommand{\ssm}{\smallsetminus}
\newcommand{\bP}{{\mathbb P}}

\newcommand{\C}{{\mathbb C}}

\newcommand{\Z}{{\mathbb Z}}
\newcommand{\N}{{\mathbb N}}
\newcommand{\cA}{{\mathcal A}}
\newcommand{\cB}{{\mathcal B}}
\newcommand{\cC}{{\mathcal C}}

\newcommand{\cF}{{\mathcal F}}
\newcommand{\cH}{{\mathcal H}}
\newcommand{\cI}{{\mathcal I}}

\newcommand{\cN}{{\mathcal N}}
\newcommand{\cO}{{\mathcal O}}
\newcommand{\cP}{{\mathcal P}}

\newcommand\oY{\accentset{\circ}{Y}}

\newcommand{\euler}[1]{\chi_{_{#1}}}
\newcommand{\pt}{\mathrm{point}}

\newcommand{\al}{{\alpha}}
\newcommand{\be}{{\beta}}
\newcommand{\ga}{{\gamma}}

\newcommand{\ka}{{\kappa}}
\newcommand{\la}{{\lambda}}

\newcommand{\om}{{\omega}}

\newcommand{\dmin}{d_{\min}}
\newcommand{\dmax}{d_{\max}}

\newcommand{\ev}{\operatorname{ev}}

\newcommand{\wh}{\widehat}
\newcommand{\wb}{\overline}
\newcommand{\ov}{\overline}
\newcommand{\pic}[2]{\includegraphics[scale=#1]{#2}}
\newcommand{\ignore}[1]{}

\newcommand{\Mb}{\wb{\mathcal M}}

\newcommand{\noin}{\noindent}

\newcommand{\sP}{{\tt P}}
\newcommand{\sQ}{{\tt Q}}
\newcommand{\pP}{{\scriptstyle\mathbf P}}

\begin{document}

\title{Seidel and Pieri products in cominuscule quantum $K$-theory}

\ifdefined\gitdate
\date{\gitdate\ revision {\tt \gittag}}
\else
\date{March 21, 2026}
\fi

\author{Anders~S.~Buch}
\address{Department of Mathematics, Rutgers University, 110
  Frelinghuysen Road, Piscataway, NJ 08854, USA}
\email{asbuch@math.rutgers.edu}

\author{Pierre--Emmanuel Chaput}
\address{Domaine Scientifique Victor Grignard, 239, Boulevard des
  Aiguillettes, Universit{\'e} de Lorraine, B.P.  70239,
  F-54506 Vandoeuvre-l{\`e}s-Nancy Cedex, France}
\email{pierre-emmanuel.chaput@univ-lorraine.fr}

\author{Nicolas Perrin}
\address{Centre de Math\'ematiques Laurent Schwartz (CMLS), CNRS, \'Ecole
polytechnique, Institut Polytechnique de Paris, 91120 Palaiseau, France}
\email{nicolas.perrin.cmls@polytechnique.edu}



\thanks{Buch was partially supported by NSF grants DMS-1205351, DMS-1503662,
DMS-2152316, a Visiting Professorship at Universit{\'e} de Lorraine, and NSF
grant DMS-1929284 while in residence at the Institute for Computational and
Experimental Research in Mathematics in Providence, RI, during the Spring of
2021. Perrin was partially supported by ANR project FanoHK, grant
ANR-20-CE40-0023.}

\begin{abstract}
  We prove a collection of formulas for products of Schubert classes in the
  quantum $K$-theory ring $\QK(X)$ of a cominuscule flag variety $X$. This
  includes a $K$-theory version of the Seidel representation, stating that the
  quantum product of a Seidel class with an arbitrary Schubert class is equal to
  a single Schubert class times a power of the deformation parameter $q$. We
  also prove new Pieri formulas for the quantum $K$-theory of maximal orthogonal
  Grassmannians and Lagrangian Grassmannians, and give a new proof of the known
  Pieri formula for the quantum $K$-theory of Grassmannians of type A. Our
  formulas have simple statements in terms of quantum shapes that represent the
  natural basis elements $q^d[\cO_{X^u}]$ of $\QK(X)$. Along the way we give a
  simple formula for $K$-theoretic Gromov-Witten invariants of Pieri type for
  Lagrangian Grassmannians, and prove a rationality result for the points in a
  Richardson variety in a symplectic Grassmannian that are perpendicular to a
  point in projective space.
\end{abstract}

\maketitle

\setcounter{tocdepth}{1}
\tableofcontents


\section{Introduction}

In this paper we prove a collection of explicit formulas for products of
Schubert classes in the quantum $K$-theory ring $\QK(X)$ of a cominuscule flag
variety. These formulas include a $K$-theory version of the Seidel
representation of the quantum cohomology ring $\QH(X)$ \cite{seidel:1,
belkale:transformation, chaput.manivel.ea:affine}, as well as Pieri formulas for
products with special Schubert classes of classical Grassmannians that
generalize earlier Pieri formulas in quantum cohomology \cite{bertram:quantum,
kresch.tamvakis:quantum, kresch.tamvakis:quantum*1} and in $K$-theory
\cite{lenart:combinatorial, buch.ravikumar:pieri}. The Pieri formula for
$\QK(X)$ is known from \cite{buch.mihalcea:quantum} when $X$ is a Grassmannian
of type A, but is new for maximal orthogonal Grassmannians and Lagrangian
Grassmannians. Our formulas have simple expressions in terms of \emph{quantum
shapes} that encode the natural basis elements $q^d\cO^u = q^d[\cO_{X^u}]$ of
$\QK(X)$, generalizing the familiar identification of cominuscule Schubert
classes with diagrams of boxes \cite{proctor:bruhat}.

Let $X = G/P_X$ be a flag variety defined by a complex semisimple linear
algebraic group $G$ and a parabolic subgroup $P_X$. Let $\Phi$ be the root
system of $G$, $W$ the Weyl group, and let $B$ be a Borel subgroup contained in
$P_X$. A simple root $\ga$ is called \emph{cominuscule} if, when the highest
root is expressed in the basis of simple roots, the coefficient of $\ga$ is one.
The flag variety $X$ is called \emph{cominuscule} if $P_X$ is a maximal
parabolic subgroup defined by a cominuscule simple root. Let $w_0^X \in W$ be
the minimal representative of the longest element $w_0$ modulo the Weyl group
$W_X$ of $P_X$. The minimal representatives $w_0^F$ defined by all cominuscule
flag varieties of $G$, together with the identity, form a subgroup of the Weyl
group:
\[
  W^\comin \,=\,
  \{ w_0^F \mid \text{$F = G/P_F$ is cominuscule} \} \cup \{1\} \,\leq\, W \,.
\]

Each element $u \in W$ defines the Schubert varieties $X_u = \ov{B u.P_X}$ and
$X^u = \ov{B^- u.P_X}$ in $X$. The Schubert classes $[X^w]$ for $w \in W^\comin$
will be called \emph{Seidel classes}. It was proved in
\cite{belkale:transformation} and also in \cite{chaput.manivel.ea:affine} that
quantum cohomology products with Seidel classes have only one term. More
precisely, for $w \in W^\comin$ and $u \in W$ we have $[X^w] \star [X^u] =
q^{\omega^\vee-u^{-1}.\omega^\vee} [X^{w u}]$ in $\QH(X)$, where $\omega^\vee$
is the unique fundamental coweight such that $w.\omega^\vee = w_0.\omega^\vee$.
This defines a representation of $W^\comin$ on $\QH(X)/\langle q-1\rangle$
called the \emph{Seidel representation}. Our first result generalizes the Seidel
representation to the quantum $K$-theory ring when $X$ is itself cominuscule. We
denote the Schubert classes in $K(X)$ by $\cO_u = [\cO_{X_u}]$ and $\cO^u =
[\cO_{X^u}]$.

\begin{thm}[Seidel representation]\label{thm:seidel}%
  Let $X = G/P_X$ be a cominuscule flag variety, and let $w \in W^\comin$ and $u
  \in W$. We have in $\QK(X)$ that
  \[
    \cO^w \star \cO^u \,=\, q^d\,\cO^{w u} \,,
  \]
  where $d$ is determined by\, $\int_d c_1(T_X) + \codim(X^{w u}) = \codim(X^w)
  + \codim(X^u)$.
\end{thm}

This result has been generalized to the equivariant quantum $K$-theory of
cominuscule flag varieties in \cite{buch.chaput.ea:equivariant}, where we also
conjecture a further generalization to arbitrary flag varieties. The action of
$W^\comin$ on $\QH(X)/\langle q-1 \rangle$ was named the \emph{Seidel
representation} in \cite{chaput.manivel.ea:affine}, due to the analogy with
Seidel's work on the quantum cohomology of symplectic manifolds \cite{seidel:1}.
It is interesting to ask if the $K$-theoretic version also has a symplectic
analogue. A version of quantum $K$-theory for symplectic manifolds has been
constructed in \cite{abouzaid.mclean.ea:gromov-witten}.

When $X = G/P_X$ is a cominuscule flag variety, the subset $W^X \subset W$ of
minimal representatives of the cosets in $W/W_X$ can be represented by
generalized Young diagrams \cite{proctor:bruhat, perrin:small*1,
buch.samuel:k-theory}. Set $\cP_X = \{ \al \in \Phi \mid \al \geq \ga \}$, where
$\ga$ is the cominuscule simple root defining $X$, and give $\cP_X$ the partial
order $\al' \leq \al$ if and only if $\al-\al'$ is a sum of positive roots. The
inversion set $I(u) = \{ \al \in \Phi^+ \mid u.\al < 0 \}$ of any element $u \in
W^X$ is a lower order ideal in $\cP_X$. The set $\cP_X$ can be identified with a
set of boxes in the plane, which in turn identifies $I(u)$ with a diagram of
boxes that we call the \emph{shape} of $u$. This defines a bijection between the
set of shapes in $\cP_X$ and the Schubert basis $\{[X^u]\}$ of $H^*(X,\Z)$.

More generally, let $\cB = \{ q^d[X^u] \mid u \in W^X, d \in \Z \}$ be the
natural $\Z$-basis of $\QH(X)_q = \QH(X) \otimes \Z[q,q^{-1}]$. It was shown in
\cite{buch.chaput.ea:positivity} that $\cB$ has a natural partial order defined
by $q^e[X^v] \leq q^d[X^u]$ if and only if $X_u$ and $X^v$ can be connected by a
rational curve of degree at most $d-e$. Moreover, this partial order is a
distributive lattice when $X$ is cominuscule. Let $\wh\cP_X \subset \cB$ be the
subset of join-irreducible elements. Then $\wh\cP_X$ is an infinite partially
ordered set that contains $\cP_X$ as an interval. When $X = \Gr(m,n)$ is a
Grassmannian of type A, $\wh\cP_X = \Z^2/\Z(m,m-n)$ is Postnikov's
\emph{cylinder} from \cite{postnikov:affine}. This poset was also defined in
\cite{hagiwara:minuscule*2}. The posets $\wh\cP_X$ defined by other cominuscule
flag varieties are isomorphic to certain \emph{full heaps} of affine Dynkin
diagrams that were constructed in \cite{green:combinatorics*1} and used to study
minuscule representations.

Define a \emph{quantum shape} to be any (non-empty, proper, lower) order ideal
$\la \subset \wh\cP_X$. A quantum shape will also be called a \emph{shape} when
it cannot be misunderstood to be a classical shape in $\cP_X$. The assignment
\[
  I(q^d[X^u]) \,=\, \{ \wh\al \in \wh\cP_X \mid \wh\al \leq q^d[X^u] \}
\]
defines an order isomorphism from $\cB$ to the set of shapes in $\wh\cP_X$,
where shapes are ordered by inclusion. We write $\cO^\la = q^d \cO^u$ when $\la
= I(q^d[X^u])$ is the quantum shape of $q^d[X^u]$.

Quantum multiplication by any Seidel class $\sigma$ defines an order
automorphism of $\cB$, which restricts to an order automorphism of $\wh\cP_X$.
If $\la \subset \wh\cP_X$ is any quantum shape, then $\sigma \star \la = \{
\sigma \star \wh\al \mid \wh\al \in \la \}$ defines a new quantum shape such
that
\[
  \sigma \star \cO^\la \,=\, \cO^{\sigma \star \la} \,.
\]
Here we have abused notation and identified $\sigma$ with the corresponding
$K$-theory class $\cO^{I(\sigma)} \in \QK(X)$. The poset $\wh\cP_X$ can be
identified with an infinite set of boxes in the plane, such that each
automorphism defined by a Seidel class is represented by a translation of the
plane, possibly combined with a reflection. This gives a simple description of
products with Seidel classes in terms of quantum shapes, see \Section{qposet},
\Example{seidel-trans}, \Example{gr:shift}, \Example{og:shift}, and
\Example{lg:shift}.

Let $X = G/P_X$ be a cominuscule classical Grassmannian, that is, a Grassmannian
$\Gr(m,n)$ of type A, a maximal orthogonal Grassmannian $\OG(n,2n)$, or a
Lagrangian Grassmannian $\LG(n,2n)$. The Chern classes of the tautological
vector bundles over $X$ are represented by the \emph{special} Schubert varieties
$X^p \subset X$, with $p \in \N$. Formulas for products with the special
Schubert classes $[X^p]$ are known as \emph{Pieri formulas}. Our Pieri formula
for $\QK(X)$ takes the form
\[
  \cO^p \star \cO^\la \,=\, \sum_\nu c(\nu/\la,p)\, \cO^\nu \,,
\]
where the sum is over all quantum shapes $\nu$ containing $\la$. The coefficient
$c(\nu/\la,p)$ depends on $p$ as well as the \emph{skew shape} $\nu/\la := \nu
\ssm \la \subset \wh\cP_X$. For Grassmannians of type A and maximal orthogonal
Grassmannians, these coefficients $c(\nu/\la,p)$ are identical to those
appearing in the Pieri formulas for the ordinary $K$-theory ring. These
coefficients are signed binomial coefficients in type A
\cite{lenart:combinatorial}, and are signed counts of KOG-tableaux of shape
$\nu/\la$ for maximal orthogonal Grassmannians \cite{buch.ravikumar:pieri}. In
fact, in these cases the Pieri formula for $\QK(X)$ is an easy consequence of
\Theorem{seidel}, the Pieri formula for $K(X)$, and a bound on the $q$-degrees
in cominuscule quantum products proved in \cite{buch.chaput.ea:positivity}.

Assume from now on that $X = \LG(n,2n)$ is a Lagrangian Grassmannian. In this
case our Pieri formula for $\QK(X)$ is more difficult to state and prove. While
the coefficients of the Pieri formula for $K(X)$ are expressed as signed counts
of KLG-tableaux in \cite{buch.ravikumar:pieri}, we need to amend the definition
of KLG-tableau with additional conditions in the quantum case. The tableaux
satisfying these conditions will be called \emph{QKLG-tableaux}. Another
difference is that the Lagrangian Grassmannian $X$ does not have enough Seidel
classes to translate the Pieri formula for $K(X)$ to one for $\QK(X)$. We must
therefore prove our quantum Pieri formula `from scratch', starting with a
geometric computation of the relevant $K$-theoretic Gromov-Witten invariants,
and then use combinatorics to translate these Gromov-Witten invariants to the
structure constants $c(\nu/\la,p)$ of Pieri products. While both parts resemble
the proof of the Pieri formula from \cite{buch.ravikumar:pieri}, the technical
challenges are harder for several reasons, and many steps rely on results proved
in \cite{buch.chaput.ea:positivity}.

Our computation of Gromov-Witten invariants of $X=\LG(n,2n)$ targets those of
the form
\[
  I_d(\cO^p, \cO^v, \cO_u) \,=\, \chi(\ev_1^*(\cO^p) \cdot \ev_2^*(\cO^v) \cdot
  \ev_3^*(\cO_u)) \,,
\]
where $\ev_1, \ev_2, \ev_3 : \Mb_{0,3}(X,d) \to X$ are the evaluation maps from
the Kontsevich moduli space. By \cite{buch.chaput.ea:projected}, these can be
computed as
\[
  I_d(\cO^p, \cO^v, \cO_u) \,=\,
  \euler{X}([\cO_{\Gamma_d(X_u,X^v)}] \cdot \cO^p) \,,
\]
where the \emph{curve neighborhood} $\Gamma_d(X_u,X^v) \subset X$ is defined as
the union of all stable curves of degree $d$ connecting $X_u$ and $X^v$. Let
$\wh X = \SF(1,n;2n)$ be the variety of two-step isotropic flags in the
symplectic vector space $\C^{2n}$, and let $\pi : \wh X \to X$ and $\eta : \wh X
\to \bP^{2n-1}$ be the projections. We then have $\cO^p = \pi_* \eta^*([\cO_L])$
for any linear subspace $L \subset \bP^{2n-1}$ of dimension $n-p$. The
projection formula therefore gives
\[
  I_d(\cO^p, \cO^v, \cO_u) \,=\,
  \euler{\bP^{2n-1}}(\eta_*\pi^*[\cO_{\Gamma_d(X_u,X^v)}] \cdot [\cO_L]) \,.
\]

We compute the right hand side by showing that the restricted map
\begin{equation}\label{eqn:intro:psi}%
  \eta : \pi^{-1}(\Gamma_d(X_u,X^v)) \,\to\, \eta(\pi^{-1}(\Gamma_d(X_u,X^v)))
\end{equation}
is cohomologically trivial, and that its image is a complete intersection in
$\bP^{2n-1}$ defined by explicitly determined equations. More precisely, define
the skew shape $\theta = I(q^d[X^u]) / I([X^v])$ in $\wh\cP_X$, let $N(\theta)$
be the number of components of $\theta$ that are disjoint from the two diagonals
in $\wh\cP_X$ (\Section{lagrange}), and let $R(\theta)$ be the size of a maximal
rim contained in $\theta$. Assuming that $R(\theta) \leq n$, we show that
$\eta(\pi^{-1}(\Gamma_d(X_u,X^v)))$ is a complete intersection in $\bP^{2n-1}$
defined by $N(\theta)$ quadratic equations and $n-R(\theta)-N(\theta)$ linear
equations. This gives the formula
\begin{equation}\label{eqn:intro:pierigw}%
  I_d(\cO^p, \cO^v, \cO_u) \,=\,
  \chi(\cO_{L \,\cap\, \eta(\pi^{-1}(\Gamma_d(X_u,X^v)))}) \,=\,
  \sum_{j=0}^{R(\theta)-p} (-1)^j\, 2^{N(\theta)-j}
  \textstyle{\binom{N(\theta)}{j}} \,.
\end{equation}
An analogue of this identity for maximal orthogonal Grassmannians has been
obtained with similar methods in {\c T}arigradschi's thesis
\cite{tarigradschi:curve}.

In the special case $d=0$ we have $\Gamma_d(X_u,X^v) = X_u \cap X^v$, so
\eqn{intro:psi} is the projection of a Richardson variety in $\wh X$. This map
was proved to be cohomologically trivial in \cite{buch.ravikumar:pieri} by
showing that its general fibers are themselves Richardson varieties. This result
has been generalized to arbitrary projections of Richardson varieties, see
\cite{billey.coskun:singularities, knutson.lam.ea:projections} and
\cite[Thm.~2.10]{buch.chaput.ea:positivity}. However, the variety
$\pi^{-1}(\Gamma_d(X_u,X^v))$ for $d>0$ is not a Richardson variety, and it is
difficult to determine the fibers of the projection \eqn{intro:psi}.

Let $Y_d = \SG(n-d,2n)$ be the symplectic Grassmannian of isotropic subspaces of
dimension $n-d$ in $\C^{2n}$, set $Z_d = \SF(n-d,n;2n)$, and let $p_d : Z_d \to
X$ and $q_d : Z_d \to Y_d$ be the projections. By the quantum-to-classical
construction (see \cite[\S5]{buch.chaput.ea:positivity} and references therein)
we have $\Gamma_d(X_u,X^v) = p_d(q_d^{-1}(R))$, where $R = q_d(p_d^{-1}(X_u))
\cap q_d(p_d^{-1}(X^v))$ is a Richardson variety in $Y_d$. Define the
\emph{perpendicular incidence variety}
\[
  S \,=\, \{(K,L) \in Y_d \times \bP^{2n-1} \mid K \subset L^\perp \} \,,
\]
and let $f : S \to \bP^{2n-1}$ and $g : S \to Y_d$ be the projections. We then
have $f(g^{-1}(R)) = \eta(\pi^{-1}(\Gamma_d(X_u,X^v)))$.

We prove that for any Richardson variety $R \subset Y_d$, the general fibers of
the map $f : g^{-1}(R) \to f(g^{-1}(R))$ are rational, and the image
$f(g^{-1}(R))$ is a complete intersection in $\bP^{2n-1}$ defined by explicitly
given linear and quadratic equations. The required properties of the projection
\eqn{intro:psi} are deduced from this result. Our results about perpendicular
incidences of Richardson varieties in $Y_d$ are stronger than required for this
paper, but of independent interest. For example, the fibers of $f : g^{-1}(R)
\to f(g^{-1}(R))$ is a plausible definition of Richardson varieties in the odd
symplectic Grassmannian $\SG(n-d,2n-1)$. Notice also that $S$ is not a flag
variety, so $f(g^{-1}(R))$ is not a projected Richardson variety.

A final step in our proof of the Pieri formula for $\QK(X)$ is to translate the
formula \eqn{intro:pierigw} for Gromov-Witten invariants of Pieri type to a
formula for the Pieri coefficients $c(\nu/\la,p)$. We first show that the
structure constants $I_d(\cO^p,\cO^v,\cI_u)$ of the undeformed product $\cO^p
\odot \cO^v$ (see \Section{qktheory}) are determined by recursive identities.
These identities are used to prove that the Pieri coefficients $c(\nu/\la,p)$
satisfy analogous recursive identities. The Pieri formula for $\QK(X)$ then
follows by checking that the signed counts of QKLG-tableaux satisfy the same
identities.

This paper is organized as follows. In \Section{comin} we fix our notation for
flag varieties and discuss preliminaries. \Section{seidel} contains the proof of
\Theorem{seidel}. In \Section{qposet} we define quantum shapes in the partially
ordered set $\wh\cP_X$, and explain how quantum multiplication by Seidel classes
correspond to order automorphisms of this set. The Pieri formulas for $\QK(X)$
are given in \Section{grass} for Grassmannians of type A, in \Section{maxog} for
maximal orthogonal Grassmannians, and in \Section{lagrange} for Lagrangian
Grassmannians. These sections also explain in detail how the posets $\wh\cP_X$
for the classical Grassmannians are identified with sets of boxes in the plane.
While the Pieri formulas for $\Gr(m,n)$ and $\OG(n,2n)$ have short proofs given
after their statements, the proof of the Pieri formula for Lagrangian
Grassmannians is given in the last three sections. \Section{incidence} proves
that the map $f : g^{-1}(R) \to f(g^{-1}(R))$ is cohomologically trivial and
identifies its image as a complete intersection in $\bP^{2n-1}$.
\Section{pierigw} uses this result to prove the formula \eqn{intro:pierigw} for
Gromov-Witten invariants $I_d(\cO^p,\cO^v,\cO_u)$ of Pieri type. Finally,
\Section{combin} proves the recursive identities that determine the invariants
$I_d(\cO^p,\cO^v,\cI_u)$ and the Pieri coefficients $c(\nu/\la,p)$.

We thank Leonardo Mihalcea for inspiring collaboration on many related papers
about quantum $K$-theory, as well as many helpful comments to this paper. We
also thank Mihail {\c T}arigradschi for helpful comments. We thank Prakash
Belkale and Robert Proctor for making us aware of the references
\cite{belkale:transformation, hagiwara:minuscule*2, green:combinatorics*1}. We
finally thank an anonymous referee for a careful reading of our manuscript and
several excellent suggestions.


\section{Cominuscule flag varieties}\label{sec:comin}%

In this section we summarize some basic notation and definitions. We follow the
notation of \cite{buch.chaput.ea:positivity}.

\subsection{Flag varieties}\label{sec:flagvar}%

\targetsec{group}{%
Let $G$ be a connected semisimple linear algebraic group over $\C$, and fix a
Borel subgroup $B$ and a maximal torus $T$ such that $T \subset B \subset G$.
The opposite Borel subgroup $B^- \subset G$ is determined by $B \cap B^- = T$.
Let $W$ be the Weyl group of $G$ and let $\Phi$ be the root system, with simple
roots $\Delta \subset \Phi^+$.}

\targetsec{flagvar}{%
A \emph{flag variety} of $G$ is a projective variety with a transitive
$G$-action. Given a flag variety $X$ of $G$, we let $P_X \subset G$ denote the
stabilizer of the unique $B$-fixed point in $X$. We obtain the identification $X
= G/P_X = \{ g.P_X \mid g \in G \}$, where $g.P_X$ is the $g$-translate of the
$B$-fixed point.}

\targetsec{schub}{%
Let $W_X \subset W$ be the Weyl group of $P_X$ and let $W^X \subset W$ be the
set of minimal representatives of the cosets in $W/W_X$. Each element $w \in W$
defines the Schubert varieties
\[
  X_w = \ov{Bw.P_X} \text{ \ \ \ \ and \ \ \ \ } X^w = \ov{B^-w.P_X} \,,
\]
and for $w \in W^X$ we have $\dim(X_w) = \codim(X^w,X) = \ell(w)$. The Bruhat
order on $W^X$ is defined by $v \leq u$ if and only if $X_v \subset X_u$.}

\targetsec{factor}{%
Any element $u \in W$ has a unique parabolic factorization $u = u^X u_X$, where
$u^X \in W^X$ and $u_X \in W_X$. The parabolic factorization of the longest
element $w_0 \in W$ is $w_0 = w_0^X w_{0,X}$, where $w_0^X$ is the longest
element in $W^X$ and $w_{0,X}$ is the longest element in $W_X$. We have $w_0.X^u
= X_{u^\vee}$ for any $u \in W^X$, where $u^\vee = w_0\, u\, w_{0,X} \in W^X$
denotes the Poincare dual basis element.}

\begin{lemma}\label{lemma:projection}%
  Let $Z = G/P_Z$ be any flag variety with $P_Z \subset P_X$, and let $p : Z \to
  X$ be the projection. Let $F = p^{-1}(1.P_X) = P_X/P_Z$ denote the fiber over
  $1.P_X$, considered as a flag variety of $P_X$. Let $u \in W^X$ and $w \in
  W^Z$.\smallskip

  \noin{\rm(a)} We have $p(Z_w) = X_w = X_{w^X}$, and the general fibers of $p :
  Z_w \to X_w$ are translates of $F_{w_X} = Z_{w_X}$.\smallskip

  \noin{\rm(b)} We have $p(Z^w) = X^w = X^{w^X}$, and the general fibers of $p :
  Z^w \to X^w$ are translates of $F^{w_X}$.\smallskip

  \noin{\rm(c)} The map $p : Z^w \to X^w$ is birational if and only if $w_X =
  w_{0,X}^Z := (w_{0,X})^Z$.\smallskip

  \noin{\rm(d)} We have $p^{-1}(X_u) = Z_{u\,w_{0,X}^Z}$, and $u\,w_{0,X}^Z \in
  W^Z$.\smallskip

  \noin{\rm(e)} We have $p^{-1}(X^u) = Z^u$, and $u \in W^Z$.
\end{lemma}
\begin{proof}
  Parts (a) and (b) are \cite[Thm.~2.8 and
  Remark~2.9]{buch.chaput.ea:positivity}, and part (c) follows from (b). Parts
  (d) and (e) hold because the $T$-fixed points in $p^{-1}(u.P_X)$ are the
  points of the form $ut.P_Z$, with $t \in W_X$.
\end{proof}

\begin{prop}\label{prop:homschub}%
  Let $Y = G/P_Y$ and $X = G/P_X$ be flag varieties, let $u \in W^Y$, and assume
  that $(P_X.P_Y) \cap Y^u \neq \emptyset$. Then $(P_X.P_Y) \cap Y^u =
  (w_0^X)^{-1}.Y^v$, where $v = w_0^X u\, ((w_{0,Y})^X)^{-1} \in W^Y$. In
  particular, $(P_X.P_Y) \cap Y^u$ is a Schubert variety in $Y$.
\end{prop}
\begin{proof}
  Set $Z = G/(P_X \cap P_Y)$, let $p : Z \to X$ and $q : Z \to Y$ be the
  projections, and set $F = p^{-1}(1.P_X) = P_X/P_Z$. Let $t = w_0 u w_{0,Z}$ be
  the Poincare dual element of $u$ in $W^Z$. By
  \cite[Thm.~2.8]{buch.chaput.ea:positivity} we have $t.F \cap Z_t =
  t^X.Z_{t_X}$. The assumption $P_X.P_Y \cap Y^u \neq \emptyset$ implies that
  $p(Z^u) = X$, hence $t^X = w_0^X$ and $t_X = (w_0^X)^{-1}t$. We obtain
  \[
    F \cap Z^u = w_0.(t.F \cap Z_t) = w_0.(t^X.Z_{t_X})
    = (w_0 t^X w_0).Z^{w_0 t_X w_{0,Z}} = (w_0^X)^{-1}.Z^{w_0^X u} \,,
  \]
  where $w_0 t_X w_{0,Z} = w_0^X u$ belongs to $W^Z$. Since $q : F \cap Z^u \to
  (P_X.P_Y) \cap Y^u$ is an isomorphism, it follows from \Lemma{projection}(c)
  that $(w_0^X u)_Y = w_{0,Y}^Z = (w_{0,Y})^X$ and $(w_0^X u)^Y = w_0^X u
  ((w_{0,Y})^X)^{-1}$. The result now follows from \Lemma{projection}(b).
\end{proof}

\subsection{Cominuscule flag varieties}\label{sec:ss:comin}%

\targetsec{gamma}{%
A simple root $\ga \in \Delta$ is called \emph{cominuscule} if the coefficient
of $\ga$ is one when the highest root of $\Phi$ is expressed in the basis of
simple roots. The flag variety $X = G/P_X$ is called cominuscule if $P_X$ is a
maximal parabolic subgroup corresponding to a cominuscule simple root $\ga$,
that is, $s_\ga$ is the unique simple reflection in $W^X$. A cominuscule flag
variety $X$ is also called \emph{minuscule} if the root system $\Phi$ is simply
laced. In the remainder of this section we assume that $X = G/P_X$ is the
cominuscule flag variety defined by the cominuscule simple root $\ga \in
\Delta$.}

\targetsec{comin}{%
The Bruhat order on $W^X$ is a distributive lattice \cite{proctor:bruhat} with
meet and join operations defined by $X_{u \cap v} = X_u \cap X_v$ and $X^{u \cup
v} = X^u \cap X^v$ for $u,v \in W^X$. The minimal representatives in $W^X$ can
be identified with shapes of boxes as follows \cite{proctor:bruhat,
perrin:small*1, buch.samuel:k-theory}. The root system $\Phi$ has a natural
partial order defined by $\al' \leq \al$ if and only if $\al - \al'$ is a sum of
positive roots. Let $\cP_X \subset \Phi^+$ be the subset
\[
  \cP_X = \{ \al \in \Phi \mid \al \geq \ga \} \,,
\]
with the induced partial order (see \Table{tablez1}). A lower order ideal $\la
\subset \cP_X$ is called a \emph{shape} in $\cP_X$. There is a natural bijection
between $W^X$ and the set of shapes in $\cP_X$ that sends $w \in W^X$ to its
inversion set
\[
  I(w) = \{ \al \in \Phi^+ \mid w.\al \in \Phi^- \} \,.
\]
This correspondence is compatible with the Bruhat order, so that $v \leq u$
holds in $W^X$ if and only if $I(v) \subset I(u)$. In addition, we have $\ell(w)
= |I(w)|$. The elements of $\cP_X$ will frequently be called \emph{boxes}. There
exists a natural labeling $\delta : \cP_X \to \Delta$ defined by $\delta(\al) =
w.\al$, where $w \in W^X$ is the unique element with shape $I(w) = \{ \al' \in
\cP_X : \al' < \al \}$. Given $u \in W^X$, write $I(u) = \{ \ga = \al_1, \al_2,
\dots, \al_\ell \}$, where the boxes of $I(u)$ are listed in non-decreasing
order, that is, $\al_i \leq \al_j$ implies $i \leq j$. Then $u =
s_{\delta(\al_\ell)} \cdots s_{\delta(\al_2)} s_{\delta(\al_1)}$ is a reduced
expression for $u$.}

\targetsec{schubla}{%
If $\la \subset \cP_X$ is any shape and $w \in W^X$ is the corresponding element
with $I(w) = \la$, then the Schubert varieties defined by $w$ will also be
denoted by
\[
  X_\la = X_w \text{ \ \ \ \ and \ \ \ \ } X^\la = X^w \,.
\]
}


\def\vmm#1{\vspace{#1mm}}
\begin{table}
\caption{Partially ordered sets $\cP_X$ of cominuscule varieties with $I(z_1)$
  highlighted. In each case the partial order is given by $\al' \leq \al$ if and
  only if $\al'$ is weakly north-west of $\al$.}
\label{tab:tablez1}
\begin{tabular}{|c|c|}
\hline
&\vmm{-2}\\
Grassmannian $\Gr(3,7)$ of type A & Max.\ orthog.\ Grassmannian $\OG(6,12)$
\\
&\vmm{-3}\\
\pic{1}{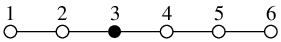} &\\
&\\
&\vmm{-3}\\
$\tableau{12}{
[aLlTt]3 & [aTtBb]4 & 5 & [aTtBbRr]6 \\
[aLlRr]2 & [a]3 & 4 & 5 \\
[aLlRrBb]1 & [a]2 & 3 & 4
}$
&\vmm{-27}\\
&\pic{1}{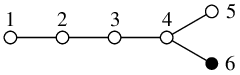}
\\ &
$\tableau{12}{
[aLlTtBb]6 & [aTt]4 & 3 & 2 & [aTtRr]1 \\
  & [aLlBb]5 & [aBb]4 & 3 & [aBbRr]2 \\
  &   & [a]6 & 4 & 3 \\
  &   &   & 5 & 4 \\
  &   &   &   & 6
}$
\\ &
\vmm{-2}\\
\hline
&\vmm{-2}\\
Lagrangian Grassmannian $\LG(6,12)$ & Cayley Plane $E_6/P_6$
\\
&\vmm{-3}\\
\pic{1}{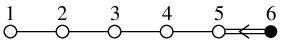} & \\
&\vmm{-2}\\
$\tableau{12}{
[aLlTtBb]6 & [aTtBb]5 & 4 & 3 & 2 & [aTtBbRr]1 \\
  & [a]6 & 5 & 4 & 3 & 2 \\
  &   & 6 & 5 & 4 & 3 \\
  &   &   & 6 & 5 & 4 \\
  &   &   &   & 6 & 5 \\
  &   &   &   &   & 6
}$
&\vmm{-36}\\
& \pic{1}{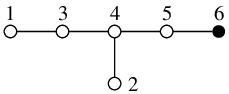} \\
&\vmm{-2}\\
&
$\tableau{12}{
[aLlTtBb]6 & [aTtBb]5 & [aTt]4 & [aTtRr]2 \\
  &   & [aLl]3 & [a]4 & [aTtRr]5 & [a]6 \\
  &   & [aLlBb]1 & [aBb]3 & [aRr]4 & [a]5 \\
  &   &   &   & [aLlBbRr]2 & [a]4 & 3 & 1
}$
\\
& \vmm{-2}\\
\hline
& \vmm{-2}\\
Even quadric $Q^{10} \subset \bP^{11}$ & Freudenthal variety $E_7/P_7$
\\
&\vmm{-3}\\
\pic{1}{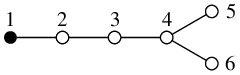} &
\\
&\vmm{-3}\\
$\tableau{12}{
[aLlTtBb]1 & [aTtBb]2 & 3 & [aTt]4 & [aTtRr]5 \\
  &   &   & [aLlBb]6 & [aBb]4 & [aTtBb]3 & [aTtBbRr]2 & [a]1
}$
& \\
&\\
\hhline{-~}
&\vmm{-2}\\
Odd quadric $Q^{11} \subset \bP^{12}$ &
\\
&\vmm{-3}\\
\pic{1}{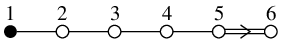} & \\
&\vmm{-1}\\
$\tableau{12}{
[aLlTtBb]1 & [aTtBb]2 & 3 & 4 & 5 & 6 & 5 & 4 & 3 & [aTtBbRr]2 & [a]1
}$
& \vmm{-52}\\
& \pic{1}{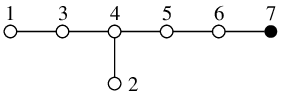} \\
&\vmm{-3}\\
&
$\tableau{12}{
[aLlTtBb]7 & [aTtBb]6 & 5 & [aTt]4 & 3 & [aTtRr]1 \\
  &   &   & [aLlBb]2 & [a]4 & [aRr]3 \\
  &   &   &   & [aLl]5 & [a]4 & [aTtRr]2 \\
  &   &   &   & [aLlBb]6 & [aBb]5 & 4 & [aTtBb]3 & [aTtBbRr]1 \\
  &   &   &   & [a]7 & 6 & 5 & 4 & 3 \\
  &   &   &   &   &   &   & 2 & 4 \\
  &   &   &   &   &   &   &   & 5 \\
  &   &   &   &   &   &   &   & 6 \\
  &   &   &   &   &   &   &   & 7
}$
\vmm{-2}\\
& \\
\hline
\end{tabular}
\end{table}

\subsection{Curve neighborhoods}\label{sec:curve-nbhd}%

\targetsec{nbhd}{%
Let $M_d = \Mb_{0,3}(X,d)$ denote the Kontsevich moduli space of 3-pointed
stable maps to $X$ of degree $d$ and genus zero, see
\cite{fulton.pandharipande:notes}. The evaluation maps are denoted $\ev_i : M_d
\to X$, for $1 \leq i \leq 3$. Given opposite Schubert varieties $X_u$ and $X^v$
in $X$ and a degree $d \geq 0$, let
\[
  M_d(X_u,X^v) = \ev_1^{-1}(X_u) \cap \ev_2^{-1}(X^v)
\]
be the \emph{Gromov-Witten variety} of stable maps that send the first two
marked points to $X_u$ and $X^v$, respectively. This variety is empty or
unirational with rational singularities \cite[\S3]{buch.chaput.ea:finiteness}.
The \emph{curve neighborhood}
\[
  \Gamma_d(X_u,X^v) = \ev_3(M_d(X_u,X^v))
\]
is the union of all stable curves of degree $d$ in $X$ that connect $X_u$ and
$X^v$. In particular, $\Gamma_d(X_u) = \Gamma_d(X_u,X)$ is the union of all
stable curves of degree $d$ that pass through $X_u$. Since this variety is a
Schubert variety in $X$ \cite[Prop.~3.2(b)]{buch.chaput.ea:finiteness}, we can
define elements $u(d), v(-d) \in W^X$ by
\[
  \Gamma_d(X_u) = X_{u(d)}
  \text{ \ \ \ \ and \ \ \ \ }
  \Gamma_d(X^v) = X^{v(-d)} \,.
\]
Define $z_d \in W^X$ by $\Gamma_d(1.P_X) = X_{z_d}$.}

\targetsec{dist}{%
The curve neighborhood $\Gamma_d(X_u,X^v)$ can be constructed as a projected
Richardson variety as follows \cite{buch.chaput.ea:projected}. Given $x, y \in
X$, let $\dist(x,y) \in H_2(X,\Z) = \Z$ denote the minimal degree of a rational
curve in $X$ that meets both $x$ and $y$. The \emph{diameter} of $X$ is the
distance $d_X(2) = \dist(1.P_X, w_0.P_X)$ between two general points. For $0
\leq d \leq d_X(2)$, we can choose points $x,y \in X$ with $\dist(x,y) = d$. Let
$\Gamma_d(x,y)$ be the union of all stable curves of degree $d$ that pass
through $x$ and $y$. Then $\Gamma_d(x,y)$ is a non-singular Schubert variety,
whose stabilizer $P_{Y_d}$ is a parabolic subgroup of $G$. The set of all
$G$-translates of $\Gamma_d(x,y)$ can therefore be identified with the flag
variety $Y_d = G/P_{Y_d}$. Let $Z_d = G/P_{Z_d}$ be the flag variety defined by
$P_{Z_d} = P_X \cap P_{Y_d}$, and let $p_d : Z_d \to X$ and $q_d : Z_d \to Y_d$
be the projections. Set
\[
  \begin{split}
    Y_d(X_u,X^v) &= q_d(p_d^{-1}(X_u)) \cap q_d(p_d^{-1}(X^v))
    \ \ \ \ \ \text{and} \\
    Z_d(X_u,X^v) &= q_d^{-1}(Y_d(X_u,X^v)) \,.
  \end{split}
\]
These varieties are Richardson varieties in $Y_d$ and $Z_d$. By
\cite[Thm.~4.1]{buch.chaput.ea:projected} and
\cite[Thm.~10.1]{buch.chaput.ea:positivity} we then have $\Gamma_d(X_u,X^v) =
p_d(Z_d(X_u,X^v))$, and the restricted projection
\begin{equation}\label{eqn:pd}
  p_d : Z_d(X_u,X^v) \to \Gamma_d(X_u,X^v)
\end{equation}
is cohomologically trivial. We let $\ka_d = (w_{0,Y_d})^X = w_{0,Y_d}^{Z_d} \in
W^X$ be the unique element such that $X_{\ka_d} = p(q^{-1}(1.P_{Y_d}))$ is a
translate of $\Gamma_d(x,y)$. A combinatorial description of the elements
$\ka_d, z_d \in W^X$ can be found in
\cite[Def.~5.2]{buch.chaput.ea:positivity}.}

\subsection{Quantum cohomology}\label{sec:qcohom}%

\targetsec{qh}{%
The (small) \emph{quantum cohomology ring} $\QH(X)$ is a $\Z[q]$-algebra, which
is defined by $\QH(X) = H^*(X,\Z) \otimes_\Z \Z[q]$ as a $\Z[q]$-module. When
$X$ is cominuscule, the multiplicative structure is given by
\[
  [X_u] \star [X^v] \,=\, \sum_{d \geq 0}
  (p_d)_* [Z_d(X_u,X^v)] \, q^d \,.
\]
This follows from the \emph{quantum equals classical} theorem
\cite{buch:quantum, buch.kresch.ea:gromov-witten, chaput.manivel.ea:quantum*1,
buch.mihalcea:quantum, chaput.perrin:rationality, buch.chaput.ea:projected}. A
mostly type-uniform proof was given in \cite{buch.chaput.ea:positivity}. Notice
that we have
\[
  (p_d)_* [Z_d(X_u,X^v)] \,=\, \begin{cases}
    [\Gamma_d(X_u,X^v)] &
    \text{if $\dim \Gamma_d(X_u,X^v) = \dim Z_d(X_u,X^v)$,} \\
    0 & \text{otherwise,}
  \end{cases}
\]
for example because the projection \eqn{pd} is cohomologically trivial. Let
\[
  \QH(X)_q = \QH(X) \otimes_{\Z[q]} \Z[q,q^{-1}]
\]
be the localization of $\QH(X)$ at the deformation parameter $q$. The set $\cB =
\{ q^d\, [X^u] \mid u \in W^X \text{ and } d \in \Z \}$ is a natural $\Z$-basis
of $\QH(X)_q$.}

\subsection{Quantum $K$-theory}\label{sec:qktheory}%

\targetsec{ktheory}{%
Let $K(X)$ denote the $K$-theory ring of algebraic vector bundles on $X$. Given
$u \in W$, we let $\cO_u = [\cO_{X_u}]$ and $\cO^u = [\cO_{X^u}]$ denote the
corresponding $K$-theoretic Schubert classes. For any shape $\la \subset \cP_X$,
we similarly write $\cO_\la = [\cO_{X_\la}]$ and $\cO^\la = [\cO_{X^\la}]$.}

\targetsec{qk}{%
The \emph{quantum $K$-theory ring} $\QK(X)$ is an algebra over the power series
ring $\Z\llbracket q \rrbracket$, which is given by $\QK(X) = K(X) \otimes_\Z
\Z\llbracket q \rrbracket$ as a $\Z\llbracket q \rrbracket$-module. An
\emph{undeformed product} on $\QK(X)$ is defined by
\[
  \cO_u \odot \cO^v \,=\, \sum_{d \geq 0} (p_d)_*[\cO_{Z_d(X_u,X^v)}] \, q^d
  \,=\, \sum_{d \geq 0} [\cO_{\Gamma_d(X_u,X^v)}] \, q^d \,.
\]
This product $\cO_u \odot \cO^v$ is not associative. Let $\psi : \QK(X) \to
\QK(X)$ be the \emph{line neighborhood operator}, defined as the $\Z\llbracket q
\rrbracket$-linear extension of the map $\psi = (\ev_2)_* (\ev_1)^* : K(X) \to
K(X)$, where $\ev_1$ and $\ev_2$ are the evaluation maps from $\Mb_{0,2}(X,1)$.
Equivalently, we have $\psi(\cO^u) = \cO^{u(-1)}$ for $u \in W^X$. Givental's
associative quantum product on $\QK(X)$ is then given by
\cite[Prop.~3.2]{buch.chaput.ea:chevalley}
\[
  \cO_u \star \cO^v \,=\, (1 - q \psi)(\cO_u \odot \cO^v) \,.
\]
Let $\QK(X)_q$ be the localization obtained by adjoining the inverse of $q$ to
$\QK(X)$. The set $\cB' = \{ q^d\, \cO^u \mid u \in W^X \text{ and } d \in \Z
\}$ is a $\Z$-basis of $\QK(X)_q$, in the sense that every element of $\cF \in \QK(X)_q$ can
be uniquely expressed as an infinite linear combination
\[
  \cF = \sum_{d \geq d_0} \sum_{u \in W^X} a_{u,d}\, q^d\, \cO^u
\]
of $\cB'$, with $a_{u,d} \in \Z$ and the degrees $d$ bounded below.}


\section{The Seidel representation on quantum $K$-theory}\label{sec:seidel}%

\targetsec{dminmax}{%
Let $X = G/P_X$ be a fixed cominuscule flag variety. In this section we prove
that certain products $\cO^u \star \cO^v$ in $\QK(X)$ are equal to a single
element $q^d\, \cO^w$ from $\cB'$. The same statement was proved in
\cite{belkale:transformation, chaput.manivel.ea:affine} for products of Schubert
classes in the quantum cohomology ring $\QH(M)$ of any flag variety $M = G/P_M$.
For $u, v \in W$ we let $\dmin(u,v)$ and $\dmax(u,v)$ denote the minimal and
maximal powers of $q$ in the quantum cohomology product $[X^u] \star [X^v] \in
\QH(X)$. Let $\dmax(u) = \dmax(u,w_0^X)$ be the maximal power of $q$ in $[X^u]
\star [1.P_X]$.}

\begin{lemma}\label{lemma:gammapt}%
  Let $u \in W^X$ and $\dmax(u) \leq d \leq d_X(2)$. Then $\Gamma_d(1.P_X, X^u)
  = (w_0^X)^{-1}.X^v$, where $v = w_0^X (u \cup \ka_d) (z_d \ka_d)^{-1} \in
  W^X$.
\end{lemma}
\begin{proof}
  Using that $\ka_d \in W_{Y_d}$, we obtain $q_d(p_d^{-1}(X^u)) =
  q_d(p_d^{-1}(X^{u \cup \ka_d}))$, and hence $\Gamma_d(1.P_X, X^u) =
  \Gamma_d(1.P_X, X^{u \cup \ka_d})$, so we may replace $u$ with $u \cup \ka_d$
  and assume that $d = \dmax(u)$ (see \cite[\S7.1]{buch.chaput.ea:positivity}).
  We have $q_d({p_d}^{-1}(1.P_X)) = P_X.P_{Y_d}$ and $q_d(p_d^{-1}(X^u)) =
  (Y_d)^{u\ka_d}$ by \Lemma{projection}, and since $\ka_d \leq u_{Y_d} \leq
  w_{0,Y_d}^{Z_d} = \ka_d = \ka_d^{-1}$, we obtain $u \ka_d \in W^{Y_d}$. It
  therefore follows from \Proposition{homschub} that $Y_d(1.P_X, X^u) =
  (P_X.P_{Y_d}) \cap Y_d^{u\ka_d} = (w_0^X)^{-1}.Y_d^v$, where
  $v=w_0^X(u\ka_d)\ka_d^{-1}=w_0^X u \in W^{Y_d}$. The result follows from this
  and \Lemma{projection}, using that $p_d : Z_d(1.P_X, X^u) \to \Gamma_d(1.P_X,
  X^u)$ is birational \cite[Prop.~7.1]{buch.chaput.ea:positivity} and
  $w_{0,X}^{Z_d} = z_d \ka_d$ \cite[Lemma~6.1]{buch.chaput.ea:positivity}.
\end{proof}

\begin{cor}\label{cor:multpt}%
  For $u \in W$ we have $[1.P_X] \star [X^u] = q^{\dmax(u)}\, [X^{w_0^X u}]$ in
  $\QH(X)$ and $[\cO_{1.P_X}] \star \cO^u = q^{\dmax(u)}\, \cO^{w_0^X u}$ in
  $\QK(X)$.
\end{cor}
\begin{proof}
  This follows from \Lemma{gammapt} together with \cite[Prop.~7.1, Thm.~8.3, and
  Thm.~8.16]{buch.chaput.ea:positivity}. Notice that the product $[\cO_{1.P_X}]
  \star \cO^u$ has no exceptional degree by the inequality in
  \cite[Def.~8.2]{buch.chaput.ea:positivity}.
\end{proof}

\targetsec{Wcomin}{%
Let $W^\comin \subset W$ be the subset of point representatives of cominuscule
flag varieties of $G$, together with the identity element:
\[
  W^\comin = \{ w_0^F \mid \text{$F$ is a cominuscule flag variety of $G$} \}
  \cup \{1\} \,.
\]
}
Remarkably, this is a subgroup of $W$, which is also isomorphic to the quotient
of the coweight lattice of $\Phi$ by the coroot lattice. The isomorphism sends
$w_0^F$ to the class of the fundamental coweight corresponding to $F$.

\targetsec{qunits}{%
The classes $q^d [X^w] \in \QH(X)_q$ and $q^d \cO^w \in \QK(X)_q$ given by $w
\in W^\comin$ and $d \in \Z$ are called \emph{Seidel classes}. The cohomological
Seidel classes $q^d [X^w]$ form a subgroup of the group of units
$\QH(X)_q^\times$ by \cite{belkale:transformation, chaput.manivel.ea:affine}. We
will see in \Corollary{seidel} below that the $K$-theoretic Seidel classes
similarly form a subgroup of $\QK(X)_q^\times$.}

The following lemma shows that $[X^u]$ is a Seidel class if and only if the dual
class $[X_u]$ is a Seidel class (when $X$ is cominuscule).

\begin{lemma}\label{lemma:dualpt}%
  Let $X = G/P_X$ and $F = G/P_F$ be flag varieties. The dual element of
  $(w_0^F)^X$ in $W^X$ is $((w_0^F)^{-1} w_0^X)^X$.
\end{lemma}
\begin{proof}
  Using that $w_0 = w_0^F w_{0,F}$, we obtain $(w_0^F)^{-1} w_0 = w_{0,F} =
  (w_{0,F})^{-1} = w_0 w_0^F$, so the dual element of $(w_0^F)^X$ is $(w_0
  w_0^F)^X = ((w_0^F)^{-1} w_0)^X = ((w_0^F)^{-1} w_0^X)^X$.
\end{proof}

The following combinatorial lemma is justified with a case-by-case argument. We
hope to give a type-independent proof in later work.

\begin{lemma}\label{lemma:seidelroot}%
  Let $X$ be a cominuscule flag variety, let $\al \in I(z_1) \ssm \{\ga\}$, and
  define $u \in W^X$ by $I(u) = \{ \al' \in \cP_X \mid \al' \leq (z_1 s_\ga).\al
  \}$. The following are equivalent.\smallskip

  \noin{\rm(a)}\, $u = w^X$ for some $w \in W^\comin$.\smallskip

  \noin{\rm(b)}\, $\delta(\al)$ is a cominuscule simple root.\smallskip

  \noin{\rm(c)}\, $\al \not\leq (z_1 s_\ga).\al$.\smallskip

  \noin{\rm(d)}\, $\cP_X \ssm I(u) = \{ \al' \in \cP_X \mid \al' \geq \al
  \}$.\smallskip

  \noin
  When these conditions hold we have $u^\vee = (w_0^F)^X$, where $F = G/P_F$ is
  the cominuscule flag variety defined by $\delta(\al)$.
\end{lemma}
\begin{proof}
  The action of $w_{0,X}$ restricts to an order-reversing involution of $\cP_X$,
  and $z_1 s_\ga : I(z_1) \ssm \{\ga\} \to w_{0,X}.(I(z_1) \ssm \{\ga\})$ is an
  order isomorphism, see \cite[Lemma~4.4 and
  Prop.~5.10]{buch.chaput.ea:positivity}. This uniquely determines $(z_1
  s_\ga).\al$ for most cominuscule flag varieties. In this proof we will
  identify shapes labeled by simple root numbers with the product of the
  corresponding simple reflections in south-east to north-west order. For
  example, the set $\cP_X$ labeled by simple root numbers, as in
  \Table{tablez1}, is identified with $w_0^X$.

  Assume first that the root system $\Phi$ has type $A_{n-1}$, with simple roots
  $\Delta = \{ \be_1, \dots, \be_{n-1} \}$. All simple roots are cominuscule.
  Let $X = \Gr(m,n)$ be defined by $\ga = \be_m$. Then $\cP_X$ is a rectangle
  with $m$ rows and $n-m$ columns, and $I(z_1) \ssm \{\ga\}$ consists of the top
  row and leftmost column of $\cP_X$, except for the minimal box $\ga$. Let $\al
  \in I(z_1) \ssm \{\ga\}$ be the box in column $c$ of the top row of $\cP_X$.
  Then $(z_1 s_\ga).\al$ is the box in column $c-1$ of the bottom row of
  $\cP_X$, and $I(u)$ is a rectangle with $m$ rows and $c-1$ columns. We also
  have $\delta(\al) = \be_{m+c-1}$, which defines $F = \Gr(m+c-1,n)$. The shape
  of $(w_0^F)^X$ is a rectangle with $m$ rows and $n-m-c+1$ columns; this
  follows because the top part of $I(w_0^F)$ cancels when $w_0^F$ is reduced
  modulo $W_X$. For example, for $X = \Gr(3,8)$ and $c=4$, we obtain $F =
  \Gr(6,8)$ and
  \[
    w_0^X = \tableau{10}{3&4&5&[Aa]6&[a]7\\2&3&4&5&6\\1&2&3&4&5} \ , \
    w_0^F = \tableau{10}{6&7\\5&6\\4&5\\3&4\\2&3\\1&2} \ , \text{ and \ }
    (w_0^F)^X = \tableau{10}{3&4\\2&3\\1&2} = s_2 s_1 s_3 s_2 s_4 s_3 \,.
  \]
  The marked box is $\al$. It follows that $u$ is dual to $(w_0^F)^X$ in $W^X$,
  and conditions (a)-(d) are satisfied. A symmetric argument applies when $\al$
  belongs to the leftmost column of $\cP_X$.

  We next assume that $\Phi$ has type $D_n$, with simple roots $\Delta =
  \{\be_1,\dots,\be_n\}$. The three cominuscule flag varieties of this type are
  $Q = D_n/P_1$, $X' = D_n/P_{n-1}$, and $X'' = D_n/P_n$. Here $Q \cong
  Q^{2n-2}$ is a quadric and $X' \cong X'' \cong \OG(n,2n)$ are maximal
  orthogonal Grassmannians. For $n = 6$, the point representatives are
  \[
    w_0^Q = \tableau{10}{1&2&3&4&[Aa]5\\&&&[Aa]6&[a]4&3&2&1} \ , \
    w_0^{X'} = \tableau{10}{
      5&4&3&2&[Aa]1\\&[Aa]6&[a]4&3&2\\&&5&4&3\\&&&6&4\\&&&&5} \ , \text{ and \ }
    w_0^{X''} = \tableau{10}{
      6&4&3&2&[Aa]1\\&[Aa]5&[a]4&3&2\\&&6&4&3\\&&&5&4\\&&&&6} \ .
  \]

  Let $X = Q$. The elements in $W^Q$ representing Seidel classes other than $1$
  and $[1.P_X]$ are the two elements of length $n-1$. For $n=6$, we obtain
  \[
    (w_0^{X'})^Q =\, \tableau{10}{1&2&3&4&5} \text{ \ \ and \ \ }
    (w_0^{X''})^Q =\, \tableau{10}{1&2&3&4\\&&&6} \ .
  \]
  The set $I(z_1)\ssm\{\ga\}$ contains all boxes of $\cP_Q$, except $\ga$ and
  the maximal box. The two incomparable boxes of $\cP_Q$ are $\al' =
  \theta+\be_{n-1}$ and $\al'' = \theta + \be_n$, where $\theta =
  \be_1+\dots+\be_{n-2}$. Since $z_1 s_\ga$ swaps $\al'$ and $\al''$ and fixes
  all other boxes of $I(z_1)\ssm\{\ga\}$, it follows that (a)-(d) are satisfied
  if and only if $\al \in \{\al',\al''\}$. Assume that $\al = \al''$. We obtain
  $u = s_{n-1} s_{n-2} \cdots s_2 s_1$, $\delta(\al)=\be_n$, and $F=X''$. If $n$
  is even, then the bottom label of $w_0^{X'}$ is $n-1$, hence $u =
  (w_0^{X'})^Q$, and otherwise $u = (w_0^{X''})^Q$. This is consistent with the
  lemma, since the elements $(w_0^{X'})^Q$ and $(w_0^{X''})^Q$ are dual to each
  other when $n$ is even and self-dual when $n$ is odd. A symmetric argument
  applies when $\al = \al'$.

  Let $X = X'$. The shape of $(w_0^Q)^{X'}$ is a single row of $n-1$ boxes, and
  $(w_0^{X''})^{X'}$ is dual to $(w_0^Q)^{X'}$ in $W^{X'}$. For $n=6$, we have
  \[
    (w_0^Q)^{X'} = \tableau{10}{5&4&3&2&1} \text{ \ and \ }
    (w_0^{X''})^{X'} = \tableau{10}{5&4&3&2\\&6&4&3\\&&5&4\\&&&6} \ .
  \]
  The set $I(z_1)\ssm\{\ga\}$ consists of the first two rows of $\cP_{X'}$, with
  $\ga$ removed. Let $\al_1,\al_n \in I(z_1)\ssm\{\ga\}$ be the unique boxes
  with labels $\delta(\al_1)=\be_1$ and $\delta(\al_n)=\be_n$. Then $(z_1
  s_\ga).\al_n = \al_1$, and $(z_1 s_\ga).\al_1$ is the second to last diagonal
  box of $\cP_X$. It follows that conditions (a)-(d) hold if and only if $\al
  \in \{\al_1,\al_n\}$, and the description of $u^\vee$ is accurate.

  If $X$ is a Lagrangian Grassmannian $\LG(n,2n)$, an odd quadric $Q^{2n-1}$, or
  the Freudenthal variety $E_7/P_7$, then no boxes of $I(z_1)\ssm\{\ga\}$
  satisfy conditions (a)--(d). The Cayley plane $E_6/P_6$ is similar to the
  cases of type $D_n$ and left to the reader.
\end{proof}

\begin{lemma}\label{lemma:pos_seidel}%
  Let $X$ be a minuscule flag variety, let $u_1, u_2, \dots, u_\ell \in W^X$,
  and assume that $\cO^{u_1} \star \cO^{u_2} \star \dots \star \cO^{u_\ell} =
  q^d$ for some $d \in \Z$. Then $\cO^{u_i} \star \cB' \subset \cB'$ for each
  $i$, where $\cB' = \{ q^e \cO^v \mid v \in W^X, d \in \Z \}$ is the $\Z$-basis
  of $\QK(X)_q$.
\end{lemma}
\begin{proof}
  It follows from \cite[Thm.~8.4]{buch.chaput.ea:positivity} that $\QK(X)_q$ has
  non-negative structure constants relative to the basis
  \[
    \cB'' = \{ (-1)^{\ell(v)+\int_d c_1(T_X)}\, q^d \cO^v \mid v \in W^X
    \text{ and } d \in \Z \} \,.
  \]
  The lemma therefore follows from the proof of
  \cite[Lemma~3]{buch.wang:positivity}. Namely, if the expansion of $\cO^{u_i}
  \star \cO^v$ contains more than one term, then so does the expansion of
  $\cO^{u_1} \star \dots \star \cO^{u_\ell} \star \cO^v = q^d \cO^v$, which is a
  contradiction.
\end{proof}

\begin{thm}\label{thm:seidelcrit}%
  Let $X$ be a cominuscule flag variety and let $u \in W^X$. The following are
  equivalent.\smallskip

  \noin{\rm(S1)}\, $u = w^X$ for some $w \in W^\comin$.\smallskip

  \noin{\rm(S2)}\, $[X^u] \star \cB \subset \cB$, where $\cB = \{ q^d\, [X^v]
  \mid v \in W^X$, $d \in \Z \}$ is the $\Z$-basis of $\QH(X)_q$.\smallskip

  \noin{\rm(S3)}\, $\cO^u \star \cB' \subset \cB'$, where $\cB' = \{ q^d\, \cO^v
  \mid v \in W^X$, $d \in \Z \}$ is the $\Z$-basis of $\QK(X)_q$.\smallskip

  \noin{\rm(S4)}\, $[X_u] \star [X^u] = [1.P_X] \,\in\, \QH(X)$.\smallskip

  \noin{\rm(S5)}\, $\cO_u \star \cO^u = [\cO_{1.P_X}] \,\in\, \QK(X)$.\smallskip

  \noin{\rm(S6)}\, $\dmax(u^\vee,u) = 0$.\smallskip

  \noin{\rm(S7)}\, We have $u \in \{1,w_0^X\}$, or $\exists$ $\al \in I(z_1)$
  such that $\al \notin I(u)$ and $(z_1 s_\ga).\al \in I(u)$.\smallskip

  \noin
  Furthermore, if $\al$ is as in condition {\rm(S7)}, then $I(u) = \{ \al' \in
  \cP_X \mid \al' \leq (z_1 s_\ga).\al \}$, $\delta(\al)$ is a cominuscule
  simple root, and $u^\vee = (w_0^F)^X$ where $F = G/P_F$ is the cominuscule
  flag variety defined by $\delta(\al)$.
\end{thm}
\begin{proof}
  We may assume $u \notin \{1,w_0^X\}$ by \Corollary{multpt}. The implications
  (S3) $\Rightarrow$ (S2) $\Rightarrow$ (S4) and (S3) $\Rightarrow$ (S5)
  $\Rightarrow$ (S4) are clear, noting that the quantum cohomology product
  $[X^u] \star [X^v]$ is the leading term of $\cO^u \star \cO^v$, and is
  non-zero by \Corollary{multpt} since $[X_u] \star [X^u] \star [X^v]
  \neq 0$. The implication (S4) $\Rightarrow$ (S6) is also clear. Using the
  notation $u_1, u^1 \in W$ defined in
  \cite[Def.~6.5]{buch.chaput.ea:positivity}, it follows from \cite[Prop.~7.1
  and Cor.~7.4]{buch.chaput.ea:positivity} that $\dmax(u^\vee,u)=0$ is
  equivalent to $u_1 \not\leq u^1$, noting that $\dmax(u)>0$ and
  $\dmax(u^\vee)>0$. The elements $u_1$ and $u^1$ are cominuscule minimal
  representatives, so $u_1 \not\leq u^1$ is equivalent to $I(u_1) \not\subset
  I(u^1)$. By \cite[Prop.~6.2 and Prop.~6.7(b)]{buch.chaput.ea:positivity} these
  inversion sets are given by
  \[
    I(u_1) = z_1^{-1}.(I(u) \cap (I(s_\ga^\vee) \ssm I(z_1^\vee)))
    \text{ \ \ and \ \ }
    I(u^1) = s_\ga.(I(u) \cap (I(z_1) \ssm \{\ga\})) \,.
  \]
  Since $(z_1 s_\ga)^{-1}.(I(s_\ga^\vee) \ssm I(z_1^\vee)) = I(z_1) \ssm
  \{\ga\}$ and $\ga \in I(u)$, we deduce that $I(u_1) \not\subset I(u^1)$ holds
  if and only if $(z_1 s_\ga)^{-1}.I(u) \cap I(z_1) \not\subset I(u)$. This
  proves that (S6) is equivalent to (S7). Assume (S7), and let $\al \in I(z_1)$
  satisfy $\al \notin I(u)$ and $(z_1 s_\ga).\al \in I(u)$. Then $\al \not\leq
  (z_1 s_\ga).\al$, so \Lemma{seidelroot} implies that $\delta(\al)$ is a
  cominuscule simple root. This is only possible when $X$ is minuscule. Using
  (S6), it follows from \cite[Thm.~8.3]{buch.chaput.ea:positivity} that $\cO_u
  \star \cO^u = [\cO_{1.P_X}]$. By \Corollary{multpt}, this implies that $(\cO_u
  \star \cO^u)^{\star m}$ is a power of $q$ for some positive integer $m$, so it
  follows from \Lemma{pos_seidel} that $\cO^u \star \cB' \subset \cB'$. This
  proves the implication (S7) $\Rightarrow$ (S3). We finally show that (S1) is
  equivalent to (S7). The implication (S7) $\Rightarrow$ (S1) follows
  immediately from \Lemma{seidelroot}. If (S1) holds, then $u^\vee = (w_0^F)^X$,
  where $F = G/P_F$ is the cominuscule flag variety defined by some cominuscule
  simple root $\ga' \in \Delta \ssm \{\ga\}$. Let $\al \in I(z_1)$ be any root
  for which $\delta(\al) = \ga'$, and define $v \in W^X$ by $I(v) = \{ \al' \in
  \cP_X \mid \al' \leq (z_1 s_\ga).\al \}$. Then \Lemma{seidelroot} shows that
  $u = v$, which proves the implication (S1) $\Rightarrow$ (S7). The last claims
  of the theorem also follow from \Lemma{seidelroot}.
\end{proof}

The following result provides the action of the subgroup of Seidel classes in
$\QK(X)_q^\times$ on the basis $\cB'$. The statement was proved for the quantum
cohomology of arbitrary flag varieties in \cite{belkale:transformation,
chaput.manivel.ea:affine}, and it has been generalized to equivariant quantum
$K$-theory in \cite[Thm.~8.4]{buch.chaput.ea:equivariant}.

\begin{cor}\label{cor:seidel}%
  Let $X$ be a cominuscule flag variety, and let $w \in W^\comin$ and $v \in W$.
  Then, $\cO^w \star \cO^v = q^{\dmin(w,v)}\, \cO^{wv}$ holds in $\QK(X)$.
\end{cor}
\begin{proof}
  It follows from \cite{belkale:transformation, chaput.manivel.ea:affine} that
  $[X^w] \star [X^v] = q^{\dmin(w,v)}\, [X^{w v}]$ holds in the quantum
  cohomology ring $\QH(X)$. The result follows from this since $[X^w] \star
  [X^v]$ is the leading term of $\cO^w \star \cO^v$, and $\cO^w \star \cO^v$ is
  a power of $q$ times a single Schubert class by \Theorem{seidelcrit}.
\end{proof}

\begin{example}\label{example:qhquadric}%
  Let $X = Q^{2n-2}$ be the quadric of type $D_n$, let $\pP \in H^{4n-4}(X)$ be
  the point class, and let $\sigma, \tau \in H^{2n-2}(X)$ be the two Schubert
  classes of middle degree. Since $W^\comin$ has order 4 and $\deg(q) =
  \deg(\pP)$, we deduce that the Seidel classes in $H^*(X)$ consist of $1$,
  $\sigma$, $\tau$, and $\pP$. If $n$ is even, then $\sigma \cdot \tau = \pP$
  and $\sigma^2 = \tau^2 = 0$ hold in $H^*(X)$. It follows that $\sigma \star
  \tau = \pP$, $\sigma^2 = \tau^2 = q$, $\sigma \star \pP = q\,\tau$, and $\tau
  \star \pP = q\,\sigma$ hold in $\QH(X)$. Similarly, if $n$ is odd, then
  $\sigma^2 = \tau^2 = \pP$, $\sigma \star \tau = q$, $\sigma \star \pP =
  q\,\tau$, and $\tau \star \pP = q\,\sigma$ hold in $\QH(X)$. Any product of a
  Seidel class with a non-Seidel Schubert class in $\QH(X)$ is the unique
  element in $\cB$ of the correct degree. This determines all products with
  Seidel classes in $\QH(X)$. Products of arbitrary Schubert classes in $\QH(X)$
  and $\QK(X)$ are determined by this together with \Corollary{seidel} and the
  quantum Chevalley formulas \cite{fulton.woodward:quantum,
  buch.chaput.ea:chevalley}.
\end{example}

\begin{example}
  \def\Pa{{\tableau{4}{{}}}}
  \def\Pb{{\tableau{4}{{}&{}}}}
  \def\Pba{{\tableau{4}{{}&{}\\{}}}}
  Let $X = \Gr(2,4)$. Then
  \[
    [X^\Pb] \star [X^\Pba] \ = \ q\,[X^\Pa]
  \]
  holds in $\QH(X)$. Let $\{e_1,e_2,e_3,e_4\}$ be the standard basis of $\C^4$.
  We claim that
  \[
    \Gamma_1(X_\Pa, X^\Pb)
    \ = \ \{ V \in X \mid V \cap \langle e_1,e_4 \rangle \neq 0 \} \,,
  \]
  that is, $\Gamma_1(X_\Pa, X^\Pb)$ is a translate of the Schubert divisor
  $X^\Pa$. The curve neighborhood $\Gamma_1(X_\Pa, X^\Pb)$ is the union of all
  lines connecting the Schubert varieties
  \[
    \begin{split}
      X_\Pa \ &= \ \{ A \in X \mid \langle e_1 \rangle \subset A
               \subset \langle e_1,e_2,e_3 \rangle \} \ \ \text{and} \\
      X^\Pb \ &= \ \{ B \in X \mid \langle e_4 \rangle \subset B \} \,.
    \end{split}
  \]
  Given $V \in \Gamma_1(X_\Pa, X^\Pb)$, we can find $A \in X_\Pa$ and $B \in
  X^\Pb$ such that
  \[
    0 \, \neq \, A \, \cap \, B \, \subset \, V \,
    \subset \, A+B \, \neq \, \C^4 \,.
  \]
  Since $V$ and $\langle e_1,e_4 \rangle$ are both contained in $A+B$, we obtain
  $V \cap \langle e_1,e_4 \rangle \neq 0$. This proves the claim, since
  $\Gamma_1(X_\Pa, X^\Pb)$ is a divisor in $X$.

  Set $Y_1 = \Fl(1,3;4)$, $Z_1 = \Fl(4)$, and let $p_1 : Z_1 \to X$ and $q_1 :
  Z_1 \to Y_1$ be the projections. We have $q_1 p_1^{-1}(X_\Pa) = (Y_1)_{3142} =
  w_0.Y_1^{2143}$ and $q_1 p_1^{-1}(X^\Pb) = Y_1^{1243}$, so it follows from
  Monk's formula that
  \[
    [Y_1(X_\Pa, X^\Pb)] \, = \, [Y_1^{2143}] \cdot [Y_1^{1243}] \, = \,
    [Y_1^{3142}] + [Y_1^{2341}] \,.
  \]
  We deduce that $Y_1(X_\Pa, X^\Pb)$ is not a Schubert variety in $Y_1$.
\end{example}

\begin{remark}\label{remark:seidel_qh}%
  Let $M = G/P_M$ be any flag variety of $G$. Recall that $H_2(M,\Z)$ can be
  identified with the coroot lattice of $G$ modulo the coroot lattice of $P_M$,
  by identifying each curve class $[M_{s_\be}]$ with the simple coroot
  $\be^\vee$ (see e.g.\ \cite[\S2]{buch.mihalcea:curve}). Let $u \in W$, $w \in
  W^\comin$, and let $\be \in \Delta \cap I(w)$ be the cominuscule simple root
  defining the cominuscule flag variety corresponding to $w$. Set $d =
  \om_\be^\vee - u^{-1}.\om_\be^\vee \in H_2(M,\Z)$, where $\om^\vee_\be$ is the
  fundamental coweight dual to $\be$. It was proved in
  \cite{belkale:transformation, chaput.manivel.ea:affine} that the identity
  \[
    [M^w] \star [M^u] = q^d [M^{w u}]
  \]
  holds in the small quantum cohomology ring $\QH(M)$. This is consistent with
  the following conjecture.
\end{remark}

\begin{conj}\label{conj:seidel-nbhd}%
  Let $M = G/P_M$ be any flag variety. For $u \in W$, $w \in W^\comin$, $I(w)
  \cap \Delta = \{\be\}$, and $d = \om_\be^\vee - u^{-1}.\om_\be^\vee \in
  H_2(M,\Z)$, we have
  \[
    \Gamma_d(M_{w_0 w}, M^u) = w^{-1}.M^{w u} \,.
  \]
\end{conj}

This conjecture follows from \Proposition{homschub} when $d=0$, from
\Lemma{gammapt} when $M$ is cominuscule and $w = w_0^M$, and from
\cite[Cor.~4.6]{li.liu.ea:seidel} when $M$ is a Grassmannian of type A and
$[M^w]$ is a special Seidel class. In response to this paper, it was proved in
\cite{tarigradschi:curve*1} that \Conjecture{seidel-nbhd} is true for all flag
varieties of type A, and the general conjecture follows from the special case
where $P_M$ is a maximal parabolic subgroup. More recently, the conjecture has
been proved in \cite{buch.chaput.ea:equivariant}.


\section{Quantum shapes}\label{sec:qposet}%

Let $X = G/P_X$ be a cominuscule flag variety. An infinite partially ordered set
$\wh\cP_X$ extending $\cP_X$ was constructed in
\cite{buch.chaput.ea:positivity}, such that elements of the set $\cB = \{ q^d
[X^u] \mid u \in W^X,\, d \in \Z \}$ correspond to order ideals in $\wh\cP_X$
that we call \emph{quantum shapes}. Isomorphic partially ordered sets were
constructed in \cite{hagiwara:minuscule*2, postnikov:affine,
green:combinatorics*1}. Products of Seidel classes with arbitrary Schubert
classes have simple combinatorial descriptions in terms of quantum shapes, and
our Pieri formulas also have their simplest expressions in terms of these
shapes. In this section we summarize the facts we need. Proofs of our claims and
more details can be found in \cite[\S7]{buch.chaput.ea:positivity}. Some claims
are justified by \Proposition{qposet} proved at the end of this section.

\targetsec{qbruhat}{%
Recall that $\cB$ is a $\Z$-basis of $\QH(X)_q$. Define a partial order on $\cB$
by
\[
  q^e [X^v] \leq q^d [X^u]
  \ \ \Longleftrightarrow \ \
  \Gamma_{d-e}(X_u,X^v) \neq \emptyset \,.
\]
The condition $\Gamma_{d-e}(X_u,X^v) \neq \emptyset$ says that some rational
curve in $X$ of degree at most $d-e$ intersects both $X_u$ and $X^v$.
Equivalently, $q^e [X^v] \leq q^d [X^u]$ holds if and only if $q^d [X^u]$ occurs
with non-zero coefficient in the expansion of $q^e[X^v] \star q^{d'}[X^w]$ in
$\QH(X)_q$, for some $w \in W^X$ and $d' \geq 0$
\cite[\S7.2]{buch.chaput.ea:positivity}. The following was proved in
\cite[Thm.~7.8]{buch.chaput.ea:positivity}.}

\begin{thm}\label{thm:interval}%
  Let $u, v \in W^X$ and $d \in\Z$. The power $q^d$ occurs in $[X^u] \star
  [X^v]$ if and only if $[X^v] \leq q^d [X_u] \leq [\pt] \star [X^v]$.
\end{thm}

\begin{cor}\label{cor:dminmax}%
  Assume that $u, u', v, v' \in W^X$ satisfy $u' \leq u$ and $v' \leq v$. Then
  $\dmin(u',v') \leq \dmin(u,v)$ and $\dmax(u',v') \leq \dmax(u,v)$.
\end{cor}
\begin{proof}
  Set $d = \dmin(u,v)$. Then $[X^{v'}] \leq [X^v] \leq q^d[X_u] \leq
  q^d[X_{u'}]$. Using that $[X^{u'}] \star [X^{v'}] \neq 0$, this shows that
  $\dmin(u',v') \leq d$. Similarly, if we set $d = \dmax(u',v')$, then $q^d[X_u]
  \leq q^d[X_{u'}] \leq [\pt]\star[X^{v'}] \leq [\pt]\star[X^v]$ and
  $[X^u]\star[X^v] \neq 0$ implies that $d \leq \dmax(u,v)$, as required.
\end{proof}

The following special case is useful for showing that a quantum product $[X^u]
\star [X^v]$ has only classical terms.

\begin{cor}\label{cor:qfree}
  Let $u, v \in W^X$. Assume that $u \leq w$ and $v \leq w_0 w$ for some
  $w \in W^\comin$. Then $\dmax(u,v) = 0$.
\end{cor}
\begin{proof}
  This follows from \Corollary{dminmax} and condition (S6) of
  \Theorem{seidelcrit}.
\end{proof}

\targetsec{qposet}{%
The partially ordered set $\cB$ is a distributive lattice by
\cite[Prop.~7.10]{buch.chaput.ea:positivity}. Let $\wh\cP_X \subset \cB$ be the
subset of all join-irreducible elements. These elements will be called
\emph{boxes}. Define a \emph{quantum shape} in $\wh\cP_X$ to be any non-empty
proper lower order ideal $\la \subset \wh\cP_X$. A quantum shape will also be
called a \emph{shape} when it cannot be misunderstood to be a classical shape in
$\cP_X$. A \emph{skew shape} in $\wh\cP_X$ is the difference $\la/\mu := \la
\ssm \mu$ between two shapes $\mu \subset \la \subset \wh\cP_X$. All shapes in
$\wh\cP_X$ are infinite, and all skew shapes in $\wh\cP_X$ are finite. Given
$q^d [X^u] \in \cB$, define
\[
  I(q^d[X^u]) = \{ \wh\al \in \wh\cP_X \mid \wh\al \leq q^d[X^u] \} \,.
\]
Notice that if $q^d[X^u] \in \wh\cP_X$, then $q^d[X^u]$ is the unique maximal
box of $I(q^d[X^u])$. By \cite[Thm.~7.13]{buch.chaput.ea:positivity}, the map
$I$ is an order isomorphism of $\cB$ with the set of all shapes in $\wh\cP_X$,
where shapes are ordered by inclusion. For any shape $\la \subset \wh\cP_X$ we
will write $\cO^\la = q^d \cO^u$, where $q^d [X^u] \in \cB$ is the unique
element with shape $I(q^d [X^u]) = \la$.}

\targetsec{embed}{%
Given $\al \in \cP_X$, define $\xi(\al) \in W^X$ by $I(\xi(\al)) = \{ \al' \in
\cP_X \mid \al' \leq \al \}$. Then the quantum shape $I([X^{\xi(\al)}]) \subset
\wh\cP_X$ contains a unique maximal box $\tau(\al)$ distinct from $1 \in \cB$,
the identity element of $\QH(X)$. The map $\tau : \cP_X \to \wh\cP_X$ is an
order isomorphism of $\cP_X$ onto an interval in $\wh\cP_X$ by
\cite[Thm.~7.13]{buch.chaput.ea:positivity}. We identify $\cP_X$ with the image
$\tau(\cP_X) \subset \wh\cP_X$. Given a classical shape $\la \subset \cP_X$, we
will abuse notation and also use $\la$ to denote the corresponding quantum shape
$I([X^\la]) = \tau(\la) \cup I(1) \subset \wh\cP_X$, see
\Proposition{qposet}(c). Both of these shapes define the same class $\cO^\la \in
\QK(X)$.}

\targetsec{shift}{%
Quantum multiplication by any Seidel class $\sigma = q^d [X^w]$ in $\QH(X)_q$
defines an order automorphism of $\cB$, which restricts to an order automorphism
of $\wh\cP_X$. Since $1 \in \wh\cP_X$ by \Proposition{qposet}(a), it follows
that all Seidel classes belong to $\wh\cP_X$. Given any shape $\la \subset
\wh\cP_X$, we define a new quantum shape by $\sigma \star \la = \{ \sigma \star
\wh\al \mid \wh\al \in \la \}$. We then have
\[
  \cO^{I(\sigma)} \star \cO^\la \,=\, \cO^{\sigma \star \la}
\]
in $\QK(X)_q$, where $\cO^{I(\sigma)} = q^d\cO^w$ is the Seidel class in
$\QK(X)_q$ corresponding to $\sigma$. The action of Seidel classes on $\wh\cP_X$
therefore determines arbitrary products with Seidel classes in $\QH(X)_q$ and
$\QK(X)_q$. For multiplication by powers of $q$, we use the notation $\la[d] =
q^d \star \la = \{ q^d \star \wh\al \mid \wh\al \in \la \}$, so that
$\cO^{\la[d]} = q^d \cO^\la$, see \Example{gr:shift}, \Example{og:shift}, and
\Example{lg:shift}. The shifting operations on shapes in $\cP_X$ (see
\cite[\S6.2]{buch.chaput.ea:positivity}) are then given by $\la(d) = \la[d] \cap
\cP_X$ (when $\la \subset \cP_X$ is identified with the quantum shape $\la \cup
I(1) \subset \wh\cP_X$).}

\begin{example}\label{example:seidel-trans}%
The following figures show the partially ordered set $\wh\cP_X$ for the quadrics
of dimensions 7 and 12, as well as the exceptional cominuscule flag varieties.
Each set has the west-to-east order, where any node is covered by the nodes
immediately northeast, east, or southeast of it. The elements of $\cP_X$ are
colored gray. Seidel classes are represented by lines marking the eastern
borders of their quantum shapes. We use $\pP$ to denote the point class, and
$\sigma$ and $\sigma'$ are used to represent Seidel classes in $H^*(X,\Z)$ that
are not in the subgroup of $\QH(X)_q^\times$ generated by $\pP$ and $q$.
Multiplication by any Seidel class corresponds to the rigid transformation of
$\wh\cP_X$ that moves the border of $1$ to the border of the Seidel class. This
rigid transformation is a horizontal translation, possibly combined with a
reflection in a horizontal line.
\medskip

$Q^7$:

\begin{tikzpicture}[x=.4cm,y=.4cm]
  \def\oxs#1#2#3{\foreach \x in {0,...,#3} {\draw [fill=lightgray] ($ #1 + \x*#2 $) circle (.33);}}
  \def\qxs#1#2#3{\foreach \x in {0,...,#3} {\draw ($ #1 + \x*#2 $) circle (.33);}}
  \qxs{(0,0)}{(1,0)}{4} \qxs{(4.7,.7)}{(0,-1.4)}{1}
  \qxs{(5.4,0)}{(1,0)}{4} \qxs{(10.1,.7)}{(0,0)}{0} \oxs{(10.1,-.7)}{(0,0)}{0}
  \oxs{(10.8,0)}{(1,0)}{4} \oxs{(15.5,.7)}{(0,0)}{0} \qxs{(15.5,-.7)}{(0,0)}{0}
  \qxs{(16.2,0)}{(1,0)}{4} \qxs{(20.9,.7)}{(0,-1.4)}{1}
  \qxs{(21.6,0)}{(1,0)}{4}
  \draw [thick] (3.3,1.4) node[above] {$q^{-1}$} -- ++(2.1,-2.1);
  \draw [thick] (8.7,1.4) node[above] {$q^{-1}\pP$} -- ++(2.1,-2.1);
  \draw [thick] (14.1,1.4) node[above] {$q$} -- ++(2.1,-2.1);
  \draw [thick] (19.5,1.4) node[above] {$q\pP$} -- ++(2.1,-2.1);
  \draw [thick] (3.3,-1.4) node[below] {$q^{-2}\pP$} -- ++(2.1,2.1);
  \draw [thick] (8.7,-1.4) node[below] {$1$} -- ++(2.1,2.1);
  \draw [thick] (14.1,-1.4) node[below] {$\pP$} -- ++(2.1,2.1);
  \draw [thick] (19.5,-1.4) node[below] {$q^2$} -- ++(2.1,2.1);
\end{tikzpicture}
\medskip

$Q^{12}$:

\begin{tikzpicture}[x=.4cm,y=.4cm]
  \def\oxs#1#2#3{\foreach \x in {0,...,#3} {\draw [fill=lightgray] ($ #1 + \x*#2 $) circle (.33);}}
  \def\qxs#1#2#3{\foreach \x in {0,...,#3} {\draw ($ #1 + \x*#2 $) circle (.33);}}
  \qxs{(0,0)}{(1,0)}{3} \qxs{(3.7,.7)}{(0,-1.4)}{1}
  \qxs{(4.4,0)}{(1,0)}{3} \qxs{(8.1,.7)}{(0,0)}{0} \oxs{(8.1,-.7)}{(0,0)}{0}
  \oxs{(8.8,0)}{(1,0)}{3} \oxs{(12.5,.7)}{(0,-1.4)}{1}
  \oxs{(13.2,0)}{(1,0)}{3} \oxs{(16.9,.7)}{(0,0)}{0} \qxs{(16.9,-.7)}{(0,0)}{0}
  \qxs{(17.6,0)}{(1,0)}{3} \qxs{(21.3,.7)}{(0,-1.4)}{1}
  \qxs{(22.0,0)}{(1,0)}{3}
  \draw [thick] (2.3,1.4) node[above] {$q^{-1}\sigma$} -- ++(2.1,-2.1);
  \draw [thick] (6.7,1.4) node[above] {$q^{-1}\pP$} -- ++(2.1,-2.1);
  \draw [thick] (11.1,1.4) node[above] {$\sigma'$} -- ++(2.1,-2.1);
  \draw [thick] (15.5,1.4) node[above] {$q$} -- ++(2.1,-2.1);
  \draw [thick] (19.9,1.4) node[above] {$q\sigma$} -- ++(2.1,-2.1);
  \draw [thick] (2.3,-1.4) node[below] {$q^{-1}\sigma'$} -- ++(2.1,2.1);
  \draw [thick] (6.7,-1.4) node[below] {$1$} -- ++(2.1,2.1);
  \draw [thick] (11.1,-1.4) node[below] {$\sigma$} -- ++(2.1,2.1);
  \draw [thick] (15.5,-1.4) node[below] {$\pP$} -- ++(2.1,2.1);
  \draw [thick] (19.9,-1.4) node[below] {$q\sigma'$} -- ++(2.1,2.1);
\end{tikzpicture}
\medskip

$E_6/P_6$:

\begin{tikzpicture}[x=.4cm,y=.4cm,rotate=45]
  \def\oxs#1#2#3{\foreach \x in {0,...,#3} {\draw [fill=lightgray] ($ #1 + \x*#2 $) circle (.33);}}
  \def\qxs#1#2#3{\foreach \x in {0,...,#3} {\draw ($ #1 + \x*#2 $) circle (.33);}}
  \qxs{(1,1)}{(1,0)}{2}
  \qxs{(0,0)}{(1,0)}{3}
  \qxs{(2,-1)}{(1,0)}{3}
  \qxs{(2,-2)}{(1,0)}{3}
  \qxs{(4,-3)}{(1,0)}{3}
  \oxs{(4,-4)}{(1,0)}{3}
  \oxs{(6,-5)}{(1,0)}{3}
  \oxs{(6,-6)}{(1,0)}{3}
  \oxs{(8,-7)}{(1,0)}{3}
  \qxs{(8,-8)}{(1,0)}{3}
  \qxs{(10,-9)}{(1,0)}{3}
  \qxs{(10,-10)}{(1,0)}{3}
  \qxs{(12,-11)}{(1,0)}{3}
  \qxs{(12,-12)}{(1,0)}{2}
  \draw [thick] (2.5,2.5) node[above] {$q^{-1}$} -- (2.5,-2.5);
  \draw [thick] (4.5,0.5) node[above] {$q^{-1}\sigma$} -- (4.5,-4.5);
  \draw [thick] (6.5,-1.5) node[above] {$q^{-1}\pP$} -- (6.5,-6.5);
  \draw [thick] (8.5,-3.5) node[above] {$q$} -- (8.5,-8.5);
  \draw [thick] (10.5,-5.5) node[above] {$q\sigma$} -- (10.5,-10.5);
  \draw [thick] (12.5,-7.5) node[above] {$q\pP$} -- (12.5,-12.5);
  \draw [thick] (0.5,-1.5) node[below] {$q^{-2}\pP$} -- (5.5,-1.5);
  \draw [thick] (2.5,-3.5) node[below] {$1$} -- (7.5,-3.5);
  \draw [thick] (4.5,-5.5) node[below] {$\sigma$} -- (9.5,-5.5);
  \draw [thick] (6.5,-7.5) node[below] {$\pP$} -- (11.5,-7.5);
  \draw [thick] (8.5,-9.5) node[below] {$q^2$} -- (13.5,-9.5);
\end{tikzpicture}
\medskip

$E_7/P_7$:

\begin{tikzpicture}[x=.4cm,y=.4cm,rotate=45]
  \def\oxs#1#2#3{\foreach \x in {0,...,#3} {\draw [fill=lightgray] ($ #1 + \x*#2 $) circle (.33);}}
  \def\qxs#1#2#3{\foreach \x in {0,...,#3} {\draw ($ #1 + \x*#2 $) circle (.33);}}
  \qxs{(2,1)}{(1,0)}{3}
  \qxs{(1,0)}{(1,0)}{4}
  \qxs{(3,-1)}{(1,0)}{2}
  \qxs{(4,-2)}{(1,0)}{2}
  \qxs{(4,-3)}{(1,0)}{5}
  \oxs{(4,-4)}{(1,0)}{5}
  \oxs{(7,-5)}{(1,0)}{2}
  \oxs{(8,-6)}{(1,0)}{2}
  \oxs{(8,-7)}{(1,0)}{4}
  \oxs{(8,-8)}{(1,0)}{4}
  \oxs{(11,-9)}{(1,0)}{1}
  \oxs{(12,-10)}{(0,-1)}{2}
  \qxs{(13,-7)}{(0,-1)}{2}
  \qxs{(13,-10)}{(1,0)}{1}
  \qxs{(13,-11)}{(1,0)}{4}
  \qxs{(13,-12)}{(1,0)}{4}
  \qxs{(15,-13)}{(1,0)}{2}
  \qxs{(16,-14)}{(1,0)}{2}
  \qxs{(16,-15)}{(1,0)}{4}
  \qxs{(16,-16)}{(1,0)}{3}
  \draw [thick] (4.5,2.5) node[above] {$q^{-2}\pP$} -- (4.5,-4.5);
  \draw [thick] (8.5,-1.5) node[above] {$q^{-1}\pP$} -- (8.5,-8.5);
  \draw [thick] (12.5,-5.5) node[above] {$\pP$} -- (12.5,-12.5);
  \draw [thick] (16.5,-9.5) node[above] {$q\pP$} -- (16.5,-16.5);
  \draw [thick] (2.5,-3.5) node[below] {$1$} -- (9.5,-3.5);
  \draw [thick] (6.5,-7.5) node[below] {$q$} -- (13.5,-7.5);
  \draw [thick] (10.5,-11.5) node[below] {$q^2$} -- (17.5,-11.5);
\end{tikzpicture}
\end{example}

The following results will be used to describe the quantum posets $\wh\cP_X$ of
classical Grassmannians in the next three sections.

\begin{prop}\label{prop:qposet}%
  Let $X = G/P_X$ be a cominuscule flag variety.\smallskip

  \noin{\rm(a)} We have $\wh\cP_X \cap H^*(X) = \{ \tau(\al) \mid \al \in \cP_X
  \ssm I(z_1^\vee) \} \cup \{1\}$.\smallskip

  \noin{\rm(b)} The map $(\wh\cP_X \cap H^*(X)) \times \Z \to \wh\cP_X$ defined
  by $([X^u],d) \mapsto q^d[X^u]$ is bijective.\smallskip

  \noin{\rm(c)} We have $\tau(\cP_X) = I([1.P_X]) \ssm I(1) \subset \wh\cP_X$.
\end{prop}
\begin{proof}
  Parts (a) and (b) follow from \cite[Def.~7.11 and
  Thm.~7.13]{buch.chaput.ea:positivity}, noting that $\tau(\al) =
  [X^{\xi(\al)}]$ holds if and only if $\al \in \cP_X \ssm I(z_1^\vee)$. Let
  $\al \in \cP_X$. Then $\tau(\al) \leq \tau(\rho) = [1.P_X]$, where $\rho \in
  \cP_X$ is the highest root. Since $[X^{\xi(\al)}] = \tau(\al) \cup 1$ by
  \cite[Thm.~7.13(a)]{buch.chaput.ea:positivity}, and $[X^{\xi(\al)}] \neq 1$,
  we obtain $\tau(\al) \not\leq 1$. This proves that $\tau(\al) \in I([1.P_X])
  \ssm I(1)$. Given $\wh\al \in I([1.P_X]) \ssm I(1)$, we may write $\wh\al =
  q^{-d}[X^{\xi(\al')}]$ for some $\al' \in \cP_X \ssm I(z_1^\vee)$ and $d \in
  \Z$. The condition $\wh\al \leq [1.P_X]$ implies $d \geq 0$, and $\wh\al
  \not\leq 1$ implies that $\al' \notin I(z_d)$ by
  \cite[Lemma~7.12]{buch.chaput.ea:positivity}. It therefore follows from
  \cite[Prop.~5.9(a) and Cor.~5.11]{buch.chaput.ea:positivity} that $\al = (z_1
  s_\ga)^{-d}.\al' \in \cP_X$, and from
  \cite[Def.~7.11]{buch.chaput.ea:positivity} that $\tau(\al) = \wh\al$.
  This proves part (c).
\end{proof}

\begin{lemma}\label{lemma:qposet-covering}%
  Let $\al$ be any non-minimal box in $\cP_X \ssm I(z_1^\vee)$, and let $\wh\al'
  \lessdot \tau(\al)$ be a covering in $\wh\cP_X$. Then $\wh\al' = \tau(\al')$
  for some $\al' \in \cP_X$, such that $\al' \lessdot \al$ is a covering in
  $\cP_X$.
\end{lemma}
\begin{proof}
  Since $\al$ is not minimal in $\cP_X \ssm I(z_1^\vee)$, it follows from
  \Proposition{qposet}(a) that $\wh\al' \not\leq 1$, hence $\wh\al' =
  \tau(\al')$ for some $\al' \in \tau(\cP_X)$ by \Proposition{qposet}(c).
  \Proposition{qposet}(c) also implies that $\al' \lessdot \al$ is a covering in
  $\cP_X$, as required.
\end{proof}


\section{Pieri formula for Grassmannians of type A}\label{sec:grass}%

\subsection{Quantum shapes}\label{sec:gr:poset}%

\targetsec{grass}{%
Let $X = \Gr(m,n)$ be the Grassmannian of $m$-dimensional vector subspaces of
$\C^n$. The quantum cohomology ring $\QH(X)$ was computed by Witten
\cite{witten:verlinde} and Bertram \cite{bertram:quantum}, and a Pieri formula
for the ordinary $K$-theory ring $K(X)$ was obtained by Lenart
\cite{lenart:combinatorial}. The Grassmannian $X$ is minuscule of type
$A_{n-1}$, and the corresponding partially ordered set $\cP_X$ is a rectangle of
boxes with $m$ rows and $n-m$ columns, endowed with the northwest-to-southeast
order discussed below.}\medskip
\[
  \cP_X \,=\, \
  \tableau{8}{
    {}&{}&{}&{}&{}&{}&{}\\
    {}&{}&{}&{}&{}&{}&{}\\
    {}&{}&{}&{}&{}&{}&{}\\
    {}&{}&{}&{}&{}&{}&{}}
\]
\smallskip

\targetsec{rectpart}{%
Each shape $\la \subset \cP_X$ can be identified with a partition
\[
  \la = (\la_1 \geq \la_2 \geq \dots \geq \la_m \geq 0)
\]
with $\la_1 \leq n-m$, where $\la_i$ is the number of boxes in the $i$-th row of
$\la$. If $\la \subset \cP_X$ consists of a single row of boxes, then $\la$ will
also be identified with the integer $p = |\la|$. The \emph{special Schubert
classes} in $K(X)$ are the classes $\cO^p$ for $1 \leq p \leq n-m$. Another
family of special classes consists of $\cO^{(1)^r}$ for $1 \leq r \leq m$, where
$(b)^a$ denotes a rectangle with $a$ rows and $b$ columns.}

\targetsec{cylinder}{%
Let $\Z^2$ denote a grid of boxes $(i,j)$ that fill the plane, where the row
number $i$ increases from north to south, and the column number $j$ increases
from west to east. We endow $\Z^2$ with the \emph{northwest-to-southeast}
partial order, defined by $(i',j') \leq (i,j)$ if and only if $i' \leq i$ and
$j' \leq j$. The quotient $\Z^2/\Z(m,m-n)$ is ordered by $(i',j') + \Z(m,m-n) \
\leq \ (i,j) + \Z(m,m-n)$ if and only if $(i',j') \leq (i+am, j+am-an)$ for some
$a \in \Z$. The cylinder $\Z^2/\Z(m,m-n)$ was used to study the quantum
cohomology ring $\QH(X)$ in \cite[\S3]{postnikov:affine}. This partially ordered
set was also defined in \cite[\S8]{hagiwara:minuscule*2}.}

\begin{prop}\label{prop:gr:qposet}%
  Let $X = \Gr(m,n)$ and set $\sigma = [X^{n-m}]$ and $\tau =
  [X^{(1)^m}]$.\smallskip

  \noin{\rm(a)} The group of Seidel classes in $\QH(X)_q^\times$ is generated by
  $\sigma$ and $\tau$.\smallskip

  \noin{\rm(b)} We have $\sigma^m = \tau^{n-m} = [1.P_X]$ and
  $\sigma \star \tau = q$ in $\QH(X)$.\smallskip

  \noin{\rm(c)} The map $\phi : \Z^2/\Z(m,m-n) \to \wh\cP_X$ defined by
  $\phi(i,j) = \sigma^i \star \tau^j \star [1.P_X]^{-1}$ is an order
  isomorphism, which identifies $\cP_X$ with the rectangle $[1,m] \times
  [1,n-m]$.\smallskip

  \noin{\rm(d)} The actions of $\sigma$ and $\tau$ on $\wh\cP_X$ are determined
  by $\sigma \star \phi(i,j) = \phi(i+1,j)$ and $\tau \star \phi(i,j) =
  \phi(i,j+1)$.
\end{prop}
\begin{proof}
  Noting that $\sigma = [X^{w_0^F}]$ and $\tau = [X^{w_0^{F'}}]$, where $F =
  \Gr(1,n)$ and $F' = \Gr(n-1,n)$, it follows that $\sigma$ and $\tau$ are
  Seidel classes in $\QH(X)$. Part (b) follows from Bertram's quantum Pieri
  formula \cite{bertram:quantum}, and is also an easy consequence of
  \Corollary{seidel}. These results also show that
  \[
    \sigma^i = [X^{(n-m)^i}]
    \text{ \ \ \ and \ \ \ }
    \tau^j = [X^{(j)^m}]
  \]
  for $1 \leq i \leq m$ and $1 \leq j \leq n-m$. Part (a) follows from this,
  noting that $\sigma$ and $\tau$ generate $n$ distinct Seidel classes in
  $H^*(X)$.

  The map $\phi$ is well defined by part (b), and order-preserving since, if
  $(i',j') \leq (i,j)$, then $\phi(i,j)$ occurs in the expansion of the product
  $\phi(i',j') \star (\sigma^{i-i'} \star \tau^{j-j'})$. The maximal box of
  $(n-m)^i$ is the $i$-th box of the rightmost column of $\cP_X$, and the
  maximal box of $(j)^m$ is the $j$-th box of the bottom row of $\cP_X$. Since
  these maximal boxes include all boxes of $\cP_X \ssm I(z_1^\vee)$, it follows
  from \Proposition{qposet} that $\phi$ is surjective. If $\sigma^i \star \tau^j
  = 1$ in $\QH(X)_q$, then since $\sigma$ has order $n$ and inverse $\tau$ in
  $\QH(X)/(q-1)$, we must have $j = i - an$ for some $a \in \Z$. Since $\sigma^i
  \star \tau^j = 1$ has degree $(n-m)i + m j = 0$ in $\QH(X)_q$, we obtain
  $(i,j) = a(m,m-n)$. This implies that $\phi$ is bijective. To show that $\phi$
  is an order isomorphism, we must show that, if $\wh\al' \lessdot \wh\al$ is a
  covering in $\wh\cP_X$, then $\phi^{-1}(\wh\al') < \phi^{-1}(\wh\al)$. If $X =
  \bP^1$, then this follows because $\Z^2/\Z(1,-1)$ is totally ordered, so
  assume that $n \geq 3$. Using that $\phi$ is surjective and quantum
  multiplication by the Seidel classes $\sigma$ and $\tau$ define order
  automorphisms of $\wh\cP_X$, we may assume that $\wh\al = [1.P_X] =
  \phi(m,n-m)$ is the maximal box of $\cP_X \subset \wh\cP_X$. We then deduce
  from \Lemma{qposet-covering} that $\wh\al' = \phi(m-1,n-m)$ or $\wh\al' =
  \phi(m,n-m-1)$, and in either case we obtain $\phi^{-1}(\wh\al') <
  \phi^{-1}(\wh\al)$. Noting that $\phi(m,n-m) = [1.P_X]$ and $\phi(m,0) =
  \phi(0,n-m) = 1$, it follows from \Proposition{qposet}(c) that $\cP_X$ is
  identified with the rectangle $[1,m] \times [1,n-m]$. This proves part (c).
  Part (d) follows from the definition of $\phi$, which completes the proof.
\end{proof}

It follows from \Proposition{gr:qposet} that any quantum shape $\la \subset
\wh\cP_X$ can be represented by a \emph{border} in $\Z^2$ that is invariant
under translation by the vector $(m,m-n)$; the inverse image of $\la$ in $\Z^2$
consists of the boxes northwest of this border. The shifted shape $\la[d] = q^d
\star \la$ is obtained by moving the border of $\la$ by $d$ diagonal steps in
southeast direction.

\begin{example}\label{example:gr:shift}%
  Let $X = \Gr(2,5)$ and set $\sigma = [X^3]$, $\tau = [X^{(1,1)}]$, and $\pP =
  [1.P_X]$. The first figure shows the rectangle $[0,3] \times [0,4] \subset
  \Z^2$, with each box $(i,j)$ labeled by $\phi(i,j) = \sigma^i \star \tau^j
  \star \pP^{-1}$ (rewritten using the relations $\sigma^2 = \tau^3 = \pP$ and
  $\sigma\tau = q$). The framed $2 \times 3$ rectangle can be identified with
  $\cP_X$. The second figure shows the border of a quantum shape $\la \subset
  \wh\cP_X$ together with the border of the shifted shape $\la[1] = q \star
  \la$.\smallskip

  \begin{center}
    \begin{tikzpicture}[x=7.5mm,y=7.5mm]
      \begin{scope}
      \draw [gray] (0,0) -- (5,0) -- (5,4) -- (0,4) -- cycle;
      \draw [gray] (0,1) -- (5,1) (0,2) -- (5,2) (0,3) -- (5,3);
      \draw [gray] (1,0) -- (1,4) (2,0) -- (2,4) (3,0) -- (3,4) (4,0) -- (4,4);
      \draw [very thick] (1,1) -- (4,1) -- (4,3) -- (1,3) -- cycle;
      \draw (0.5,3.5) node {$\pP^{-1}$};
      \draw (0.5,2.5) node {$\sigma^{-1}$};
      \draw (1.5,3.5) node {$\tau^{-2}$};
      \draw (0.5,1.5) node {$1$} (3.5,3.5) node {$1$};
      \draw (1.5,2.5) node {$\sigma^{\!-\!1}\tau$};
      \draw (2.5,3.5) node {$\tau^{-1}$};
      \draw (0.5,0.5) node {$\sigma$} (3.5,2.5) node {$\sigma$};
      \draw (1.5,1.5) node {$\tau$} (4.5,3.5) node {$\tau$};
      \draw (2.5,2.5) node {$\tau^{\!-\!1}\sigma$};
      \draw (1.5,.5) node {$q$} (4.5,2.5) node {$q$};
      \draw (2.5,1.5) node {$\tau^2$};
      \draw (2.5,0.5) node {$q\tau$};
      \draw (3.5,1.5) node {$\pP$};
      \draw (3.5,0.5) node {$q\tau^2$};
      \draw (4.5,1.5) node {$q\sigma$};
      \draw (4.5,0.5) node {$q\pP$};
      \end{scope}
      \begin{scope}[shift={(8,0)}]
      \draw [gray] (0,0) -- (5,0) -- (5,4) -- (0,4) -- cycle;
      \draw [gray] (0,1) -- (5,1) (0,2) -- (5,2) (0,3) -- (5,3);
      \draw [gray] (1,0) -- (1,4) (2,0) -- (2,4) (3,0) -- (3,4) (4,0) -- (4,4);
      \draw [very thick] (1,1) -- (4,1) -- (4,3) -- (1,3) -- cycle;
      \draw [red,line width=2.2pt] (-.5,0) -- (0,0) -- ++(0,1) -- ++(1,0)
      -- ++(0,1) -- ++(2,0) -- ++(0,1) -- ++(1,0) -- ++(0,1) -- ++(1.5,0);
      \draw [red,line width=2.2pt] (1,-.5) -- (1,0) -- ++(1,0)
      -- ++(0,1) -- ++(2,0) -- ++(0,1) -- ++(1,0) -- ++(0,1) -- ++(.5,0);
      \draw (2.47,2.4) node {$\la$};
      \draw (3.47,1.4) node {$\la[1]$};
      \end{scope}
    \end{tikzpicture}
  \end{center}
\end{example}

\begin{remark}\label{remark:cylshape}%
  Let $X = \Gr(m,n)$. The map from \Proposition{gr:qposet}(c) defines an
  order-preserving bijection $\phi : [1,m] \times \Z \to \wh\cP_X$. A non-empty
  proper lower order ideal $\la \subset [1,m] \times \Z$ can be represented by
  the decreasing sequence $(\la_1 \geq \dots \geq \la_m)$, where $\la_i \in \Z$
  is maximal such that $(i,\la_i) \in \la$. The image $\phi(\la)$ is a shape in
  $\wh\cP_X$ if and only if $\la_1 - \la_m \leq n-m$, and any shape in
  $\wh\cP_X$ has this form. In this case the corresponding basis element
  $q^d[X^\mu]$ is obtained by removing rim-hooks from $\la$, see
  \cite{bertram.ciocan-fontanine.ea:quantum}. Notice also that the bijection
  $\phi : [1,m] \times \Z \to \wh\cP_X$ is an order isomorphism if and only if
  $m=1$. Indeed, we have $\wh\cP_X \cong \Z$ when $X \cong \bP^{n-1}$ is
  projective space, whereas the cylinder $\wh\cP_X = \Z^2/\Z(m,m-n)$ is not
  isomorphic to a convex subset of $\Z^2$ for other Grassmannians.
\end{remark}

\subsection{Pieri formula}\label{sec:gr:pieri}%

\targetsec{binom}{%
Let $\theta \subset \wh\cP_X$ be a skew shape. A \emph{row} of $\theta$ means a
subset of the form $\theta \cap \phi(\{k\} \times \Z)$, where $k \in \Z$ and
$\phi$ is the map defined in \Proposition{gr:qposet}, and a \emph{column} of
$\theta$ is a subset of the form $\theta \cap \phi(\Z \times \{k\})$. The skew
shape $\theta \subset \wh\cP_X$ is called a \emph{horizontal strip} if each
column of $\theta$ contains at most one box. Let $r(\theta)$ denote the
number of non-empty rows in $\theta$. For $p \geq 1$ we define
\[
  \cA(\theta,p) = \begin{cases}
    (-1)^{|\theta|-p} \binom{r(\theta)-1}{|\theta|-p} &
    \text{if $\theta$ is a horizontal strip,} \\
    0 & \text{otherwise.}
  \end{cases}
\]
}

A Pieri formula for products of the form $\cO^p \star \cO^\la$ in $\QK(X)$ was
proved in \cite{buch.mihalcea:quantum}. We proceed to show that this formula is
an easy consequence of \Corollary{seidel}, Lenart's Pieri formula for $K(X)$
\cite{lenart:combinatorial}, and a bound on the $q$-degrees in quantum
$K$-theory products proved in \cite{buch.chaput.ea:positivity}.

\begin{thm}\label{thm:gr:qkpieri}%
  Let $X = \Gr(m,n)$, let $\la \subset \wh\cP_X$ be any quantum shape, and let
  $1 \leq p \leq n-m$. Then
  \[
    \cO^p \star \cO^\la = \sum_\nu \cA(\nu/\la,p)\, \cO^\nu
  \]
  holds in $\QK(X)_q$, where the sum is over all quantum shapes $\nu \subset
  \wh\cP_X$ containing $\la$.
\end{thm}
\begin{proof}
  Set $\tau = \cO^{(1)^m}$ and choose $k \in \Z$ maximal such that $\phi(m,k)
  \in \la$. By \Corollary{seidel} and \Proposition{gr:qposet} we have $\tau^{-k}
  \star \cO^\la = \cO^\mu$, where $\mu \subset \cP_X$ is a classical shape with
  $\mu_m = 0$. \Corollary{qfree} then implies that $\dmax(p,\mu) = 0$, so
  \cite[Thm.~8.3]{buch.chaput.ea:positivity} shows that $\cO^p \star \cO^\mu$
  agrees with the classical product $\cO^p \cdot \cO^\mu$ in $K(X)$. Notice
  that, if $\nu \supset \mu$ is any quantum shape such that $\nu/\mu$ is a
  horizontal strip, then $\nu$ is a classical shape. It therefore follows from
  \cite[Thm.~3.2]{lenart:combinatorial} that
  \[
    \cO^p \star \cO^\mu = \sum_\nu \cA(\nu/\mu,p)\, \cO^\nu
  \]
  holds in $\QK(X)$, where the sum is over all shapes $\nu \subset \wh\cP_X$
  containing $\mu$. Since quantum multiplication by $\tau^k$ defines a module
  automorphism of $\QK(X)$ and defines an order automorphism of $\wh\cP_X$, this
  identity is equivalent to the theorem.
\end{proof}

The following version of \Theorem{gr:qkpieri} is equivalent to the Pieri formula
for $\QK(X)$ proved in \cite{buch.mihalcea:quantum}.

\begin{cor}\label{cor:gr:qkpieri}%
  Let $\la \subset \cP_X$ be any shape and let $1 \leq p \leq n-m$. Then
  \[
    \cO^p \star \cO^\la = \sum_\mu \cA(\mu/\la,p)\, \cO^\mu
    + q \sum_\nu \cA(\nu[1]/\la,p)\, \cO^\nu
  \]
  holds in $\QK(X)$, where the first sum is over all shapes $\mu \subset \cP_X$
  containing $\la$, and the second sum is over all shapes $\nu \subset \cP_X$
  for which $\nu[1]$ contains $\la$.
\end{cor}
\begin{proof}
  This is a direct translation of \Theorem{gr:qkpieri}, using that $\cO^{\nu[1]}
  = q\,\cO^\nu$.
\end{proof}

\begin{example}
  \def\pl{\scriptscriptstyle+}%
  Let $X = \Gr(3,7)$. By \Remark{cylshape} we can represent a quantum shape $\la
  \subset \wh\cP_X$ by a decreasing sequence $\la = (\la_1 \geq \la_2 \geq
  \la_3)$ such that $\la_1-\la_3 \leq 4$. When $\la_3 \geq 0$, we will display
  $\la$ as a Young diagram with at most 3 rows. We will also identify $\la$ with
  the class $\cO^\la$ in $\QK(X)$. With this notation we have
  \[
    \cO^3 \, \star \ \tableau{7}{{}&{}&{}\\{}&{}&{}\\{}} \ \ = \ \
    \tableau{7}{{}&{}&{}&{\pl}\\{}&{}&{}\\{}&{\pl}&{\pl}}
    \ + \ \tableau{7}{{}&{}&{}&{\pl}&{\pl}\\{}&{}&{}\\{}&{\pl}}
    \ - \ \tableau{7}{{}&{}&{}&{\pl}&{\pl}\\{}&{}&{}\\{}&{\pl}&{\pl}} \ ,
  \]
  where added boxes are indicated by pluses. This is equivalent to
  \[
    \cO^3 \, \star \ \tableau{7}{{}&{}&{}\\{}&{}&{}\\{}} \ \ = \ \
    \tableau{7}{{}&{}&{}&{}\\{}&{}&{}\\{}&{}&{}}
    \ + \ q \ \tableau{7}{{}&{}\\{}}
    \ - \ q \ \tableau{7}{{}&{}\\{}&{}} \ .
  \]
  Notice that the shape
  \[
    \tableau{7}{{}&{}&{}&{\pl}&{\pl}&{\pl}\\{}&{}&{}\\{}&{\pl}}
  \]
  is not included, as the box added to the third row is in the same column of
  $\wh\cP_X$ as the rightmost box added to the first row.
\end{example}


\section{Pieri formula for maximal orthogonal Grassmannians}\label{sec:maxog}%

\subsection{Quantum shapes}\label{sec:og:poset}%

\targetsec{og}{%
Let $X = \OG(n,2n)$ be the maximal orthogonal Grassmannian, parametrizing one
component of the maximal isotropic subspaces of $\C^{2n}$ endowed with an
orthogonal bilinear form. The quantum cohomology ring $\QH(X)$ was computed in
\cite{kresch.tamvakis:quantum*1}, and a Pieri formula for the ordinary
$K$-theory ring $K(X)$ was obtained in \cite{buch.ravikumar:pieri}.}

The orthogonal Grassmannian $X$ is minuscule of type $D_n$. We identify the
simple roots of type $D_n$ with the vectors
\[
  \Delta = \{e_n-e_{n-1}, \dots, e_3-e_2, e_2-e_1, e_2+e_1 \} \,,
\]
where $\ga = e_1+e_2$ is the cominuscule simple root defining $X$. We then
obtain
\[
  \cP_X = \{ e_i+e_j \mid 1 \leq i < j \leq n \} \,,
\]
where the partial order is given by $e_{i'}+e_{j'} \leq e_i+e_j$ if and only if
$i'\leq i$ and $j' \leq j$. We represent $\cP_X$ as a staircase shape with $n-1$
rows, where $e_i+e_j$ is represented by the box in row $i$ and column $j$:
\[
  \cP_{\OG(6,12)} \, = \
  \tableau{8}{
    {}&{}&{}&{}&{}\\
      &{}&{}&{}&{}\\
      &  &{}&{}&{}\\
      &  &  &{}&{}\\
      &  &  &  &{}}
\]
Each shape $\la \subset \cP_X$ can be identified with a strict partition
\[
  \la = (\la_1 > \la_2 > \dots > \la_\ell > 0)
\]
with $\la_1 \leq n-1$, where $\la_i$ is the number of boxes in the $i$-th row of
$\la$. If $\la \subset \cP_X$ consists of a single row of boxes, then $\la$ will
also be identified with the integer $p = |\la|$. The special Schubert classes in
$K(X)$ are the classes $\cO^p$ for $1 \leq p \leq n-1$.

\targetsec{og:poset}{%
Define the set
\[
  \wb\cP_X = \{(i,j) \in \Z^2 \mid i < j < i+n \} \,,
\]
and give $\wb\cP_X$ the northwest-to-southeast order $(i',j') \leq (i,j)$ if and
only if $i' \leq i$ and $j' \leq j$. We represent $\wb\cP_X$ as an infinite set
of boxes $(i,j)$ in the plane, where the row number $i$ increases from north to
south, and the column number $j$ increases from west to east. Each row in
$\wb\cP_X$ contains $n-1$ boxes. The set $\cP_X$ will be identified with the
subset $\{(i,j) \in \Z^2 \mid 1 \leq i < j \leq n \} \subset \wb\cP_X$.\medskip}

\begin{center}
  \begin{tikzpicture}[x=.3cm,y=.3cm]
    \def\yy{-- ++(0,-1) -- ++(1,0)}
    \def\zz{-- ++(0,1) -- ++(-1,0)}
    \draw (0,5) -- (5,5) \yy\yy\yy\yy\yy\yy\yy\yy -- ++(0,-1)
    -- ++(-5,0) \zz\zz\zz\zz\zz\zz\zz\zz -- cycle;
    \draw[very thick] (2,3) -- (7,3) -- (7,-2) -- ++(-1,0)
    \zz\zz\zz\zz -- cycle;
    \node at (-3.5,0) {$\wb\cP_{\OG(6,12)} \, =$};
  \end{tikzpicture}
\end{center}
\smallskip

Recall the map $\tau : \cP_X \to \wh\cP_X$ from \Section{qposet}.

\begin{prop}\label{prop:og:qposet}%
  Let $X = \OG(n,2n)$.\smallskip

  \noin{\rm(a)} The group of Seidel classes in $\QH(X)_q^\times$ is generated by
  $[1.P_X]$ and $[X^{n-1}]$.\smallskip

  \noin{\rm(b)} We have $[X^{n-1}]^2 = q$ and $[1.P_X]^2 = [X^{n-1}]^n$ in
  $\QH(X)$.\smallskip

  \noin{\rm(c)} The map $\phi : \wb\cP_X \to \wh\cP_X$ defined by $\phi(i,j) =
  [X^{n-1}]^{j-n} \star \tau(e_{i+n-j}+e_n)$ is an order isomorphism which
  identifies $\cP_X$ with the set $\{(i,j) \in \Z^2 \mid 1 \leq i < j \leq n
  \}$.\smallskip

  \noin{\rm(d)} The action of Seidel classes on $\wh\cP_X$ is given by
  $[X^{n-1}] \star \phi(i,j) = \phi(i+1,j+1)$ and $[1.P_X] \star \phi(i,j) =
  \phi(j,i+n)$.
\end{prop}
\begin{proof}
  Let $F = Q^{2n-2}$ be the quadric of type $D_n$. Then we have the relation
  $w_0^F = s_1 \cdots s_{n-2}s_{n-1}s_ns_{n-2} \cdots s_1$, hence $(w_0^F)^X =
  s_1 \cdots s_{n-2}s_n$. This shows that $[X^{n-1}] = [X^{w_0^F}]$. Since
  $(w_0^F)^2 = 1$ holds in $W$, it follows from \Corollary{seidel} that
  $[X^{n-1}]^2$ is a power of $q$. Using that $\deg(q)=2n-2$, we obtain
  $[X^{n-1}]^2 = q$. Since $W^\comin$ has order 4, we have $(w_0^X)^2 \in \{ 1,
  w_0^F \}$, so \Corollary{seidel} implies that either $[1.P_X]^2$ or $[X^{n-1}]
  \star [1.P_X]^2$ is a power of $q$. In either case, $[1.P_X]^2$ is a power of
  $[X^{n-1}]$, and since $\dim(X) = \binom{n}{2}$, we must have $[1.P_X]^2 =
  [X^{n-1}]^n$. Parts (a) and (b) follow from these observations.

  For convenience we set $\al_i = e_i+e_n$ for $1 \leq i \leq n-1$ and $\al'_i =
  e_i+e_{n-1}$ for $1 \leq i \leq n-2$, so that $\cP_X \ssm I(z_1^\vee) = \{
  \al'_1, \dots, \al'_{n-2}, \al_1, \dots, \al_{n-1} \}$. Then $I(\tau(\al_i))
  \cap \cP_X$ consists of the top $i$ rows of $\cP_X$, and $I(\tau(\al'_i)) \cap
  \cP_X$ is obtained by removing the rightmost column in this shape. Notice also
  that $\tau(\al_1) = [X^{n-1}]$, $\tau(\al_{n-1}) = [1.P_X]$, and $\phi(i,j) =
  [X^{n-1}]^{j-n} \star \tau(\al_{i+n-j})$. It follows from
  \cite{kresch.tamvakis:quantum*1} or \Corollary{seidel} that $[X^{n-1}] \star
  \tau(\al'_i) = \tau(\al_{i+1})$ holds in $\QH(X)$ for $1 \leq i \leq n-2$.
  \Proposition{qposet} therefore implies that
  \[
    \begin{split}
      \wh\cP_X \cap H^*(X) \
      &= \ \{ 1, \tau(\al'_1), \dots, \tau(\al'_{n-2}),
      \tau(\al_1), \dots, \tau(\al_{n-1}) \} \\
      &= \ \{ [X^{n-1}]^\epsilon \star \tau(\al_i) \mid 1 \leq i \leq n-1
      \text{ and } \epsilon \in \{0,-1\} \}
    \end{split}
  \]
  and that $\phi$ is bijective. Since $\al_i < \al_{i+1}$ holds in $\cP_X$ and
  $[X^{n-1}]$ is a Seidel class, we obtain $\phi(i,j) < \phi(i+1,j)$ for $i+1 <
  j < i+n$. For $i < j < i+n-1$ we have
  \[
    \begin{split}
      \phi(i,j) \
      &= \ [X^{n-1}]^{j-n} \star \tau(\al_{i+n-j})
      \ = \ [X^{n-1}]^{j+1-n} \star \tau(\al'_{i+n-j-1}) \\
      &< \ [X^{n-1}]^{j+1-n} \star \tau(\al_{i+n-j-1})
      \ = \ \phi(i,j+1) \,.
    \end{split}
  \]
  This implies that $\phi$ is order-preserving. Assume that $\wh\al' \lessdot
  \wh\al$ is a covering in $\wh\cP_X$. We must show that $\phi^{-1}(\wh\al') <
  \phi^{-1}(\wh\al)$. Since $\phi$ is surjective and quantum multiplication by
  $[X^{n-1}]$ is an order automorphism of $\wh\cP_X$, we may assume that $\wh\al
  = \tau(\al_i)$ for some $i$. \Lemma{qposet-covering} then shows that $\wh\al'
  = \tau(\al')$ for some $\al' \in \cP_X$. We deduce that $\wh\al' =
  \tau(\al'_i)$ or $\wh\al' = \tau(\al_{i-1})$. In either case we obtain
  $\phi^{-1}(\wh\al') < \phi^{-1}(\wh\al)$. This proves that $\phi$ is an order
  isomorphism. Finally, using that $\phi(0,n-1) = 1$ and $\phi(n-1,n) =
  [1.P_X]$, the last claim in part (c) follows from \Proposition{qposet}(c).

  The identity $[X^{n-1}] \star \phi(i,j) = \phi(i+1,j+1)$ follows from the
  definition of $\phi$. Quantum multiplication by $[1.P_X]$ corresponds to an
  order automorphism of $\wb\cP_X$ that commutes with multiplication by
  $[X^{n-1}]$, and any such order automorphism of $\wb\cP_X$ is a translation
  along a northwest-to-southeast line, possibly combined with a reflection in
  such a line. Using that $[1.P_X] \star \phi(0,n-1) = \phi(n-1,n)$, we deduce
  that multiplication by $[1.P_X]$ corresponds to the automorphism $(i,j)
  \mapsto (j,i+n)$ of $\wb\cP_X$, which proves part (d).
\end{proof}

We may identify $\wh\cP_X$ with the set of boxes $\wb\cP_X$. Given a shape $\la
\subset \wh\cP_X$ and $d \in \Z$, it follows from \Proposition{og:qposet} that
the shifted shape $\la[d] = q^d \star \la = [X^{n-1}]^{2d} \star \la$ is
obtained by moving $\la$ by $2d$ diagonal steps in southeast direction.

\begin{remark}
  It is natural to extend the notation $\la[d]$ to half-integer shifts by
  setting $\la[k/2] = [X^{n-1}]^k \star \la$. We then have $(\cO^{n-1})^k \star
  \cO^\la = \cO^{\la[k/2]}$ in $\QK(X)_q$.
\end{remark}

\begin{example}\label{example:og:shift}%
  Let $X = \OG(7,14)$. The following figure shows the border of a quantum shape
  $\la$ in $\wh\cP_X$ together with the border of the shifted shape $\la[1] = q
  \star \la$.
  \begin{center}
  \begin{tikzpicture}[x=.5cm,y=.5cm]
    \def\yy{-- ++(0,-1) -- ++(1,0)}
    \def\zz{-- ++(0,1) -- ++(-1,0)}
    \draw (1,5) -- ++(6,0) \yy\yy\yy\yy\yy\yy\yy -- ++(0,-1)
    -- ++(-6,0) \zz\zz\zz\zz\zz\zz\zz -- cycle;
    \draw[very thick] (2,4) -- ++(6,0) -- ++(0,-6) -- ++(-1,0)
    \zz\zz\zz\zz\zz -- cycle;
    \draw [red,line width=2.2pt] (4,2) -- ++(2,0) -- ++(0,1) -- ++(1,0) -- ++(0,1);
    \draw [red,line width=2.2pt] (6,0) -- ++(2,0) -- ++(0,1) -- ++(1,0) -- ++(0,1);
    \draw (5.0,2.65) node {$\la$};
    \draw (7.0,0.65) node {$\la[1]$};
  \end{tikzpicture}
  \end{center}
\end{example}

\subsection{Pieri formula}\label{sec:og:pieri}%

The Pieri formula for the $K$-theory ring $K(X)$ proved in
\cite{buch.ravikumar:pieri} expresses the structure constants of Pieri products
as signed counts of KOG-tableaux, defined as follows.

\begin{defn}[KOG-tableau, \cite{buch.ravikumar:pieri}]\label{defn:kog}%
  Given a skew shape $\theta \subset \cP_X$, a \emph{KOG-tableau} of shape
  $\theta$ is a labeling of the boxes of $\theta$ with positive integers, such
  that (i) each row of $\theta$ is strictly increasing from left to right; (ii)
  each column of $\theta$ is strictly increasing from top to bottom; and (iii)
  each box of $\theta$ is either smaller than or equal to all boxes south-west
  of it, or it is greater than or equal to all boxes south-west of it. The
  \emph{content} of a KOG-tableau is the set of integers contained in its boxes.
  Let $\cB(\theta,p)$ denote $(-1)^{|\theta|-p}$ times the number of
  KOG-tableaux of shape $\theta$ with content $\{1,2,\dots,p\}$.
\end{defn}

The skew shape $\theta$ is called a \emph{rim} if no box in $\theta$ is strictly
south and strictly east of another box in $\theta$. If $\theta$ is not a rim,
then there are no KOG-tableau of shape $\theta$, hence $\cB(\theta,p) = 0$ for
all $p$.

\begin{thm}\label{thm:og:qkpieri}%
  Let $X = \OG(n,2n)$, let $\la \subset \wh\cP_X$ be any quantum shape, and
  let $1 \leq p \leq n-1$. Then
  \[
    \cO^p \star \cO^\la = \sum_\nu \cB(\nu/\la, p)\, \cO^\nu
  \]
  holds in $\QK(X)_q$, where the sum is over all quantum shapes $\nu \subset
  \wh\cP_X$ containing $\la$.
\end{thm}
\begin{proof}
  Choose $k$ maximal such that $\phi(k,k+n-1) \in \la$. By \Corollary{seidel}
  and \Proposition{og:qposet} we have $(\cO^{n-1})^{-k} \star \cO^\la =
  \cO^\mu$, where $\mu \subset \cP_X$ is a classical shape with $\mu_1 \leq
  n-2$. \Corollary{qfree} then implies that $\dmax(p,\mu) = 0$, so
  \cite[Thm.~8.3]{buch.chaput.ea:positivity} shows that $\cO^p \star \cO^\mu$
  agrees with the classical product $\cO^p \cdot \cO^\mu$ in $K(X)$. Notice
  that, if $\nu \supset \mu$ is any quantum shape such that $\nu/\mu$ is a rim,
  then $\nu$ is a classical shape. It therefore follows from
  \cite[Cor.~4.8]{buch.ravikumar:pieri} that
  \[
    \cO^p \star \cO^\mu = \sum_\nu \cB(\nu/\mu,p)\, \cO^\nu
  \]
  holds in $\QK(X)$, where the sum is over all shapes $\nu \subset \wh\cP_X$
  containing $\mu$. Since quantum multiplication by $(\cO^{n-1})^k$ defines a
  module automorphism of $\QK(X)$ and defines an order automorphism of
  $\wh\cP_X$, this identity is equivalent to the theorem.
\end{proof}

\begin{cor}\label{cor:og:qkpieri}%
  Let $\la \subset \cP_X$ be any shape and let $1 \leq p \leq n-1$. Then
  \[
    \cO^p \star \cO^\la \ = \
    \sum_\mu \cB(\mu/\la)\, \cO^\mu +
    q \sum_\nu \cB(\nu[1]/\la)\, \cO^\nu
  \]
  holds in $\QK(X)$, where the first sum is over all shapes $\mu \subset \cP_X$
  containing $\la$, and the second sum is over all shapes $\nu \subset \cP_X$
  for which $\nu[1]$ contains $\la$.
\end{cor}
\begin{proof}
  This is a direct translation of \Theorem{og:qkpieri}, using that $\cO^{\nu[1]}
  = q\,\cO^\nu$.
\end{proof}

\begin{example}
  Let $X = \OG(5,10)$. Then the following holds in $\QK(X)$.
  \[
    \cO^2 \star \cO^{(4,2)} \, = \,
    2\,\cO^{(4, 3, 1)} - \cO^{(4, 3, 2)} + q - 2\,q\,\cO^1 + q\,\cO^2 \,.
  \]
  The corresponding KOG-tableaux are:\smallskip

  {
    \SMALL
    \def\hsep{\hspace{10pt}}
    $\tableau{7}{{}&{}&{}&{}\\&{}&{}&{1}\\&&{2}}$ \hsep
    $\tableau{7}{{}&{}&{}&{}\\&{}&{}&{2}\\&&{1}}$ \hsep
    $\tableau{7}{{}&{}&{}&{}\\&{}&{}&{1}\\&&{1}&{2}}$ \hsep
    $\tableau{7}{{}&{}&{}&{}\\&{}&{}&{1}&{2}}$ \hsep
    $\tableau{7}{{}&{}&{}&{}\\&{}&{}&{1}&{2}\\&&{1}}$ \hsep
    $\tableau{7}{{}&{}&{}&{}\\&{}&{}&{1}&{2}\\&&{2}}$ \hsep
    $\tableau{7}{{}&{}&{}&{}\\&{}&{}&{1}&{2}\\&&{1}&{2}}$
    }
\end{example}


\section{Pieri formula for Lagrangian Grassmannians}\label{sec:lagrange}%

\subsection{Quantum shapes}\label{sec:lg:poset}%

\targetsec{lg}{%
Let $X = \LG(n,2n)$ be the Lagrangian Grassmannian of maximal isotropic
subspaces of $\C^{2n}$ endowed with a symplectic bilinear form. The quantum
cohomology ring $\QH(X)$ was computed in \cite{kresch.tamvakis:quantum}, and a
Pieri formula for the ordinary $K$-theory ring $K(X)$ was obtained in
\cite{buch.ravikumar:pieri}.}

The Lagrangian Grassmannian $X$ is cominuscule, but not minuscule, of type
$C_n$. We identify the simple roots of type $C_n$ with the vectors
\[
  \Delta = \{ e_n-e_{n-1}, \dots, e_3-e_2, e_2-e_1, 2e_1 \} \,,
\]
where $\ga = 2e_1$ is the cominuscule simple root defining $X$. We then obtain
\[
  \cP_X = \{ e_i+e_j \mid 1 \leq i \leq j \leq n \} \,,
\]
where the partial order is given by $e_{i'}+e_{j'} \leq e_i+e_j$ if and only if
$i' \leq i$ and $j' \leq j$. We represent $\cP_X$ as a staircase shape with $n$
rows, where $e_i+e_j$ corresponds to the box in row $i$ and column $j$:
\[
  \cP_{\LG(6,12)} \, = \
  \tableau{8}{
  {}&{}&{}&{}&{}&{}\\
    &{}&{}&{}&{}&{}\\
    &  &{}&{}&{}&{}\\
    &  &  &{}&{}&{}\\
    &  &  &  &{}&{}\\
    &  &  &  &  &{}}
\]
Each shape $\la \subset \cP_X$ can be identified with a strict partition
\[
  \la = (\la_1 > \la_2 > \dots > \la_\ell > 0)
\]
with $\la_1 \leq n$, where $\la_i$ is the number of boxes in the $i$-th row of
$\la$. If $\la \subset \cP_X$ consists of a single row of boxes, then $\la$ will
also be identified with the integer $p = |\la|$. The special Schubert classes in
$K(X)$ are the classes $\cO^p$ for $1 \leq p \leq n$.

\targetsec{lg:poset}{%
Define the set
\[
  \wb\cP_X = \{ (i,j) \in \Z^2 \mid i \leq j \leq i+n \} \,,
\]
and give $\wb\cP_X$ the northwest-to-southeast order $(i',j') \leq (i,j)$ if and
only if $i' \leq i$ and $j' \leq j$. We represent $\wb\cP_X$ as an infinite set
of boxes $(i,j)$ in the plane, where the row number $i$ increases from north to
south, and the column number $j$ increases from east to west. Each row in
$\wb\cP_X$ contains $n+1$ boxes. The set $\cP_X$ will be identified with the
subset $\{(i,j) \in \Z^2 \mid 1 \leq i \leq j \leq n \} \subset
\wb\cP_X$.}\medskip

\begin{center}
  \begin{tikzpicture}[x=.3cm,y=.3cm]
    \def\yy{-- ++(0,-1) -- ++(1,0)}
    \def\zz{-- ++(0,1) -- ++(-1,0)}
    \draw (0,6) -- ++(7,0) \yy\yy\yy\yy\yy\yy\yy\yy\yy -- ++(0,-1)
    -- ++(-7,0) \zz\zz\zz\zz\zz\zz\zz\zz\zz -- cycle;
    \draw[very thick] (2,4) -- (8,4) -- (8,-2) -- ++(-1,0)
    \zz\zz\zz\zz\zz -- cycle;
    \node at (-3.5,.5) {$\wb\cP_{\LG(6,12)} \, =$};
  \end{tikzpicture}
\end{center}
\smallskip

Recall that $z_d \in W^X$ is defined by $X_{z_d} = \Gamma_d(1.P_X)$ for $d \geq
0$.

\begin{prop}\label{prop:lg:qposet}%
  Let $X = \LG(n,2n)$.\smallskip

  \noin{\rm(a)} The group of Seidel classes in $\QH(X)_q^\times$ is generated
  by $[1.P_X]$ and $q$.\smallskip

  \noin{\rm(b)} We have $[1.P_X]^2 = q^n$ in $\QH(X)$.

  \noin{\rm(c)} The map $\phi : \wb\cP_X \to \wh\cP_X$ defined by $\phi(i,j) =
  q^{j-n}\,[X^{z_{i+n-j}}]$ is an order isomorphism which identifies $\cP_X$
  with the set $\{(i,j) \in \Z^2 \mid 1 \leq i \leq j \leq n \}$.

  \noin{\rm(d)} The action of Seidel classes on $\wh\cP_X$ is determined by $q
  \star \phi(i,j) = \phi(i+1,j+1)$ and $[1.P_X] \star \phi(i,j) = \phi(j,i+n)$.
\end{prop}
\begin{proof}
  Since the root system of type $C_n$ has only one cominuscule root, we have
  $W^\comin = \{ 1, w_0^X \}$. It follows that $[1.P_X]^2$ is a power of $q$ in
  $\QH(X)$, and since $\dim(X) = \binom{n+1}{2}$ and $\deg(q) = n+1$, we must
  have $[1.P_X]^2 = q^n$. Parts (a) and (b) follow from this.

  We have $\cP_X \ssm I(z_1^\vee) = \{ e_1+e_n, \dots, e_{n-1}+e_n, 2e_n \}$.
  Since $e_i+e_n$ is the unique maximal box of $I(z_i)$, it follows from
  \Proposition{qposet} that the map $\phi$ is bijective. Notice that for $a,b
  \in [0,n]$ and $d \in \Z$, $[X^{z_a}] \leq q^d [X^{z_b}]$ holds in $\wh\cP_X$
  if and only if $d \geq 0$ and $\Gamma_d(X_{z_b}) \cap X^{z_a} \neq \emptyset$,
  which is equivalent to $d \geq 0$ and $a \leq b+d$, see
  \cite[Lemma~7.9]{buch.chaput.ea:positivity}. It follows that $\phi(i',j') \leq
  \phi(i,j)$ holds in $\wh\cP_X$ if and only if $(i',j') \leq (i,j)$ holds in
  $\wb\cP_X$. This shows that $\phi$ is an order isomorphism. The last claim in
  part (c) follows from \Proposition{qposet}(c), noting that $\phi(0,n) = 1$ and
  $\phi(n,n) = [1.P_X]$.

  The identity $q \star \phi(i,j) = \phi(i+1,j+1)$ follows from the definition
  of $\phi$. Quantum multiplication by $[1.P_X]$ corresponds to an order
  automorphism of $\wb\cP_X$ that commutes with multiplication by $q$, and any
  such order automorphism of $\wb\cP_X$ is a translation along a
  northwest-to-southeast line, possibly combined with a reflection in such a
  line. Using that $[1.P_X] \star \phi(0,n) = \phi(n,n)$, we deduce the formula
  $[1.P_X] \star \phi(i,j) = \phi(j,i+n)$, proving part (d).
\end{proof}

We may identify $\wh\cP_X$ with the set of boxes $\wb\cP_X$. Given a shape $\la
\subset \wh\cP_X$ and $d \in \Z$, it follows from \Proposition{lg:qposet} that
the shifted shape $\la[d] = q^d \star \la$ is obtained by moving $\la$ by $d$
diagonal steps in southeast direction.

\begin{example}\label{example:lg:shift}%
  Let $X = \LG(6,12)$. The following figure shows the border of a quantum shape
  $\la$ in $\wh\cP_X$ together with the border of the shifted shape $\la[1] = q
  \star \la$.
  \begin{center}
  \begin{tikzpicture}[x=.5cm,y=.5cm]
    \def\yy{-- ++(0,-1) -- ++(1,0)}
    \def\zz{-- ++(0,1) -- ++(-1,0)}
    \draw (1,5) -- ++(7,0) \yy\yy\yy\yy\yy\yy\yy -- ++(0,-1)
    -- ++(-7,0) \zz\zz\zz\zz\zz\zz\zz -- cycle;
    \draw[very thick] (2,4) -- (8,4) -- (8,-2) -- ++(-1,0)
    \zz\zz\zz\zz\zz -- cycle;
    \draw [red,line width=2.2pt] (5,1) -- ++(0,2) -- ++(2,0) -- ++(0,1) -- ++(1,0);
    \draw [red,line width=2.2pt] (6,0) -- ++(0,2) -- ++(2,0) -- ++(0,1) -- ++(1,0);
    \draw (6.0,3.5) node {$\la$};
    \draw (7.0,2.5) node {$\la[1]$};
  \end{tikzpicture}
  \end{center}
\end{example}

\subsection{Pieri formula}\label{sec:lg:pieri}%

The Pieri formula for the $K$-theory ring $K(X)$ proved in
\cite{buch.ravikumar:pieri} expresses the structure constants of Pieri products
as signed counts of KLG-tableaux, defined as follows. A rim is defined as in
\Section{og:pieri}.

\begin{defn}[KLG-tableau, \cite{buch.ravikumar:pieri}]\label{defn:klgtab}%
  Let $\theta \subset \cP_X$ be a rim. A \emph{KLG-tableau} of shape $\theta$ is
  a labeling of the boxes of $\theta$ with elements from the ordered set
  \[
    \{1' < 1 < 2' < 2 < \cdots \}
  \]
  such that (i) each row of $\theta$ is strictly increasing from left to right;
  (ii) each column of $\theta$ is strictly increasing from top to bottom; (iii)
  each box containing an unprimed integer is larger than or equal to all boxes
  southwest of it; (iv) each box containing a primed integer is smaller than or
  equal to all boxes southwest of it; (v) no SW diagonal box contains a primed
  integer. The \emph{content} of a KLG-tableau is the set of integers $i$ such
  that some box contains $i$ or $i'$. Define $\cC(\theta,p)$ to be
  $(-1)^{|\theta|-p}$ times the number of KLG-tableaux of shape $\theta$ with
  content $\{1,2,\dots,p\}$. If $\theta \subset \cP_X$ is a skew shape that is
  not a rim, then set $\cC(\theta,p) = 0$.
\end{defn}

In contrast to the case of maximal orthogonal Grassmannians, we need to adjust
the definition of KLG-tableau with extra conditions in the quantum case.

\begin{defn}[QKLG-tableau]\label{defn:qklgtab}%
  Let $T$ be a KLG-tableau whose shape is a rim contained in $\wh\cP_X$. A box
  of $T$ is called \emph{unrepeated} if its label is distinct from all other
  labels when ignoring primes. A box of $T$ is a {\em quantum box} if it belongs
  to the NE diagonal of $\wh\cP_X$ or is connected to an unrepeated quantum box.
  A box of $T$ is \emph{terminal} if it is not on the SW diagonal of $\wh\cP_X$
  and is not connected to a box to the left or below it. We say that $T$ is a
  \emph{QKLG-tableau} if (vi) every primed non-terminal quantum box is
  unrepeated, and (vii) every terminal quantum box is primed. For any rim
  $\theta$ contained in $\wh\cP_X$, we let $\cN(\theta,p)$ denote
  $(-1)^{|\theta|-p}$ times the number of QKLG-tableaux of shape $\theta$ with
  content $\{1,2,\dots,p\}$. If $\theta \subset \wh\cP_X$ is a skew shape that
  is not a rim, then set $\cN(\theta,p) = 0$.
\end{defn}

The integers $\cN(\theta,p)$ can also be defined recursively, see
\Definition{lgqkpieri}. Notice that $\cN(\theta,p) = \cC(\theta,p)$ holds
whenever $\theta \subset \cP_X$.

\begin{thm}\label{thm:lg:qkpieri}%
  Let $X = \LG(n,2n)$, let $\la \subset \wh\cP_X$ be any quantum shape, and let
  $1 \leq p \leq n$. Then
  \[
    \cO^p \star \cO^\la \,=\, \sum_\nu\,
    \cN(\nu/\la,p)\, \cO^\nu
  \]
  holds in $\QK(X)_q$, where the sum is over all quantum shapes $\nu \subset
  \wh\cP_X$ containing $\la$.
\end{thm}

The proof of \Theorem{lg:qkpieri} is given in the three remaining sections of
this paper.

\begin{cor}\label{cor:lg:qkpieri}%
  Let $\la \subset \cP_X$ be any shape and let $1 \leq p \leq n$. Then
  \[
    \cO^p \star \cO^\la \,=\, \sum_\mu \cC(\mu/\la,p)\, \cO^\mu
    + q \sum_\nu \cN(\nu[1]/\la,p)\, \cO^\nu
  \]
  holds in $\QK(X)$, where the first sum is over all shapes $\mu \subset \cP_X$
  containing $\la$, and the second sum is over all shapes $\nu \subset \cP_X$
  for which $\nu[1]$ contains $\la$.
\end{cor}

\begin{example}
  Let $X = \LG(7,14)$ and set $\la = (7,5,4,2)$ and $\nu = (7,5,3,2)$. Then
  $\nu[1]/\la$ meets both the SW diagonal and the NE diagonal of $\wh\cP_X$. The
  coefficient of $q\,\cO^\nu$ in $\cO^6 \star \cO^\la$ is $-4$, due to the
  following list of QKLG-tableaux of shape $\nu[1]/\la$ with content
  $\{1,2,3,4,5,6\}$:\medskip

  \newcommand{\hspc}{\hspace{8mm}}
  \newcommand{\sspc}{\hspace{5mm}}
  \noindent\hspc%
  $\tableau{10}{&&&[Aa]{1'}\\&&{2'}&{6}\\&&{3'}\\&[a]{4'}\\{4}&{5}}$\sspc%
  $\tableau{10}{&&&[Aa]{1'}\\&&{2'}&{6}\\&&{3'}\\&[a]{3'}\\{4}&{5}}$\sspc%
  $\tableau{10}{&&&[Aa]{1'}\\&&[a]{2'}&[Aa]{6}\\&&[a]{3'}\\&{4'}\\{5}&{6}}$\sspc%
  $\tableau{10}{&&&[Aa]{1'}\\&&[a]{2'}&[Aa]{6}\\&&[a]{6}\\&{3'}\\{4}&{5}}$\sspc%
  \bigskip

  \noindent
  Quantum boxes are indicated with a think border. There are five additional
  KLG-tableaux of shape $\nu[1]/\la$ with content $\{1,2,3,4,5,6\}$ which do not
  satisfy the additional conditions of \Definition{qklgtab}:\medskip

  \noindent\hspc%
  $\tableau{10}{&&&[Aa]{1'}\\&&{2'}&{6}\\&&{5}\\&[a]{3'}\\{3}&{4}}$\sspc%
  $\tableau{10}{&&&[Aa]{1'}\\&&{2'}&{6}\\&&{5}\\&[a]{3'}\\{4}&{5}}$\sspc%
  $\tableau{10}{&&&[Aa]{1'}\\&&{2'}&{6}\\&&[a]{5}\\&{2'}\\{3}&{4}}$\sspc%
  $\tableau{10}{&&&[Aa]{1'}\\&&[a]{1'}&{6}\\&&{2'}\\&{3'}\\{4}&{5}}$\sspc%
  $\tableau{10}{&&&[Aa]{1'}\\&&[a]{1'}&{6}\\&&{5}\\&{2'}\\{3}&{4}}$\sspc%
  \bigskip

  \noindent
  The first two violate condition (vii) and the last three violate condition
  (vi).
\end{example}


\section{Perpendicular incidences of symplectic Richardson varieties}
\label{sec:incidence}%

Let $Y_\sP^\sQ$ be a Richardson variety in the symplectic Grassmannian $Y =
\SG(m,2n)$. Each point $L \in \bP^{2n-1}$ defines the subvariety $Y_\sP^\sQ \cap
L^\perp = \{ V \in Y_\sP^\sQ \mid V \subset L^\perp \}$. Let $\bP_\sP^\sQ \subset
\bP^{2n-1}$ be the subset of points $L$ for which $Y_\sP^\sQ \cap L^\perp$ is not
empty. In this section we show that $\bP_\sP^\sQ$ is a complete intersection defined
by explicitly given equations. We also show that $Y_\sP^\sQ \cap L^\perp$ is
rational for all points $L$ in a dense open subset of $\bP_\sP^\sQ$. This will be
used in \Section{pierigw} to compute the Gromov-Witten invariants required to
prove our Pieri formula for the quantum $K$-theory of Lagrangian Grassmannians.

\subsection{Symplectic Grassmannians}\label{sec:sgrass}%

\targetsec{sg}{%
Let $\{ e_1, \dots, e_{2n} \}$ denote the standard basis of $\C^{2n}$. Define
the symplectic vector space $E = \C^{2n}$, where the skew-symmetric bilinear
form on $E$ is given by $(e_i,e_j) = \delta_{i+j,2n+1}$ for $1 \leq i \leq j
\leq 2n$. Given $0 \leq m \leq n$, let $Y = \SG(m,E) = \SG(m,2n)$ denote the
symplectic Grassmannian of $m$-dimensional isotropic subspaces of $E$,
\[
  Y = \SG(m,E) = \{ V \subset E \mid \dim(V)=m \text{ and } (V,V)=0 \} \,.
\]
This space has a transitive action by the symplectic group $G = \Sp(E)$. Let $B
\subset G$ be the Borel subgroup of upper triangular matrices, let $B^- \subset
G$ be the opposite Borel subgroup of lower triangular matrices, and let $T = B
\cap B^-$ be the maximal torus of symplectic diagonal matrices.}

\targetsec{sg:schub}{%
For $a,b \in \Z$, let $[a,b] = \{ x \in \Z \mid a \leq x \leq b \}$ denote the
corresponding integer interval. Given any subset $\sP \subset [1,2n]$, we let $E_\sP
= \Span_\C \{ e_p \mid p \in \sP \}$ be the span of the basis elements
corresponding to $\sP$. A \emph{Schubert symbol} for $\SG(m,2n)$ is a subset $\sP
\subset [1,2n]$ of cardinality $m$, such that $p'+p'' \neq 2n+1$ for all $p',p''
\in \sP$. The subspace $E_\sP$ is a point of $\SG(m,2n)$ if and only if $\sP$ is a
Schubert symbol, and the $T$-fixed points of $\SG(m,E)$ are exactly the points
$E_\sP$ for which $\sP$ is a Schubert symbol for $Y$. Each Schubert symbol $\sP$
defines the Schubert varieties
\[
  Y_\sP = \ov{B.E_\sP} \text{ \ \ \ \ and \ \ \ \ }
  Y^\sP = \ov{B^-.E_\sP} \ \subset Y \,.
\]
These varieties can also be defined by (see
\cite[\S4.1]{buch.kresch.ea:giambelli})
\[
  \begin{split}
    Y_\sP \, &= \, \{ V \in Y \mid \dim(V \cap E_{[1,b]}) \geq \#(\sP \cap [1,b])
    ~\forall~ b \in [1,2n] \} \text{ \ \ \ and} \\
    Y^\sP \, &= \, \{ V \in Y \mid \dim(V \cap E_{[a,2n]}) \geq \#(\sP \cap [a,2n])
    ~\forall~ a \in [1,2n] \} \,.
  \end{split}
\]
}

\targetsec{sg:bruhat}{%
Given Schubert symbols $\sP$ and $\sQ$ for $Y$, we will denote the elements of these
sets by $\sP = \{ p_1 < p_2 < \dots < p_m \}$ and $\sQ = \{ q_1 < q_2 < \dots < q_m
\}$. The \emph{Bruhat order} on Schubert symbols is defined by $\sQ \leq \sP$ if and
only if $q_i \leq p_i$ for $1 \leq i \leq m$. With this notation we have
\[
  \sQ \leq \sP \ \ \Leftrightarrow \ \
  E_\sQ \in Y_\sP \ \ \Leftrightarrow \ \
  Y_\sQ \subset Y_\sP \ \ \Leftrightarrow \ \
  Y_\sP \cap Y^\sQ \neq \emptyset \,.
\]
}

\targetsec{sg:length}{%
Define the length $\ell(\sP)$ to be
\[
  \ell(\sP) \ = \
  \sum_{i=1}^m (p_i - i) \,-\, \# \{ i < j : p_i + p_j > 2n+1 \} \,.
\]
We then have $\dim(Y_\sP) = \codim(Y^\sP,Y) = \ell(\sP)$. Notice also that $Y^\sP$ is a
translate of $Y_{\sP^\vee}$, where $\sP^\vee = \{ 2n+1-p \mid p \in \sP \}$ is the
Poincare dual Schubert symbol.}

\subsection{Richardson varieties}\label{sec:sg:rich}%

\targetsec{sgrich}{%
Two Schubert symbols $\sP$ and $\sQ$ for $Y = \SG(m,2n)$ such that $\sQ \leq \sP$ define
the Richardson variety
\[
  Y_\sP^\sQ = Y_\sP \cap Y^\sQ \,.
\]
This variety is known to be rational \cite{richardson:intersections}. Using that
$\dim(Y_\sP^\sQ) = \ell(\sP) - \ell(\sQ)$, we obtain
\begin{equation}\label{eqn:richdim}%
  \dim(Y_\sP^\sQ) \ = \
  \sum_{i=1}^m (p_i - q_i) \,-\, \# \{ i<j : q_i+q_j < 2n+1 < p_i+p_j \} \,.
\end{equation}
For any point $V \in Y_\sP^\sQ$ we have $V \subset E_{[q_1,p_m]}$ and $V \cap
E_{[q_i,p_i]} \neq 0$ for $1 \leq i \leq m$; this holds because $\dim(V \cap
E_{[1,p_i]}) \geq i$, $\dim(V \cap E_{[q_i,2n]}) \geq m+1-i$, and $\dim(V) = m$.}

\targetsec{sflag}{%
Let $Y = \SG(m,E)$ and $Y' = \SG(m-1,E)$, and define the 2-step symplectic
flag variety
\[
  Z = \SF(m-1,m;E) =
  \{ (V',V) \in Y' \times Y \mid V' \subset V \} \,.
\]
Let $a : Z \to Y$ and $b : Z \to Y'$ be the projections. The $T$-fixed points in
$Z$ have the form $(E_{\sP'},E_\sP)$, where $\sP'$ and $\sP$ are Schubert symbols for
$Y'$ and $Y$, respectively, such that $\sP' \subset \sP$. The corresponding Schubert
varieties are denoted
\[
  Z_{\sP',\sP} = \ov{B.(E_{\sP'},E_\sP)}
  \text{ \ \ \ and \ \ \ }
  Z^{\sP',\sP} = \ov{B^-.(E_{\sP'},E_\sP)} \,.
\]
A Richardson variety in $Z$ is denoted by $Z_{\sP',\sP}^{\sQ',\sQ} = Z_{\sP',\sP} \cap
Z^{\sQ',\sQ}$. Recall our standing notation $\sP = \{ p_1 < \dots < p_m \}$ and
$\sQ = \{ q_1 < \dots < q_m \}$ for Schubert symbols for $\SG(m,2n)$.}

\begin{prop}\label{prop:ontorich}%
  Let $\sQ \leq \sP$ be Schubert symbols for $Y = \SG(m,2n)$, and let $1 \leq k \leq
  m$. Set $\sQ' = \sQ \ssm \{ q_k \}$ and $\sP' = \sP \ssm \{ p_k \}$. Then the
  restricted map $a : Z_{\sP',\sP}^{\sQ',\sQ} \to Y_\sP^\sQ$ is birational. In addition, the
  restricted map $b : Z_{\sP',\sP}^{\sQ',\sQ} \to {Y'}_{\sP'}^{\sQ'}$ is surjective if and
  only if $ \dim({Y'}_{\sP'}^{\sQ'}) \leq \dim(Y_\sP^\sQ)$.
\end{prop}

\targetsec{sp:weyl}{%
We will prove \Proposition{ontorich} after introducing some additional notation
and results. We will identify the Weyl group of $\Sp(2n)$ with the group of
permutations
\[
  W = \{ w \in S_{2n} \mid w(i) + w(2n+1-i) = 2n+1
  \text{ for } i \in [1,2n] \} \,.
\]
This group is generated by the simple reflections $s_1, \dots, s_n \in W$
determined by
\[
  s_i(i) = i+1 \,, \ \ \  s_i(i+1) = i \,, \ \ \ \text{and} \ \ \
  s_i(j) = j \ \text{ for } j \in [1,n] \ssm \{i,i+1\} \,.
\]
}
The simple reflection $s_n$ corresponds to the unique long simple root of the
root system of type $C_n$. The parabolic subgroup $P_Y \subset G$ defining $Y =
\SG(m,2n)$ corresponds to the subgroup $W_Y \subset W$ generated by $s_i$ for $i
\neq m$. Let $W^Y \subset W$ be the subset of minimal representatives of the
cosets in $W/W_Y$. Then $W^Y$ is in bijective correspondence with the Schubert
symbols of $Y$. The Schubert symbol $\sP = \{ p_1 < p_2 < \dots < p_m \}$
corresponds to the permutation $w \in W^Y$ defined by
\[
  w(j) = p_j \ \text{ for } \ 1 \leq j \leq m \,, \ \ \ \text{and} \ \ \
  w(m+1) < w(m+2) < \dots < w(n) \leq n \,.
\]
This correspondence preserves the Bruhat order.

\targetsec{w-hat}{%
The permutation $\wh w \in W^Z$ corresponding of a $T$-fixed point
$(E_{\sP'},E_\sP)$ of $Z = \SF(m-1,m;2n)$, with $\sP' = \sP \ssm \{ p_k \}$, is defined
by
\[
  \wh w(j) = \begin{cases}
    p_j & \text{if $1 \leq j < k$,} \\
    p_{j+1} & \text{if $k \leq j < m$,} \\
    p_k & \text{if $j = m$,}
  \end{cases}
\]
and $\wh w(m+1) < \wh w(m+2) < \dots < \wh w(n) \leq n$. Equivalently, if $w \in
W^Y$ corresponds to $\sP$, then
\[
  \wh w = w s_k s_{k+1} \dots s_{m-1} \,.
\]
}

\targetsec{w-prime}{%
Let $w' \in W^{Y'}$ be the permutation corresponding to $\sP'$. Then $w'$ is
obtained from $\wh w$ by first replacing the value $\wh w(m)$ with $\min(p_k,
2n+1-p_k)$, and then rearranging the values $\wh w(m), \dots, \wh w(n)$ in
increasing order. Since $\wh w(m+1) < \dots < \wh w(n) \leq n$, we can write $w'
= \wh w y$, where $y$ is the product of the first $\ell(\wh w) - \ell(w')$
simple reflections in the product
\begin{equation}\label{eqn:simprefprod}%
  s_m s_{m+1} \cdots s_{n-1} s_n s_{n-1} \cdots s_{m+1} s_m \,.
\end{equation}
}

\targetsec{projrich}{%
Let $F = \Sp(2n)/B$ be the variety of complete symplectic flags, and let $M =
\Sp(2n)/P_M$ be any flag variety of $G = \Sp(2n)$. For $\tau \leq \sigma$ in
$W$, let $\Pi_\sigma^\tau(M) \subset M$ denote the \emph{projected Richardson
variety} obtained as the image of $F_\sigma^\tau$ under the projection $F \to
M$. Recall from \cite[\S2.1]{buch.chaput.ea:positivity} that the $M$-Bruhat order
$\leq_M$ on $W$ can be defined by
\[
  \tau \leq_M \sigma
  \ \ \ \Longleftrightarrow \ \ \
  \tau \leq \sigma \ \ \text{and} \ \ \sigma_M \leq_L \tau_M \,,
\]
where $\sigma = \sigma^M \sigma_M$ and $\tau = \tau^M \tau_M$ are the parabolic
factorizations with respect to $W_M$, and $\leq_L$ is the left weak Bruhat order
on $W$. We need the following properties of projected Richardson varieties from
\cite{knutson.lam.ea:projections} (see also
\cite[\S3]{buch.chaput.ea:positivity}).}

\begin{prop}[\cite{knutson.lam.ea:projections}]\label{prop:projrich}%
  Let $\tau \leq \sigma$ in $W$. The projected Richardson variety
  $\Pi_\sigma^\tau(M)$ satisfies the following properties.\smallskip

  \noin{\rm(a)} We have $\Pi_\sigma^\tau(M) \subset M_\sigma^\tau$.\smallskip

  \noin{\rm(b)} If $\sigma \in W^M$, then equality $\Pi_\sigma^\tau(M) =
  M_\sigma^\tau$ holds if and only if $\tau \in W^M$.\smallskip

  \noin{\rm(c)} The projection $F_\sigma^\tau \to \Pi_\sigma^\tau(M)$ is
  birational if and only if $\tau \leq_M \sigma$.\smallskip

  \noin{\rm(d)} For any simple reflection $s_i \in W_M$ with $\sigma s_i <
  \sigma$, we have $\Pi_\sigma^\tau(M) = \Pi_{\sigma s_i}^{\min(\tau,\tau
  s_i)}$.
\end{prop}

\targetsec{minbruhat}{%
Here $\min(\tau, \tau s_i)$ denotes the smaller element among $\tau$ and $\tau
s_i$ in the Bruhat order on $W$.}

\begin{proof}[Proof of \Proposition{ontorich}]
  Let $u \in W^Y$ correspond to $\sP$ and let $v \in W^Y$ correspond to $\sQ$. Then
  $Y_\sP^\sQ = Y_u^v$ and $Z_{\sP',\sP}^{\sQ',\sQ} = Z_{\wh u}^{\wh v}$, where $\wh u = u x$
  and $\wh v = v x$, with $x = s_k s_{k+1} \cdots s_{m-1}$. Since $\wh u, \wh v
  \in W^Z$, we have $Z_{\wh u}^{\wh v} = \Pi_{\wh u}^{\wh v}(Z)$ by
  \Proposition{projrich}(b). Using that $\wh u = u x$ and $\wh v = v x$ are
  parabolic factorizations with respect to $W_Y$, we obtain $\wh v \leq_Y \wh
  u$, so \Proposition{projrich}(d,b,c) shows that $\Pi_{\wh u}^{\wh v}(Y) =
  \Pi_u^v(Y) = Y_u^v$ and $a : Z_{\wh u}^{\wh v} \to Y_u^v$ is birational. This
  proves the first claim.

  Since $Z_{\wh u}^{\wh v} = \Pi_{\wh u}^{\wh v}(Z)$, we have
  $b(Z_{\sP',\sP}^{\sQ',\sQ}) = \Pi_{\wh u}^{\wh v}(Y')$. Let $u', v' \in W^{Y'}$ be the
  elements corresponding to $\sP'$ and $\sQ'$. Then $u' = \wh u y$ and $v' = \wh v
  z$, where $y$ is the product of the first $\ell(\wh u) - \ell(u')$ simple
  reflections in \eqn{simprefprod}, and $z$ is the product of the first
  $\ell(\wh v) - \ell(v')$ simple reflections. Using \Proposition{projrich}(d),
  we obtain
  \[
    \Pi_{\wh u}^{\wh v}(Y') = \Pi_{\wh u y}^{\wh v \min(y,z)}(Y') \,.
  \]
  By \Proposition{projrich}(b), this variety is equal to ${Y'}_{u'}^{v'}$ if and
  only if $z \leq y$, which is equivalent to $\ell(u') - \ell(v') \leq \ell(\wh
  u) - \ell(\wh v)$. The second claim follows from this.
\end{proof}

\subsection{Matrix representations of Richardson varieties}\label{sec:matspace}%

\targetsec{MPQ}{%
We need a parametrization of an open subset of $Y_\sP^\sQ$ by matrices with
perpendicular rows, which is based on a combinatorial diagram used in
\cite{buch.kresch.ea:quantum, ravikumar:triple*1}. Let $M_\sP^\sQ$ be the variety of
all $m \times (2n)$-matrices $A = (a_{i,j})$, with $a_{i,j} \in \C$, such that
for $1 \leq i \leq m$ we have $a_{i,q_i} = 1$, $a_{i,p_i} \neq 0$, and $a_{i,j}
= 0$ for $j \notin [q_i,p_i]$, and such that each pair of rows of $A$ are
perpendicular as vectors in $E$, that is,
\begin{equation}\label{eqn:perp}%
  \sum_{t=1}^n (a_{i,t}\, a_{j,2n+1-t} - a_{i,2n+1-t}\, a_{j,t}) = 0
\end{equation}
for $1 \leq i < j \leq m$. Notice that this equation is vacuous unless
\[
  q_i+q_j < 2n+1 < p_i+p_j \,.
\]
}
We will say that rows $i$ and $j$ in $M_\sP^\sQ$ are \emph{correlated} if $i \neq j$
and this inequality holds. We will show in \Theorem{param} that $M_\sP^\sQ$ is
isomorphic to a dense open subset of the Richardson variety $Y_\sP^\sQ$. In
particular, $M_\sP^\sQ$ is non-empty and irreducible. Identity \eqn{richdim} states
that $\dim(Y_\sP^\sQ)$ is equal to the number of entries $a_{i,j}$ that are not
explicitly assigned to a constant value, minus the number of pairs of correlated
rows in $M_\sP^\sQ$.

\begin{example}\label{example:param}%
  Let $Y = \SG(4,12)$, $\sQ = \{ 2, 3, 8, 9 \}$, and $\sP = \{ 5, 7, 10, 12 \}$.
  Then $M_\sP^\sQ$ is the variety of all matrices of the form
  \[
    \setcounter{MaxMatrixCols}{12}
    A = \begin{bmatrix}
      0 & 1 & a_{1,3} & a_{1,4} & a_{1,5} & 0 & 0 & 0 & 0 & 0 & 0 & 0 \\
      0 & 0 & 1 & a_{2,4} & a_{2,5} & a_{2,6} & a_{2,7} & 0 & 0 & 0 & 0 & 0 \\
      0 & 0 & 0 & 0 & 0 & 0 & 0 & 1 & a_{3,9} & a_{3,10} & 0 & 0 \\
      0 & 0 & 0 & 0 & 0 & 0 & 0 & 0 & 1 & a_{4,10} & a_{4,11} & a_{4,12}
    \end{bmatrix} \,,
  \]
  such that $a_{1,5} \neq 0$, $a_{2,7} \neq 0$, $a_{3,10} \neq 0$, $a_{4,12}
  \neq 0$, and the rows of $A$ are pairwise perpendicular. The variety $M_\sP^\sQ$
  has 12 unassigned entries and 4 pairs of correlated rows, so $\dim(Y_\sP^\sQ) =
  8$.
\end{example}

\begin{remark}\label{remark:dimdiff}%
  Let $\sQ \leq \sP$ be Schubert symbols for $Y = \SG(m,2n)$ and $1 \leq k \leq m$.
  Set $\sQ' = \sQ \ssm \{q_k\}$ and $\sP' = \sP \ssm \{p_k\}$. Then $\sQ' \leq \sP'$ are
  Schubert symbols for $Y' = \SG(m-1,2n)$ and we have
  \[
    \dim Y_\sP^\sQ - \dim {Y'}_{\sP'}^{\sQ'} = (p_k - q_k) -
    \# \{ j \in [1,m] \mid j \neq k \text{ and } q_j+q_k < 2n+1 < p_j+p_k \} \,.
  \]
  This is the number of unassigned entries in row $k$ of $M_\sP^\sQ$, minus the
  number of rows correlated to row $k$.
\end{remark}

\targetsec{openrich}{%
Let $\oY_\sP^\sQ \subset Y_\sP^\sQ$ be the open subvariety defined by
\[
  \oY_\sP^\sQ = \{ V \in Y_\sP^\sQ \mid
  \forall\, 1 \leq i \leq m :
  V \cap E_{[q_i+1,p_i]} = V \cap E_{[q_i,p_i-1]} = 0 \} \,.
\]
The following result confirms a conjecture of Ravikumar \cite[Conj.~6.5.3]{ravikumar:triple}.}

\begin{thm}\label{thm:param}%
  The variety $\oY_\sP^\sQ$ is a dense open subset of $Y_\sP^\sQ$. Moreover, the map
  $M_\sP^\sQ \to \oY_\sP^\sQ$ sending a matrix to its row span is an isomorphism of
  varieties.
\end{thm}
\begin{proof}
  Since $Y_\sP^\sQ$ is irreducible and the subsets
  \[
    \begin{split}
      U^L_i &= \{ V \in Y_\sP^\sQ \mid V \cap E_{[q_i+1,p_i]} = 0 \}
      \text{ \ \ \ and} \\
      U^R_i &= \{ V \in Y_\sP^\sQ \mid V \cap E_{[q_i,p_i-1]} = 0 \}
    \end{split}
  \]
  are open in $Y_\sP^\sQ$, the first claim will follow if we can show that $U^L_i$
  and $U^R_i$ are non-empty for all $1 \leq i \leq m$. By replacing $Y_\sP^\sQ$ with
  $Y_{\sQ^\vee}^{\sP^\vee}$, if required, we may assume that $q_1+p_m \geq 2n+1$.
  The sets $U^L_m$ and $U^R_1$ are non-empty since $E_\sQ \in U^L_m$ and $E_\sP \in
  U^R_1$. Set $\Omega = \{ V \in Y \mid V \subset E_{[1,p_m-1]} \}$. Then $Y_\sP
  \cap \Omega$ is a $B$-stable proper closed subset of $Y_\sP$, so $Y_\sP \cap
  \Omega$ is a union of Schubert varieties $Y_{\wh\sP}$ that are properly
  contained in $Y_\sP$. It follows that $Y_\sP^\sQ \cap \Omega$ is a union of
  Richardson varieties $Y_{\wh\sP}^\sQ$ that are properly contained in $Y_\sP^\sQ$.
  Therefore, $U^L_m \ssm \Omega$ is a dense open subset of $Y_\sP^\sQ$.

  Set $Y' = \SG(m-1,2n)$, $\sQ' = \{q_1 < \dots < q_{m-1}\}$, and $\sP' = \{ p_1 <
  \dots < p_{m-1} \}$. By induction we may assume ${\oY'}_{\sP'}^{\sQ'} \neq
  \emptyset$. By \Remark{dimdiff}, the condition $q_1+p_m \geq 2n+1$ implies
  that $\dim({Y'}_{\sP'}^{\sQ'}) \leq \dim(Y_\sP^\sQ)$. In fact, if row $i$ of $M_\sP^\sQ$
  is correlated to row $m$, then $2n+1-p_m \leq q_1 \leq q_i < 2n+1-q_m$, so row
  $m$ is correlated to at most $p_m-q_m$ rows. Using \Proposition{ontorich}, we
  can therefore choose a point $(V',V) \in Z_{\sP',\sP}^{\sQ',\sQ}$ such that $V' \in
  {\oY'}_{\sP'}^{\sQ'}$ and $V \in U^L_m \ssm \Omega$. Since $V' \subset
  E_{[1,p_m-1]}$ and $V \not\subset E_{[1,p_m-1]}$, we have $V' = V \cap
  E_{[1,p_m-1]}$. The condition $V' \in {\oY'}_{\sP'}^{\sQ'}$ therefore implies that
  $V \in U^L_i \cap U^R_i$ for $1 \leq i \leq m-1$. Set $L = V \cap
  E_{[q_m,p_m]}$. Since $V' \subset E_{[1,p_{m-1}]}$, we obtain $V' \cap L
  \subset V' \cap E_{[q_{m-1}+1,p_{m-1}]} = 0$, hence $V = V' \oplus L$ and
  $\dim(L) = 1$. Since $V' \subset E_{[1,p_m-1]}$ and $V \not\subset
  E_{[1,p_m-1]}$, it follows that $L \not\subset E_{[1,p_m-1]}$. We deduce that
  $V \cap E_{[q_m,p_m-1]} = L \cap E_{[q_m,p_m-1]} = 0$, so that $V \in U^R_m$.
  We conclude that $V \in \oY_\sP^\sQ$, so this set is a dense open subset of
  $Y_\sP^\sQ$.

  It is clear from the definitions that $A \mapsto \Span(A)$ is a well defined
  morphism of varieties $M_\sP^\sQ \to \oY_\sP^\sQ$. On the other hand, given $V \in
  \oY_\sP^\sQ$, each space $L_i = V \cap E_{[q_i,p_i]}$ is one-dimensional, for $1
  \leq i \leq m$. In addition, if we write $L_i = \C u_i$ with $u_i \in E$, then
  the $q_i$-th and $p_i$-th coordinates of $u_i$ are non-zero. By rescaling
  $u_i$, we may assume that the $q_i$-th coordinate is $1$. Let $A$ be the $m
  \times (2n)$ matrix whose $i$-th row is $u_i$. Then $A \in M_\sP^\sQ$ and
  $\Span(A) = V$. This completes the proof.
\end{proof}

\subsection{Solvable and movable rows}\label{sec:movable}%

\targetsec{constraints}{
Let $\sQ \leq \sP$ be Schubert symbols for $Y = \SG(m,2n)$, let $1 \leq k \leq m$,
and let $A = (a_{i,j}) \in M_\sP^\sQ$. Define the \emph{submatrix of constraints} on
row $k$ in $A$ to be the matrix $A[k]$ with entries $a_{i,j}$ for which $i \neq
k$, $q_i+q_k < 2n+1 < p_i+p_k$, and $2n+1-p_k \leq j \leq 2n+1-q_k$. This matrix
has one row for each row correlated to the $k$-th row of $M_\sP^\sQ$. For example,
if $A$ is the matrix of \Example{param}, then the submatrix of constraints on
row $2$ is the matrix
\[
  A[2] = \begin{bmatrix}
    0 & 0 & 1 & a_{3,9} & a_{3,10} \\
    0 & 0 & 0 & 1 & a_{4,10}
  \end{bmatrix} \,.
\]
}
The constraints \eqn{perp} on row $k$ in $A$ imposed by the other rows depend
only on the entries of $A[k]$. We will say that the vector $v = (v_{q_k}, \dots,
v_{p_k}) \in \C^{p_k-q_k+1}$ is \emph{perpendicular} to $A[k]$ if the entries of
$v$ satisfy the quadratic equations \eqn{perp} imposed on the $k$-th row in $A$,
that is,
\[
  \sum_{t=1}^n (a_{i,t}\, v_{2n+1-t} - a_{i,2n+1-t}\, v_t) = 0
\]
for all $i \neq k$ with $q_i+q_k < 2n+1 < p_i+p_k$, where we set $v_t = 0$ for
$t \notin [q_k,p_k]$.

Set $\sQ' = \sQ \ssm \{q_k\}$, $\sP' = \sP \ssm \{p_k\}$, and $Y' = \SG(m-1,2n)$.
Motivated by \Proposition{ontorich} and \Theorem{param}, we will say that the
$k$-th row of $M_\sP^\sQ$ is \emph{solvable} if $\dim({Y'}_{\sP'}^{\sQ'}) \leq
\dim(Y_\sP^\sQ)$. By \Remark{dimdiff}, this means that there are at most $p_k-q_k$
constraints on the $k$-th row of $M_\sP^\sQ$. The $k$-th row of $M_\sP^\sQ$ is
\emph{movable} if $\dim({Y'}_{\sP'}^{\sQ'}) < \dim(Y_\sP^\sQ)$, that is, there are fewer
than $p_k-q_k$ constraints on the $k$-th row. If the $k$-th row of $M_\sP^\sQ$ is
movable, then for each matrix $A \in M_\sP^\sQ$, we can vary the $k$-th row of $A$
in a positive dimensional parameter space while fixing the remaining rows.

\begin{cor}\label{cor:fullrank}%
  Let $\sQ \leq \sP$ be Schubert symbols for $\SG(m,2n)$, and assume that the $k$-th
  row of $M_\sP^\sQ$ is solvable. Then $M_\sP^\sQ$ contains a dense open subset of
  points $A$ for which the submatrix $A[k]$ of constraints on row $k$ has
  linearly independent rows.
\end{cor}
\begin{proof}
  Set $\sQ' = \sQ \ssm \{ q_k \}$ and $\sP' = \sP \ssm \{ p_k \}$. Given $A \in M_\sP^\sQ$,
  let $A' \in M_{\sP'}^{\sQ'}$ denote the result of removing the $k$-th row from
  $A$. It follows from \Proposition{ontorich} and \Theorem{param} that $A
  \mapsto A'$ defines a dominant morphism $M_\sP^\sQ \to M_{\sP'}^{\sQ'}$. This implies
  that, for all points $A$ in a dense open subset of $M_\sP^\sQ$, the fiber over
  $A'$ in $M_\sP^\sQ$ is non-empty of dimension $\dim(M_\sP^\sQ) - \dim(M_{\sP'}^{\sQ'})$.
  This fiber can be identified with the set of vectors $v = (1, v_{q_k+1},
  \dots, v_{p_k})$, with $v_{p_k} \neq 0$, that are perpendicular to $A[k]$. We
  deduce that the rows of $A[k]$ are linearly independent by \Remark{dimdiff}.
\end{proof}

\subsection{Perpendicular incidence varieties}\label{sec:ss:incidence}%

\targetsec{perp}{%
Let $Y = \SG(m,2n)$ and define the \emph{perpendicular incidence variety}
\[
  S = \{ (V,L) \in Y \times \bP(E) \mid V \subset L^\perp \} \,.
\]
Let $f : S \to \bP(E)$ and $g : S \to Y$ be the projections. Given Schubert
symbols $\sQ \leq \sP$ for $Y$, we set $S_\sP^\sQ = g^{-1}(Y_\sP^\sQ)$. Since $g$ is locally
trivial with fibers $g^{-1}(V) = \bP(V^\perp)$ by
\cite[Prop.~2.3]{buch.chaput.ea:finiteness}, it follows that $S_\sP^\sQ$ is
irreducible with $\dim(S_\sP^\sQ) = \dim(Y_\sP^\sQ) + 2n - m - 1$.}

\targetsec{cut}{%
Following \cite{buch.kresch.ea:quantum, ravikumar:triple*1}, we define a
\emph{cut} of $M_\sP^\sQ$ to be an integer $c \in [0,2n]$ such that $p_i \leq c$ or
$c < q_i$ holds for each $i \in [1,m]$. This implies that no row of $M_\sP^\sQ$
contains non-zero entries in both column $c$ and column $c+1$. A \emph{lone
star} is an integer $s \in [1,2n]$ such that $q_i = p_i = s$ for some $i \in
[1,m]$. This implies that $s-1$ and $s$ are cuts of $M_\sP^\sQ$. The integer $c$ is
a \emph{double-cut} of $M_\sP^\sQ$ if both $c$ and $2n-c$ are cuts. A
\emph{component} of $M_\sP^\sQ$ is a pair of integers $(a,b)$, with $0 \leq a < b
\leq n$, such that (i) $a$ is a double-cut, (ii) $b$ is a double-cut or $b=n$,
and (iii) no double-cut belongs to $[a+1,b-1]$. We will say that row $i$ of
$M_\sP^\sQ$ is contained in the component $(a,b)$ if $a < q_i \leq p_i \leq b$, or
$2n-b < q_i \leq p_i \leq 2n-a$, or $b=n$ and $a < q_i \leq p_i \leq 2n-a$. Each
row of $M_\sP^\sQ$ belongs to a unique component, and two rows can be correlated
only if they belong to the same component. Any component $(a,b)$ contains at
most $b-a$ rows. The component $(a,b)$ is called a \emph{quadratic component} if
$b$ is a double-cut, $b-a \geq 2$, and $(a,b)$ contains $b-a$ rows.}

\targetsec{complint}{%
Let $\bP_\sP^\sQ \subset \bP(E)$ denote the closed subvariety defined by the linear
equations $x_{2n+1-s} = 0$ for all lone stars $s$ of $M_\sP^\sQ$, as well as the
quadratic equations
\[
  x_{a+1} x_{2n-a} + \dots + x_b x_{2n+1-b} = 0
\]
given by all quadratic components $(a,b)$ of $M_\sP^\sQ$. Using that the quadratic
equations involve pairwise disjoint sets of variables, it follows that $\bP_\sP^\sQ$
is an irreducible complete intersection in $\bP(E)$ with rational singularities.}

\begin{example}
  Let $Y = \SG(8,20)$ and define $\sQ = \{1,2,4,6,9,11,16,18\}$ and $\sP =
  \{2,3,7,8,11,12,16,20\}$. The shape of non-zero entries in $M_\sP^\sQ$ is given by
  the diagram:
  \[
    \newcommand{\nd}{\cdot}
    \newcommand{\nS}{\star}
    \newcommand{\md}{&\cdot}
    \newcommand{\mS}{&\star}
    \left[
      \begin{array}{ccc|ccccc|cccc|ccccc|ccc}
        \nS\mS\md\md\md\md\md\md\md\md\md\md\md\md\md\md\md\md\md\md\\
        \nd\mS\mS\md\md\md\md\md\md\md\md\md\md\md\md\md\md\md\md\md\\
        \nd\md\md\mS\mS\mS\mS\md\md\md\md\md\md\md\md\md\md\md\md\md\\
        \nd\md\md\md\md\mS\mS\mS\md\md\md\md\md\md\md\md\md\md\md\md\\
        \nd\md\md\md\md\md\md\md\mS\mS\mS\md\md\md\md\md\md\md\md\md\\
        \nd\md\md\md\md\md\md\md\md\md\mS\mS\md\md\md\md\md\md\md\md\\
        \nd\md\md\md\md\md\md\md\md\md\md\md\md\md\md\mS\md\md\md\md\\
        \nd\md\md\md\md\md\md\md\md\md\md\md\md\md\md\md\md\mS\mS\mS\\
      \end{array}
    \right]
  \]
  Here we ignore that the lone star in column $16$ forces the entry in column
  $5$ to vanish. The double-cuts of $M_\sP^\sQ$ are indicated with vertical
  lines. The components of $M_\sP^\sQ$ are $(0,3)$, $(3,8)$, and $(8,10)$. The
  component $(0,3)$ is quadratic and contains rows $1$, $2$, and $8$; component
  $(3,8)$ contains rows $3$, $4$, and $7$; and component $(8,10)$ contains rows
  $5$ and $6$. All rows except row $7$ are solvable, and rows $3$, $4$, and $5$
  are movable. The complete intersection $\bP_\sP^\sQ$ is given
  by
  \[
    \bP_\sP^\sQ \,=\, Z(x_5\,,\, x_1 x_{20} + x_2 x_{19} + x_3 x_{18})
    \,\subset\, \bP^{19} \,.
  \]
\end{example}

Our main result about perpendicular incidences is the following theorem, which
will be proved at the end of this section.

\begin{thm}\label{thm:oimain}%
  Let $\sQ \leq \sP$ be Schubert symbols for $Y = \SG(m,2n)$. Then $f$ restricts to
  a surjective morphism $f : S_\sP^\sQ \to \bP_\sP^\sQ$ with rational general fibers.
\end{thm}

The analogue of \Theorem{oimain} with $S \subset Y \times \bP(E)$ defined by the
condition $L \subset V$ has been established in \cite{buch.kresch.ea:quantum,
buch.ravikumar:pieri, ravikumar:triple*1}. When $Y = \LG(n,2n)$ is a Lagrangian
Grassmannian, the conditions $V \subset L^\perp$ and $L \subset V$ are
equivalent, so this case of \Theorem{oimain} is equivalent to
\cite[Lemma~5.2]{buch.ravikumar:pieri}. One new complication in our case is that
$S$ is not a flag variety, so the map $f : S_\sP^\sQ \to \bP_\sP^\sQ$ is not a
projection from a Richardson variety, as studied in e.g.\
\cite{billey.coskun:singularities, knutson.lam.ea:projections,
buch.chaput.ea:positivity}.

\begin{lemma}\label{lemma:midcross}%
  Let $\sQ \leq \sP$ be Schubert symbols for $\SG(m,2n)$ and let $1 \leq k \leq m$.
  If $q_k \leq n < p_k$, then row $k$ of $M_\sP^\sQ$ is movable.
\end{lemma}
\begin{proof}
  Assume that row $j$ is correlated to row $k$. If $j < k$, then $2n+1-p_k < p_j
  < p_k$, which holds for at most $p_k-n-1$ rows $j$. If $j > k$, then $q_k <
  q_j < 2n+1-q_k$, which holds for at most $n-q_k$ rows $j$. It follows that row
  $k$ is correlated to at most $p_k-q_k-1$ rows.
\end{proof}

\begin{prop}\label{prop:nonmovable}%
  Let $\sQ \leq \sP$ be Schubert symbols for $Y = \SG(m,2n)$, and let $(a,b)$ be a
  component of $M_\sP^\sQ$ with $b-a \geq 2$. Then $(a,b)$ is a quadratic component
  if and only if no row contained in $(a,b)$ is movable. In this case all rows
  contained in $(a,b)$ are solvable, and $M_\sP^\sQ$ has no cuts $c$ with $a < c <
  b$ or $\,2n-b < c < 2n-a$.
\end{prop}
\begin{proof}
  Since two rows of $M_\sP^\sQ$ can be correlated only if they belong to the same
  component, we may assume that $(a,b) = (0,n)$ is the only component of
  $M_\sP^\sQ$. By \Lemma{midcross} we may further assume that $n$ is a cut. By
  replacing $M_\sP^\sQ$ with $M_{\sQ^\vee}^{\sP^\vee}$, if necessary, we may also assume
  that $p_m = 2n$. Set $r = p_m-q_m \geq 1$. If row $m$ of $M_\sP^\sQ$ is not
  movable, then since $1 \notin \sP$ and $r+1 = 2n+1-q_m \notin \sQ$, we must have
  $q_i = i < p_i$ for $1 \leq i \leq r$. The same conclusion holds if $(0,n)$ is
  a quadratic component of $M_\sP^\sQ$, since in this case we have $x \in \sQ$ or
  $2n+1-x \in \sQ$ for all $x \in [1,n]$. Set $\sQ' = (\sQ \ssm \{r,q_m\}) \cup
  \{r+1,q_m+1\}$, so that the shape of $M_\sP^{\sQ'}$ is obtained from the shape of
  $M_\sP^\sQ$ by removing the leftmost entry from rows $r$ and $m$. Then $M_\sP^\sQ$ and
  $M_\sP^{\sQ'}$ have the same pairs of correlated rows, except that rows $r$ and
  $m$ are correlated in $M_\sP^\sQ$ but not in $M_\sP^{\sQ'}$. It follows that any row
  is movable in $M_\sP^\sQ$ if and only if it is movable in $M_\sP^{\sQ'}$, and the same
  holds with movable replaced by solvable. The component $(0,n)$ is quadratic if
  and only if $m=n$. Since $M_\sP^{\sQ'}$ has no empty components, $m=n$ holds if
  and only if all components of $M_\sP^{\sQ'}$ are quadratic or lone stars. By
  induction on $\sum_{i=1}^m (p_i-q_i)$, this holds if and only if $M_\sP^{\sQ'}$
  has no movable rows, which proves the first claim. Assuming that $(0,n)$ is a
  quadratic component, it also follows by induction that all rows of $M_\sP^\sQ$ are
  solvable. Noting that all double-cuts of $M_\sP^{\sQ'}$ belong to the set $\{0, r,
  n\}$, it follows by induction that all cuts of $M_\sP^{\sQ'}$ belong to $\{0, r,
  n, 2n-r, 2n\}$. The last claim follows from this since $r$ and $2n-r$ are not
  cuts of $M_\sP^\sQ$.
\end{proof}

\begin{cor}\label{cor:remsolv}%
  Let $\sQ \leq \sP$ be Schubert symbols, and assume that row $k$ in $M_\sP^\sQ$ is
  movable. Then $\bP_\sP^\sQ = \bP_{\sQ'}^{\sP'}$, where $\sQ' = \sQ \ssm \{ q_k \}$ and $\sP'
  = \sP \ssm \{ p_k \}$.
\end{cor}
\begin{proof}
  This holds because a movable row cannot be a lone star and cannot belong to a
  quadratic component by \Proposition{nonmovable}.
\end{proof}

\targetsec{M-hat}{%
Given Schubert symbols $\sQ \leq \sP$ for $Y = \SG(m,2n)$, define the variety
\[
  \wh M_\sP^\sQ = \{ (A,u) \in M_\sP^\sQ \times E \mid A \perp u \} \,,
\]
where $A \perp u$ indicates that $u$ is perpendicular to all rows of $A$. The
variety $\wh M_\sP^\sQ$ is irreducible with $\dim(\wh M_\sP^\sQ) = \dim(M_\sP^\sQ) + 2n-m$.}

\begin{prop}\label{prop:delmovable}%
  Let $\sQ \leq \sP$ be Schubert symbols for $Y = \SG(m,2n)$ and assume that the
  $k$-th row of $M_\sP^\sQ$ is movable. Set $\sQ' = \sQ \ssm \{ q_k \}$, $\sP' = \sP \ssm \{
  p_k \}$, and $r = \dim(M_\sP^\sQ) - \dim(M_{\sP'}^{\sQ'}) > 0$. Let $\pi : \wh M_\sP^\sQ
  \to \wh M_{\sP'}^{\sQ'}$ be the projection that forgets row $k$ in its first
  argument. There exists a morphism $\phi : \wh M_\sP^\sQ \to \C^{r-1}$, given by
  projection to $r-1$ of the entries of the $k$-th row of $M_\sP^\sQ$, such that the
  morphism $\pi \times \phi : \wh M_\sP^\sQ \to \wh M_{\sP'}^{\sQ'} \times \C^{r-1}$ is
  birational.
\end{prop}
\begin{proof}
  By \Corollary{fullrank} we can choose $A \in M_\sP^\sQ$ such that the submatrix
  $A[k]$ of constraints on row $k$ has linearly independent rows. The number of
  rows in $A[k]$ is equal to $p_k-q_k-r$ by \Remark{dimdiff}. We can therefore
  choose a vector
  \[
    (u_{2n+1-p_k}, \dots, u_{2n+1-q_k}) \in \C^{p_k-q_k+1}
  \]
  which is perpendicular to the $k$-th row of $A$ and not in the row span of
  $A[k]$. Using that $a_{i,q_i}=1$ and $a_{i,p_i} \neq 0$ for each row $i$, we
  can extend this vector to $u \in E$, such that $u$ is perpendicular to all
  rows of $A$. Let $A' \in M_{\sP'}^{\sQ'}$ be the result of removing the $k$-th row
  from $A$. Then the fiber of $\pi : \wh M_\sP^\sQ \to \wh M_{\sP'}^{\sQ'}$ over
  $(A',u)$ contains $(A,u)$, so it is not empty. This fiber can be identified
  with the set of vectors $(1, v_{q_k+1}, \dots, v_{p_k})$, with $v_{p_k} \neq
  0$, which are perpendicular to both $A[k]$ and $(u_{2n+1-p_k}, \dots,
  u_{2n+1-q_k})$. Therefore the fiber has dimension $r-1 = \dim(\wh M_\sP^\sQ) -
  \dim(\wh M_{\sP'}^{\sQ'})$. Since $\wh M_{\sP'}^{\sQ'}$ is irreducible, this implies
  that $\pi : \wh M_\sP^\sQ \to \wh M_{\sP'}^{\sQ'}$ is dominant. It also follows that
  $(A,u)$ is determined by $(A',u)$ together with some collection of $r-1$
  entries $a_{k,j}$ from the $k$-th row of $A$. Since this holds whenever a
  particular minor in $(A',u)$ is non-zero, we deduce that $(A,u)$ is determined
  by $(A',u)$ and the same $r-1$ entries from row $k$, for all points $(A,u)$ in
  a dense open subset of $\wh M_\sP^\sQ$. The result follows from this.
\end{proof}

\targetsec{S-prime}{%
Assume that $c \in [1,n-1]$ is a double-cut of $M_\sP^\sQ$, and set $\sQ' = \sQ \cap
[c+1,2n-c]$, $\sP' = \sP \cap [c+1,2n-c]$, $\sQ'' = \sQ \ssm \sQ'$, and $\sP'' = \sP \ssm \sP'$.
Set $m' = \#\sP'$, $Y' = \SG(m',2n)$, $m'' = \#\sP''$, $Y'' = \SG(m'',2n)$, and let
$S' \subset Y' \times \bP(E)$ and $S'' \subset Y'' \times \bP(E)$ be the
corresponding perpendicular incidence varieties, with projections $f' : S' \to
\bP(E)$ and $f'' : S'' \to \bP(E)$. Since we have $\bP_\sP^\sQ = \bP_{\sP'}^{\sQ'} \cap
\bP_{\sP''}^{\sQ''}$, the following lemma shows that \Theorem{oimain} can be proved
under the assumption that $M_\sP^\sQ$ has only one component $(0,n)$.}

\begin{lemma}\label{lemma:richprod}%
  The map $(V',V'') \mapsto V' \oplus V''$ is an isomorphism ${Y'}_{\sP'}^{\sQ'}
  \times {Y''}_{\sP''}^{\sQ''} \cong Y_\sP^\sQ$, and we have $f(S_\sP^\sQ) =
  f'({S'}_{\sP'}^{\sQ'}) \cap f''({S''}_{\sP''}^{\sQ''})$. For all points $L \in
  f(S_\sP^\sQ)$, the fiber of $f : S_\sP^\sQ \to f(S_\sP^\sQ)$ over $L$ is the product of
  the fibers of $f' : {S'}_{\sP'}^{\sQ'} \to f'({S'}_{\sP'}^{\sQ'})$ and $f'' :
  {S''}_{\sP''}^{\sQ''} \to f''({S''}_{\sP''}^{\sQ''})$ over $L$.
\end{lemma}
\begin{proof}
  Set $E' = E_{[c+1,2n-c]}$ and $E'' = E_{[1,c] \cup [2n-c+1,2n]}$. Using that
  $V' \subset E'$ holds for all $V' \in {Y'}_{\sP'}^{\sQ'}$, and $V'' \subset E''$
  holds for all $V'' \in {Y''}_{\sP''}^{\sQ''}$, it follows that the map
  ${Y'}_{\sP'}^{\sQ'} \times {Y''}_{\sP''}^{\sQ''} \to Y_\sP^\sQ$ is well defined. The
  inverse map $V \mapsto (V\cap E', V\cap E'')$ is well defined because $\dim(V
  \cap E') = m'$, $\dim(V \cap E_{[1,c]}) = \#(\sP \cap [1,c])$, and $\dim(V \cap
  E_{[2n-c+1,2n]}) = \#(\sP \cap [2n-c+1,2n])$ holds for all $V \in Y_\sP^\sQ$. This
  proves the first claim. The remaining claims follow because $V' \oplus V''
  \subset L^\perp$ is equivalent to $V' \subset L^\perp$ and $V'' \subset
  L^\perp$.
\end{proof}

\begin{proof}[Proof of \Theorem{oimain}]
  We may assume that $(0,n)$ is the only component of $M_\sP^\sQ$ by
  \Lemma{richprod}. If $M_\sP^\sQ$ has no movable rows, then
  \Proposition{nonmovable} implies that $m=n$, so the claim follows from
  \cite[Lemma~5.2]{buch.ravikumar:pieri}. Otherwise $M_\sP^\sQ$ has at least one
  movable row, say row $k$. Set $Y' = \SG(m-1,2n)$, $\sQ' = \sQ \ssm \{ q_k \}$, $\sP'
  = \sP \ssm \{ p_k \}$, and $r = \dim(M_\sP^\sQ) - \dim(M_{\sP'}^{\sQ'})$. Let $S \subset
  Y \times \bP(E)$ and $S' \subset Y' \times \bP(E)$ be the perpendicular
  incidence varieties, with projections $f : S \to \bP(E)$ and $f' : S' \to
  \bP(E)$. It follows from \Proposition{delmovable} that $f(S_\sP^\sQ) =
  f'({S'}_{\sP'}^{\sQ'})$, and for all points $L$ in a dense open subset of
  $f(S_\sP^\sQ)$, the fiber $f^{-1}(L) \cap S_\sP^\sQ$ is birational to $({f'}^{-1}(L)
  \cap {S'}_{\sP'}^{\sQ'}) \times \C^{r-1}$. The result therefore follows by
  induction on $m$.
\end{proof}


\section{Gromov-Witten invariants of Pieri type}\label{sec:pierigw}%

\subsection{Incidences of projected Richardson varieties}\label{sec:incprojrich}%

\targetsec{ZSF}{%
Let $X = \LG(n,2n)$ be a Lagrangian Grassmannian and $Y = \SG(m,2n)$ a
symplectic Grassmannian. Set $Z = \SF(m,n;2n)$ and let $p : Z \to X$ and $q : Z
\to Y$ be the projections. We also set $\wh X = \SF(1,n;2n)$, with projections
$\eta : \wh X \to \bP^{2n-1}$ and $\pi : \wh X \to X$. Our computation of
Gromov-Witten invariants of $X$ is based on the following result.}

\begin{thm}\label{thm:projtriv}%
  Let $\sQ \leq \sP$ be Schubert symbols for $Y = \SG(m,2n)$, and let $X_\sP^\sQ =
  p(q^{-1}(Y_\sP^\sQ))$ be the corresponding projected Richardson variety in $X =
  \LG(n,2n)$. Then $\eta$ restricts to a cohomologically trivial morphism $\eta
  : \pi^{-1}(X_\sP^\sQ) \to \bP_\sP^\sQ$.
\end{thm}
\begin{proof}
  Define the non-singular varieties $S = \{ (K,L) \in Y \times \bP^{2n-1} \mid K
  \subset L^\perp \}$ and $\wh Z = Z \times_X \wh X = \{ (K,V,L) \in Y \times X
  \times \bP^{2n-1} \mid K \subset V \supset L \}$. Consider the following
  commutative diagram, where all morphisms are the natural projections.
  \[
    \xymatrix{
      S \ar[drrr]^f \ar[dddr]_g \\
      & \wh Z \ar[r]_{\wh p} \ar[d]^\tau \ar[ul]_\sigma &
      \wh X \ar[r]_{\eta\ \ \ } \ar[d]^\pi & \bP^{2n-1} \\
      & Z \ar[r]_p \ar[d]^q & X \\
      & Y
    }
  \]
  Since the morphisms of this diagram are equivariant for the action of
  $\Sp(2n)$, and all targets other than $S$ are flag varieties of $\Sp(2n)$, it
  follows that all morphisms other than $\sigma$ are locally trivial fibrations
  with non-singular fibers \cite[Prop.~2.3]{buch.chaput.ea:finiteness}. Let
  $Z_\sP^\sQ$, $\wh Z_\sP^\sQ$, and $S_\sP^\sQ$ be the inverse images of $Y_\sP^\sQ$ in $Z$,
  $\wh Z$, and $S$, respectively, and set $\wh X_\sP^\sQ = \pi^{-1}(X_\sP^\sQ)$. Since
  $Y_\sP^\sQ$ and $X_\sP^\sQ$ have rational singularities \cite{brion:positivity,
  billey.coskun:singularities, knutson.lam.ea:projections}, it follows that
  $Z_\sP^\sQ$, $\wh Z_\sP^\sQ$, $S_\sP^\sQ$, and $\wh X_\sP^\sQ$ have rational singularities as
  well.

  All fibers of $\sigma$ are rational. In fact, for $(K,L) \in S$ we have
  $\sigma^{-1}(K,L) = \LG(m',(K+L)^\perp/(K+L))$, where $m' = n - \dim(K+L)$.
  Since $\wh Z_\sP^\sQ = \sigma^{-1}(S_\sP^\sQ)$, this implies that $\sigma : \wh Z_\sP^\sQ
  \to S_\sP^\sQ$ is cohomologically trivial \cite[Thm.~3.1]{buch.mihalcea:quantum}.
  Since $f : S_\sP^\sQ \to \bP_\sP^\sQ$ is cohomologically trivial by \Theorem{oimain},
  it follows that $\eta \wh p = f \sigma : \wh Z_\sP^\sQ \to \bP_\sP^\sQ$ is
  cohomologically trivial \cite[Lemma~2.4]{buch.chaput.ea:projected}.

  Using that the outer rectangle and the right square of the following diagram
  are fiber squares, it follows that $\wh p : \wh Z_\sP^\sQ \to \wh X_\sP^\sQ$ is the
  base extension of $p : Z_\sP^\sQ \to X_\sP^\sQ$ along $\pi$.
  \[
    \xymatrix{
      \wh Z_\sP^\sQ \ar[r]_{\wh p} \ar[d]^\tau &
      \wh X_\sP^\sQ \ar[r]_\subset \ar[d]^\pi & \wh X \ar[d]^\pi \\
      Z_\sP^\sQ \ar[r]_p & X_\sP^\sQ \ar[r]_\subset & X
    }
  \]
  This implies that $\wh p : \wh Z_\sP^\sQ \to \wh X_\sP^\sQ$ is cohomologically
  trivial, for example because its general fibers are Richardson varieties by
  \cite[Cor.~2.11]{buch.chaput.ea:positivity}. It follows that $\eta : \wh
  X_\sP^\sQ \to \bP_\sP^\sQ$ is cohomologically trivial, which completes the
  proof.
\end{proof}

\subsection{Gromov-Witten invariants of Pieri type}\label{sec:ss:gwpieri}

The Schubert varieties in $X = \LG(n,2n)$ are indexed by shapes $\la \subset
\cP_X$. The Schubert symbol $\sP$ corresponding to $\la \subset \cP_X$ is obtained
as follows. The border of $\la$ forms a path from the upper-right corner of
$\cP_X$ to the diagonal. Number the steps of this path from $1$ to $n$, starting
from the upper-right corner. Then $\sP$ consists of the integers $i$ for which the
$i$-th step is horizontal, and the integers $2n+1-i$ for which the $i$-th step
is vertical. By observing that the map from shapes to Schubert symbols is
compatible with the Bruhat order, this description of the Schubert varieties in
$X$ follows from e.g. \cite[Lemma~2.9]{buch.samuel:k-theory}.

\begin{example}
  Let $X = \LG(7,14)$ and $\la = (7,4,2,1)$. Then $\la$ corresponds to the
  Schubert symbol $\sP = \{ 2, 3, 5, 8, 9, 11, 14 \}$.

  \begin{center}
    \begin{tikzpicture}[x=.4cm,y=.4cm]
      \def\zz{-- ++(-1,0) -- ++(0,1)}
      \draw (0,7) -- (7,7) -- (7,0) \zz\zz\zz\zz\zz\zz
      -- ++(-1,0) -- cycle;
      \draw [very thick] (7,7) -- ++(0,-1) -- ++(-2,0) -- ++(0,-1) -- ++(-1,0)
      -- ++(0,-2);
    \end{tikzpicture}
  \end{center}
\end{example}

Recall that a classical shape $\la \subset \cP_X$ is identified with the quantum
shape $I([X^\la]) = \la \cup I(1)$ in $\wh\cP_X$, and $\la[d]$ is the result of
shifting this shape by $d$ diagonal steps for each $d \in \Z$.

\targetsec{qskew}{%
Let $\la, \mu \subset \cP_X$ be shapes and $d \geq 0$ a degree. Then
$\Gamma_d(X_\la,X^\mu) \neq \emptyset$ if and only if $\mu \subset \la[d]$. When
this holds, we let $\la[d]/\mu$ be the skew shape in $\wh\cP_X$ of boxes in
$\la[d]$ that are not contained in $\mu$. Let $R(\la[d]/\mu)$ denote the size of
a maximal rim in this skew shape, and let $N(\la[d]/\mu)$ be the number of
connected components of $\la[d]/\mu$ that are disjoint from both of the
diagonals in $\wh\cP_X$. The following result interprets \Theorem{projtriv} when
the projected Richardson variety in $X$ is a curve neighborhood
$\Gamma_d(X_\la,X^\mu)$.}

\begin{cor}\label{cor:oigamma}%
  Let $X = \LG(n,2n)$ and let $\la, \mu \subset \cP_X$ be shapes such
  that\linebreak
  $\Gamma_d(X_\la,X^\mu) \neq \emptyset$. Set $\theta = \la[d]/\mu$. If
  $R(\theta) = n+1$, then $\eta(\pi^{-1}(\Gamma_d(X_\la,X^\mu))) = \bP^{2n-1}$.
  Otherwise $\eta(\pi^{-1}(\Gamma_d(X_\la,X^\mu)))$ is a complete intersection
  in $\bP^{2n-1}$ of\linebreak dimension $n+R(\theta)-1$, defined by $N(\theta)$
  quadratic equations and $n-R(\theta)-N(\theta)$ linear equations. Moreover,
  the restricted map $\eta : \pi^{-1}(\Gamma_d(X_\la,X^\mu)) \to
  \eta(\pi^{-1}(\Gamma_d(X_\la,X^\mu)))$ is cohomologically trivial.
\end{cor}
\begin{proof}
  Write $X_\la = X_\sP$ and $X^\mu = X^\sQ$, where $\sP = \{p_1 < \dots < p_n\}$ and
  $\sQ = \{q_1 < \dots < q_n\}$ are the Schubert symbols corresponding to $\la$
  and $\mu$. Then $q(p^{-1}(X_\la)) = Y_{\sP'}$ and $q(p^{-1}(X^\mu)) = Y^{\sQ'}$,
  where $\sP' = \{ p_{d+1}, \dots, p_n \}$ and $\sQ' = \{ q_1, \dots, q_{n-d} \}$,
  so we have $\Gamma_d(X_\la,X^\mu) = X_{\sP'}^{\sQ'}$. \Theorem{projtriv} shows
  that
  \[
    \eta : \pi^{-1}(\Gamma_d(X_\la,X^\mu)) \to \bP_{\sP'}^{\sQ'}
  \]
  is cohomologically trivial. It remains to show that $\bP_{\sP'}^{\sQ'}$ is a
  complete intersection defined by the expected equations. If $R(\theta) = n+1$,
  then we can make $d$ and $\la$ smaller and $\mu$ larger until we obtain
  $R(\theta)=n$ and $N(\theta)=0$. This will make $\Gamma_d(X_\la,X^\mu)$
  smaller, while the corollary still asserts that
  $\eta(\pi^{-1}(\Gamma_d(X_\la,X^\mu))) = \bP^{2n-1}$. We may therefore assume
  that $R(\theta) \leq n$, which implies that the borders of $\mu$ and $\la[d]$
  meet somewhere. In particular, $\mu$ has at least $d$ vertical steps, and
  $\la[d]$ has at least $d$ horizontal steps.

  Let $\ell(\mu)$ be the number of vertical steps of $\mu$. Then $\mu$ has
  $n-\ell(\mu)$ horizontal steps. Notice that, if $1 \leq k \leq n-\ell(\mu)$,
  then $q_k$ is the step number of the $k$-th horizontal step of $\mu$, while if
  $n-\ell(\mu) < k \leq n$, then $2n+1-q_k$ is the step number of the
  $(n+1-k)$-th vertical step of $\mu$. Since the starting point of $\mu$ is $d$
  boxes north-west of the starting point of $\la$, and the endpoint of $\mu$ is
  north-west of the endpoint of $\la[d]$, we have $\ell(\mu) \leq \ell(\la)+d$.
  The condition $R(\theta) \leq n$ implies that $\ell(\mu) \geq d$ and
  $\ell(\la) \leq n-d$.

  Write $\sP' = \{ p'_1, \dots, p'_{n-d} \}$ and $\sQ' = \{ q'_1, \dots, q'_{n-d}
  \}$, where $p'_i = p_{i+d}$ and $q'_i = q_i$. It follows from the construction
  of $\sP$ and $\sQ$ from $\la$ and $\mu$ that the rows in $M_{\sP'}^{\sQ'}$ are in
  bijection with some of the steps of $\la[d]$, and also with some of the steps
  of $\mu$. We will explain how to obtain the resulting bijection between steps
  of $\la[d]$ and $\mu$, and how to obtain the rows of $M_{\sP'}^{\sQ'}$ from the
  corresponding pairs of steps in $\la[d]$ and $\mu$. This will include drawing
  \emph{connectors} between the paired steps of $\la[d]$ and $\mu$, see
  \Example{skew2M}.

  Consider row $k$ of $M_{\sP'}^{\sQ'}$. Assume first that $d+k \leq n-\ell(\la)$.
  Then $k \leq n-\ell(\mu)$, $q'_k$ is the step number of the $k$-th horizontal
  step of $\mu$, and $p'_k$ is the step number of the $(d+k)$-th horizontal step
  of $\la[d]$. These steps of $\mu$ and $\la[d]$ are in the same column, and
  $p'_k-q'_k$ is the distance (number of boxes) between the two steps. We draw a
  vertical line segment (connector) from the $k$-th horizontal step of $\mu$ to
  the $(d+k)$-th horizontal step of $\la[d]$.

  Assume next that $k > n-\ell(\mu)$. Then $d+k > n-\ell(\la)$, $2n+1-q'_k$ is
  the step number of the $(n+1-k)$-th vertical step of $\mu$, and $2n+1-p'_k$ is
  the step number of the $(n-d+1-k)$-th vertical step of $\la[d]$. These steps
  of $\mu$ and $\la[d]$ are in the same row, and $p'_k-q'_k$ is the distance
  between the two steps. We draw a horizontal line segment (connector) from the
  $(n+1-k)$-th vertical step of $\mu$ to the $(n-d+1-k)$-th vertical step of
  $\la[d]$.

  We finally assume that $d+k > n-\ell(\la)$ and $k \leq n-\ell(\mu)$. Then
  $q'_k$ is the step number of the $k$-th horizontal step of $\mu$, and
  $2n+1-p'_k$ is the step number of the $(n-d+1-k)$-th vertical step of
  $\la[d]$. In this case, if we draw a vertical line segment going down from the
  horizontal step of $\mu$, and a horizontal line segment going to the left from
  the vertical step of $\la[d]$, then these line segments meet in a diagonal box
  of $\wh\cP_X$. In this case the connector representing row $k$ of
  $M_{\sP'}^{\sQ'}$ is obtained by connecting the two line segments, and $p'_k-q'_k$
  is the number of boxes this connector passes through.

  It follows from this description that the lone stars of $M_{\sP'}^{\sQ'}$
  correspond to steps shared by $\mu$ and $\la[d]$, and there are exactly $n -
  R(\theta) - N(\theta)$ such steps. It also follows that, if $\mu$ and $\la[d]$
  meet after $c$ steps, then $c$ is a double-cut of $M_{\sP'}^{\sQ'}$. The only
  other cuts of $M_{\sP'}^{\sQ'}$ are the integers in the set $[0,q'_1-1] \cup
  [p'_{n-d},2n]$. We deduce that any component of $\theta$ that is disjoint from
  both diagonals in $\wh\cP_X$ produces a quadratic component of $M_{\sP'}^{\sQ'}$.
  If a component of $\theta$ meets the SW diagonal of $\wh\cP_X$, then the
  corresponding component of $M_{\sP'}^{\sQ'}$ contains a row that crosses the
  middle, so this component is not quadratic. Finally, if a component of
  $\theta$ intersects the NE diagonal of $\wh\cP_X$, then the corresponding
  component $(a,b)$ of $M_{\sP'}^{\sQ'}$ has fewer than $b-a$ rows, so it is not
  quadratic. It follows that $M_{\sP'}^{\sQ'}$ has exactly $N(\theta)$ quadratic
  components.
\end{proof}

\begin{example}\label{example:skew2M}%
  Let $X = \LG(12,24)$, $\mu = (12,11,9,6,5)$, and $\la = (11,8,6,3,1)$, and
  $d=2$. Then $\theta = \la[d]/\mu$ is the skew shape between the two thick
  black paths in the following picture. The connectors of $\theta$ are colored
  pink. We have $R(\theta) = 10$ and $N(\theta) = 1$.\medskip
  \begin{center}
    \begin{tikzpicture}[x=.4cm,y=.4cm]
      \def\zz{-- ++(0,-1) -- ++(1,0)}
      \draw (-4,4) \zz\zz\zz\zz\zz\zz;
      \draw (3,9) -- ++(1,0) \zz\zz\zz\zz\zz\zz -- ++(0,-1);
      \draw [very thick] (5,7) -- ++(0,-2) -- ++(-1,0) -- ++(0,-1) -- ++(-2,0)
      -- ++(0,-2) -- ++(-4,0);
      \draw [very thick] (0,0) -- ++(0,1) -- ++(1,0) -- ++(0,1) -- ++(2,0)
      -- ++(0,1) -- ++(1,0) -- ++(0,1) -- ++(2,0) -- ++(0,1) -- ++(1,0);
      \tikzstyle{connect} = [line width=2pt,line cap=round,pink];
      \draw [connect] (4.5,5) -- ++(0,-1);
      \draw [connect] (3.5,4) -- ++(0,-1);
      \draw [connect] (2.5,4) -- ++(0,-2);
      \draw [connect] (1.5,2) -- ++(0,0);
      \draw [connect] (0.5,2) -- ++(0,-1);
      \draw [connect] (-.5,2) -- ++(0,-1) arc (180:270:.5);
      \draw [connect] (-1.5,2) arc (180:270:.5) -- ++(2,0);
      \draw [connect] (2,2.5) -- ++(1,0);
      \draw [connect] (2,3.5) -- ++(2,0);
      \draw [connect] (4,4.5) -- ++(2,0);
    \end{tikzpicture}
  \end{center}
  \medskip
  The shapes $\mu$ and $\la$ correspond to $\sQ =
  \{3,5,6,9,10,11,12,17,18,21,23,24\}$ and $\sP =
  \{1,3,4,6,8,9,11,13,15,18,20,23\}$. We obtain $\Gamma_d(X_\la,X^\mu) =
  X_{\sP'}^{\sQ'}$, where $\sQ'$ and $\sP'$ are determined by the shape of
  $M_{\sP'}^{\sQ'}$:
  \medskip
  \[
    \newcommand{\nd}{\cdot}
    \newcommand{\nS}{\star}
    \newcommand{\md}{&\cdot}
    \newcommand{\mS}{&\star}
    \left[
    \begin{array}{c|c|cc|cccc|c|cccccc|c|cccc|ccc|c}
      \nd\md\mS\mS\md\md\md\md\md\md\md\md\md\md\md\md\md\md\md\md\md\md\md\md\\
      \nd\md\md\md\mS\mS\md\md\md\md\md\md\md\md\md\md\md\md\md\md\md\md\md\md\\
      \nd\md\md\md\md\mS\mS\mS\md\md\md\md\md\md\md\md\md\md\md\md\md\md\md\md\\
      \nd\md\md\md\md\md\md\md\mS\md\md\md\md\md\md\md\md\md\md\md\md\md\md\md\\
      \nd\md\md\md\md\md\md\md\md\mS\mS\md\md\md\md\md\md\md\md\md\md\md\md\md\\
      \nd\md\md\md\md\md\md\md\md\md\mS\mS\mS\md\md\md\md\md\md\md\md\md\md\md\\
      \nd\md\md\md\md\md\md\md\md\md\md\mS\mS\mS\mS\md\md\md\md\md\md\md\md\md\\
      \nd\md\md\md\md\md\md\md\md\md\md\md\md\md\md\md\mS\mS\md\md\md\md\md\md\\
      \nd\md\md\md\md\md\md\md\md\md\md\md\md\md\md\md\md\mS\mS\mS\md\md\md\md\\
      \nd\md\md\md\md\md\md\md\md\md\md\md\md\md\md\md\md\md\md\md\mS\mS\mS\md\\
    \end{array}
    \right] \smallskip
  \]
  This diagram has $12 - R(\theta) - N(\theta) = 1$ lone stars, and $N(\theta) =
  1$ quadratic components. The unique quadratic component is $(4,8)$. The rows
  of $M_{\sP'}^{\sQ'}$ correspond to the connectors in $\theta$, see the proof of
  \Corollary{oigamma}. Rows 6, 7, and 10 are movable.
\end{example}

\targetsec{hfunc}{%
Consider a complete intersection $Y \subset \bP^{a+b}$ of dimension $b$, defined
by $a$ quadratic equations. The $K$-theory class of $Y$ is $[\cO_Y] =
(2H-H^2)^a$, where $H \in K(\bP^{a+b})$ is the hyperplane class. It follows that
the sheaf Euler characteristic of $Y$ is given by $\chi(\cO_Y) = h(a,b)$, where
$h : \N \times \Z \to \Z$ is defined by \cite[\S4]{buch.ravikumar:pieri}
\begin{equation}\label{eqn:hfcn}
  h(a,b) \,=\, \sum_{j=0}^b (-1)^j\, 2^{a-j} \binom{a}{j} \,.
\end{equation}
}
Here we set $\binom{a}{j} = 0$ unless $0 \leq j \leq a$. Notice that for $b \geq
a$ we have $h(a,b) = (2-1)^a = 1$, and $h(a,b)=0$ for $b<0$. We record for later
the identity
\begin{equation}\label{eqn:hbinom}%
  h(a+1,b) + h(a,b-1) = 2\, h(a,b) \,,
\end{equation}
which follows from the binomial formula. The following result is the quantum
generalization of \cite[Prop.~5.3]{buch.ravikumar:pieri}.

\begin{cor}\label{cor:pierigw}%
  The $K$-theoretic Gromov-Witten invariants of $X = \LG(n,2n)$ of Pieri type
  are given by $I_d(\cO_\la, \cO^\mu, \cO^p) = h(N(\theta), R(\theta)-p)$, with
  $\theta = \la[d]/\mu$.
\end{cor}
\begin{proof}
  Let $L \subset \bP^{2n-1}$ be the $B^-$-stable linear subspace of dimension
  $n-p$. Then $\pi : \eta^{-1}(L) \to X^p$ is a birational isomorphism, so
  $\cO^p = \pi_*(\eta^*([\cO_L]))$. Using
  \cite[Thm.~4.1]{buch.chaput.ea:projected}, the projection formula, and
  \Corollary{oigamma}, we obtain
  \[
    \begin{split}
      I_d(\cO_\la, \cO^\mu, \cO^p) \
      &= \ \euler{X}([\cO_{\Gamma_d(X_\la,X^\mu)}] \cdot \pi_*\eta^*[\cO_L]) \\
      &= \ \euler{\bP^{2n-1}}([\cO_{\eta(\pi^{-1}(\Gamma_d(X_\la,X^\mu)))}]
      \cdot [\cO_L]) \,.
    \end{split}
  \]
  If $R(\theta) \leq n$, then this is the sheaf Euler characteristic of a
  complete intersection of dimension $R(\theta)-p$ defined by $N(\theta)$
  quadratic equations as well as linear equations in $\bP^{2n-1}$, which proves
  the result. Finally, if $R(\theta)=n+1$, then $I_d(\cO_\la,\cO^\mu,\cO^p) =
  h(N(\theta),R(\theta)-p) = 1$, so the corollary also holds in this case.
\end{proof}

\subsection{Quantum multiplication by special Schubert classes}\label{sec:multspec}%

We finish this section by proving some preliminary formulas for quantum products
with special Schubert classes. We start with the undeformed product $\cO^p \odot
\cO^\mu$, see \Section{qktheory} or \cite[\S2.5]{buch.chaput.ea:chevalley}.

\targetsec{theta-circ}{%
Given a skew shape $\theta \subset \wh\cP_X$, let $\theta^\circ \subset \theta$
be the skew shape obtained by removing all maximal boxes from $\theta$ that do
not belong to the north-east diagonal of $\wh\cP_X$.}

\begin{center}
  \begin{tikzpicture}[x=.4cm,y=.4cm]
    \def\yy{-- ++(0,-1) -- ++(1,0)}
    \def\zz{-- ++(0,1) -- ++(-1,0)}
    \draw (11,12) \yy\yy\yy\yy\yy -- ++(0,-1);
    \draw (9,2) \zz\zz\zz\zz\zz -- ++(0,1);
    \draw [very thick] (12,10) -- ++(-2,0) -- ++(0,-1) -- ++(-2,0) -- ++(0,-1) -- ++(-1,0) -- ++(0,-2) -- ++(-1,0);
    \draw [very thick] (14,8) -- ++(-3,0) -- ++(0,-1) -- ++(-1,0) -- ++(0,-1)
    -- ++(-2,0) -- ++(0,-2);
    \draw (10,7) -- ++(1,0) -- ++(0,1) -- ++(-1,0) -- cycle -- ++(1,1) ++(-1,0) -- ++(1,-1);
    \draw (9,6) -- ++(1,0) -- ++(0,1) -- ++(-1,0) -- cycle -- ++(1,1) ++(-1,0) -- ++(1,-1);
    \draw (7,4) -- ++(1,0) -- ++(0,1) -- ++(-1,0) -- cycle -- ++(1,1) ++(-1,0) -- ++(1,-1);
  \end{tikzpicture}
\end{center}

\targetsec{cH-func}{%
For $p \in \Z$ we then define
\begin{equation}\label{eqn:Hdef}%
  \cH(\theta,p) \,=\, \sum_{\theta^\circ \subset \varphi \subset \theta}
  (-1)^{|\theta|-|\varphi|}\, h(N(\varphi),R(\varphi)-p) \,,
\end{equation}
the sum over all subsets $\varphi$ of $\theta$ that contain $\theta^\circ$.}

\begin{prop}\label{prop:qklgpieri.undeformed}%
  For any shape $\mu \subset \wh\cP_X$ and $1 \leq p \leq n$, we have
  \[
    \cO^p \odot \cO^\mu \,=\, \sum_\nu \cH(\nu/\mu, p)\, \cO^\nu
  \]
  in $\QK(X)_q$, where the sum is over all shapes $\nu \subset \wh\cP_X$
  containing $\mu$.
\end{prop}
\begin{proof}
  Given a shape $\nu \subset \cP_X$ we let $\cI_\nu \in K(X)$ denote the dual
  element of $\cO^\nu$, defined by $\euler{X}(\cI_\nu \cdot \cO^\la) =
  \delta_{\nu,\la}$ for all shapes $\la \subset \cP_X$. We have
  \cite[Lemma~3.5]{buch.ravikumar:pieri}
  \[
    \cI_\nu = \sum_{\nu/\ka\text{ rook strip}}
    (-1)^{|\nu/\ka|}\, \cO_\ka \,,
  \]
  where the sum is over all shapes $\ka \subset \nu$ such that $\nu/\ka$ is a
  rook strip, that is, $\nu/\ka$ has at most one box in each row and column.
  Assume that $\mu \subset \cP_X$ is a classical shape. By \Corollary{pierigw}
  and equation \eqn{Hdef} we have
  \[
    \begin{split}
      I_d(\cO^p, \cO^\mu, \cI_\nu)
      &=
      \sum_{\nu/\ka\text{ rook strip}}
      (-1)^{|\nu/\ka|}\, I_d(\cO^p, \cO^\mu, \cO_\ka) \\
      &=
      \sum_{\nu/\ka\text{ rook strip}}
      (-1)^{|\nu/\ka|}\, h\big(N(\ka[d]/\mu), R(\ka[d]/\mu)-p\big) \\
      &=\,
      \cH(\nu[d]/\mu, p) \,,
    \end{split}
  \]
  where the sums are over all shapes $\ka \subset \cP_X$ such that $\mu \subset
  \ka[d] \subset \nu[d]$ and $\nu/\ka$ is a rook strip. By the definition of the
  undeformed product \cite[\S2.5]{buch.chaput.ea:chevalley}, we obtain
  \[
    \cO^p \odot \cO^\mu
    = \sum_{\nu,d} I_d(\cO^p,\cO^\mu,\cI_\nu)\, q^d\, \cO^\nu
    = \sum_{\nu,d} \cH(\nu[d]/\mu,p)\, \cO^{\nu[d]} \,,
  \]
  with the sum over $\nu \subset \cP_X$ and $d \geq 0$ such that $\mu \subset
  \nu[d]$. The proposition is equivalent to this identity.
\end{proof}

\targetsec{theta-minus}{%
We next consider the associative quantum product $\cO^p \star \cO^\mu$. Given a
skew shape $\theta \subset \wh\cP_X$, let $\theta^- \subset \theta$ be the skew
shape obtained by removing the maximal box on the north-east diagonal, if any,
as well as any boxes in the same row that do not have a box immediately below
them in $\theta$.}

\begin{center}
  \begin{tikzpicture}[x=.4cm,y=.4cm]
    \def\yy{-- ++(0,-1) -- ++(1,0)}
    \def\zz{-- ++(0,1) -- ++(-1,0)}
    \draw (11,12) \yy\yy\yy\yy\yy -- ++(0,-1);
    \draw (9,2) \zz\zz\zz\zz\zz -- ++(0,1);
    \draw [very thick] (12,10) -- ++(-2,0) -- ++(0,-1) -- ++(-2,0) -- ++(0,-1) -- ++(-1,0) -- ++(0,-2) -- ++(-1,0);
    \draw [very thick] (14,8) -- ++(-3,0) -- ++(0,-1) -- ++(-1,0) -- ++(0,-1)
    -- ++(-2,0) -- ++(0,-2);
    \draw (11,8) -- ++(1,0) -- ++(0,1) -- ++(-1,0) -- cycle -- ++(1,1) ++(-1,0) -- ++(1,-1);
    \draw (12,8) -- ++(1,0) -- ++(0,1) -- ++(-1,0) -- cycle -- ++(1,1) ++(-1,0) -- ++(1,-1);
    \draw (13,8) -- ++(1,0) -- ++(0,1) -- ++(-1,0) -- cycle -- ++(1,1) ++(-1,0) -- ++(1,-1);
  \end{tikzpicture}
\end{center}

\targetsec{cN-hat}{%
For $p \in \Z$ we then define
\begin{equation}\label{eqn:Nhat}
  \wh\cN(\theta,p) \,=\, \cH(\theta,p) -
  \sum_{\theta^- \subset \varphi \subsetneq \theta} \cH(\varphi, p) \,,
\end{equation}
the sum over all proper lower order ideals $\varphi$ in $\theta$ that contain
$\theta^-$.}

We will prove in \Corollary{lgqkpieri.proof} that $\wh\cN(\theta,p) =
\cN(\theta,p)$ holds for all skew shapes $\theta \subset \wh\cP_X$ and $p \in
\Z$, that is, $\wh\cN(\theta,p)$ is equal to $(-1)^{|\theta|-p}$ times the
number of QKLG-tableaux of shape $\theta$ with content $\{1,2,\dots,p\}$.
\Theorem{lg:qkpieri} is therefore equivalent to the following statement.

\begin{prop}\label{prop:qklgpieri.geom}%
  For any shape $\mu \subset \wh\cP_X$ and $1 \leq p \leq n$, we have
  \[
    \cO^p \star \cO^\mu \,=\, \sum_\nu \wh\cN(\nu/\mu, p)\, \cO^\nu
  \]
  in $\QK(X)_q$, where the sum is over all shapes $\nu \subset \wh\cP_X$
  containing $\mu$.
\end{prop}
\begin{proof}
  For any shape $\la \subset \wh\cP_X$, set $\la^+ = \la \cup I(q^{d+1})$, where
  $d \in \Z$ is maximal with $I(q^d) \subset \la$. In other words, $\la^+
  \subset \wh\cP_X$ is the smallest shape that contains $\la$ and contains one
  more box than $\la$ on the north-east diagonal of $\wh\cP_X$. We then have
  $q\, \psi(\cO^\la) = \cO^{\la^+}$, where $\psi$ is the line neighborhood
  operator from \Section{qktheory}. It therefore follows from
  \Proposition{qklgpieri.undeformed} that the coefficient of $\cO^\nu$ in the
  product
  \begin{equation}\label{eqn:Nhatcoef}
    \cO^p \star \cO^\mu \,=\,
    \cO^p \odot \cO^\mu - q\, \psi(\cO^p \odot \cO^\mu)
  \end{equation}
  is equal to
  \[
    \cH(\nu/\mu,p) -
    \sum_{\la:\, \mu \subset \la \text{ and } \la^+=\nu} \cH(\la/\mu,p)
    \ = \ \wh\cN(\nu/\mu,p) \,,
  \]
  as required.
\end{proof}

\begin{remark}
  The constants $\wh\cN(\theta,p)$ have alternating signs by
  \Corollary{lgqkpieri.proof}, but the constants $\cH(\theta,p)$ do not have
  easily predictable signs.
\end{remark}


\section{Combinatorial Identities}\label{sec:combin}%

In this section we complete the proof of \Theorem{lg:qkpieri}. Let $X =
\LG(n,2n)$ be a Lagrangian Grassmannian. Any shape $\la \subset \cP_X$ and
integer $1 \leq p \leq n$ define three products
\[
  \begin{split}
    \cO^p \cdot \cO^\la \, &= \, \sum_\nu \cC(\nu/\la,p)\, \cO^\nu
    \ \in K(X) \,, \\
    \cO^p \odot \cO^\la \, &= \, \sum_\nu \cH(\nu/\la,p)\, \cO^\nu
    \ \in \QK(X) \,, \text{ \ and} \\
    \cO^p \star \cO^\la \, &= \, \sum_\nu \wh\cN(\nu/\la,p)\, \cO^\nu
    \ \in \QK(X) \,.
  \end{split}
\]
The first sum is over all shapes $\nu \subset \cP_X$ containing $\la$, and the
two last sums are over all quantum shapes $\nu \subset \wh\cP_X$ containing
$\la$. The constants $\cH(\theta,p)$ and $\wh\cN(\theta,p)$ are defined whenever
$\theta$ is a skew shape in $\wh\cP_X$, and these constants depend on where
$\theta$ is located in $\wh\cP_X$, including whether $\theta$ meets the two
diagonals in $\wh\cP_X$. The constants $\cH(\theta,p)$ and $\wh\cN(\theta,p)$
are therefore bound to our chosen Lagrangian Grassmannian $X = \LG(n,2n)$. On
the other hand, the constant $\cC(\theta,p)$ does not depend on any NE diagonal,
and its definition extends naturally to any (finite) skew shape $\theta$ in the
partially ordered set $\cP_X^\infty = \bigcup_m \wh\cP_{\LG(m,2m)}$, which is
unbounded in north-east direction. This is equivalent to considering
$\cC(\theta,p)$ as a structure constant of $\varprojlim K(\LG(m,2m))$. Notice
that $\cC(\theta,p) = \cH(\theta,p) = \wh\cN(\theta,p)$ holds whenever $\theta
\subset \wh\cP_X$ is disjoint from the NE diagonal.

\Theorem{lg:qkpieri} states that each quantum structure constant
$\wh\cN(\theta,p)$ is equal to the (signed) number $\cN(\theta,p)$ of
QKLG-tableaux. We prove this by showing that $\wh\cN(\theta,p)$ and
$\cN(\theta,p)$ are determined by the same recursive identities. These
identities simultaneously provide an alternative definition of these constants.
We also prove an analogous recursive definition of the undeformed structure
constants $\cH(\theta,p)$ when $\theta$ contains at most one box on the NE
diagonal of $\wh\cP_X$. Our recursive definitions refer to (quantum or
undeformed) structure constants computed in the quantum $K$-theory of smaller
Lagrangian Grassmannians $X' = \LG(n',2n')$. For this reason we will introduce
additional notation to make it easier to refer to the constants $\cH(\theta,p)$
and $\wh\cN(\theta,p)$ when $\theta$ is regarded as a skew shape in
$\wh\cP_{X'}$. We summarize this notation here and give precise definitions
below. We will regard any skew shape $\theta$ as a subset of $\cP_X^\infty$.
Suppose $\theta$ is contained in a specific set $\wh\cP_{X'}$, and we wish to
refer to the constants $\cH(\theta,p)$ and $\wh\cN(\theta,p)$ computed in
$\QK(X')$. If $\theta$ is disjoint from the NE diagonal of $\wh\cP_{X'}$, then
we can use the structure constant $\cC(\theta,p)$ of the ordinary $K$-theory
ring $K(X)$. On the other hand, if $\theta$ meets the NE diagonal of
$\wh\cP_{X'}$, then the values of $\cH(\theta,p)$ and $\wh\cN(\theta,p)$
computed in $\QK(X')$ will be denoted $\cH_q(\theta,p)$ and $\cN_q(\theta,p)$.
Equivalently, given any skew shape $\theta \subset \cP_X^\infty$, we can define
$\cH_q(\theta,p)$ and $\cN_q(\theta,p)$ as the values of $\cH(\theta,p)$ and
$\wh\cN(\theta,p)$ computed in $\QK(X')$, where $X' = \LG(n',2n')$ is the
smallest Lagrangian Grassmannian for which $\theta \subset \wh\cP_{X'}$.

\targetsec{cPXinfty}{%
Define $\cP_X^\infty = \{ (i,j) \in \Z^2 \mid i \leq j \}$, and equip this set
with the partial order defined by $(i',j') \leq (i'',j'')$ if and only if $i'
\leq i''$ and $j' \leq j''$. We will consider $\cP_X$ and $\wh\cP_X$ as subsets
of $\cP_X^\infty$. Define a \emph{skew shape} in $\cP_X^\infty$ to be any finite
subset obtained as the difference between two lower order ideals. Given a skew
shape $\theta \subset \cP_X^\infty$, let $R(\theta)$ denote the size of a
maximal rim contained in $\theta$, and let $N'(\theta)$ be the number of
components of $\theta$ that are disjoint from the SW diagonal. Let $\theta'$
denote the skew shape obtained by removing all south-east corners from $\theta$.
Given an integer $p \in \Z$, it was proved in \cite{buch.ravikumar:pieri} that
the constant $\cC(\theta,p)$ from \Definition{klgtab} is given by
\[
  \cC(\theta,p) = \sum_{\theta' \subset \varphi \subset \theta}
  (-1)^{|\theta|-|\varphi|}\, h(N'(\varphi), R(\varphi)-p) \,,
\]
where the function $h : \N \times \Z \to \Z$ is defined by \eqn{hfcn}.}

Let $\theta \subset \cP_X^\infty$ be a non-empty skew shape. Then $\theta$
contains a unique \emph{north-east box} $Q$. A skew shape in $\cP_X^\infty$ will
be called a \emph{line} if its boxes are contained in a single row or a single
column. The \emph{north-east arm} of $\theta$ is the largest line $\psi$ that
can be obtained by intersecting $\theta$ with a square whose upper-right box is
$Q$.

\begin{center}
  \begin{tikzpicture}[x=.3cm,y=.3cm]
    \begin{scope}
      \draw (0,0) -- ++(-3,0) -- ++(0,-2) -- ++(-1,0) -- ++(0,-1) -- ++(-1,0)
        -- ++(0,-1) -- ++(2,0) -- ++(0,1) -- ++(1,0) -- ++(0,2) -- ++(2,0)
        -- cycle;
      \draw [very thick] (0,0) -- ++(-2,0) -- ++(0,-1) -- ++ (2,0) -- cycle;
    \end{scope}
    \begin{scope}[shift={(9,0)}]
      \draw (0,0) -- ++(-1,0) -- ++(0,-2) -- ++(-2,0) -- ++(0,-1) -- ++(-2,0)
        -- ++(0,-1) -- ++(3,0) -- ++(0,1) -- ++(2,0) -- cycle;
      \draw [very thick] (0,0) -- ++(-1,0) -- ++(0,-2) -- ++(1,0) -- cycle;
    \end{scope}
    \begin{scope}[shift={(18,0)}]
      \draw (-1,-1) -- ++(-1,0) -- ++(0,-1) -- ++(-1,0) -- ++(0,-1) -- ++(-2,0)
        -- ++(0,-1) -- ++(3,0) -- ++(0,1) -- ++(1,0) -- cycle;
      \draw [very thick] (0,0) -- ++(-1,0) -- ++(0,-1) -- ++(1,0) -- cycle;
    \end{scope}
    \begin{scope}[shift={(27,0)}]
      \draw (0,0) -- ++(-2,0) -- ++(0,-2) -- ++(-2,0) -- ++(0,-1) -- ++(-1,0)
        -- ++(0,-1) -- ++(2,0) -- ++(0,1) -- ++(3,0) -- cycle;
      \draw [very thick] (0,0) -- ++(-1,0) -- ++(0,-1) -- ++ (1,0) -- cycle;
    \end{scope}
  \end{tikzpicture}
\end{center}

\targetsec{arm}{%
We will say that the north-east arm $\psi$ is a \emph{row} if $\theta$ contains
no box immediately below $Q$, and $\psi$ is a \emph{column} if $\theta$ contains
no box immediately to the left of $Q$. Notice that $\psi$ can be both a row and
a column (if it is a disconnected single box), and it can be neither a row nor a
column (only if $\theta$ is not a rim). We let $\wh\theta = \theta \ssm \psi$
denote the complement of the north-east arm. This set $\wh\theta$ is a skew
shape if and only if $\psi$ is a row or a column. If $\psi$ is not connected to
$\wh\theta$, then $\psi$ is not a row if and only if $\psi$ is a column with at
least two boxes, and $\psi$ is not a column if and only if $\psi$ is a row with
at least two boxes. We set $\chi(\text{true}) = 1$ and $\chi(\text{false}) = 0$.}

\begin{prop}[\cite{buch.ravikumar:pieri}]\label{prop:lgkpieri}%
  Let $\theta \subset \cP_X^\infty$ be any skew shape and let $p \in \Z$. If $\theta$
  is not a rim, then $\cC(\theta,p)=0$, and $\cC(\emptyset,p) = \chi(p \leq 0)$.
  If $\theta$ is a non-empty rim with north-east arm $\psi = \theta \ssm
  \wh\theta$ of size $a$, then $\cC(\theta,p)$ is determined by the following
  rules.\smallskip

  \noin{\rm(i)} If $\wh\theta = \emptyset$ and $\theta$ meets the SW diagonal,
  then $\cC(\theta,p) = \delta_{p,|\theta|}$ if $\theta$ is a row, and
  $\cC(\theta,p) = \delta_{p,|\theta|} - \delta_{p,|\theta|-1}$ if $\theta$ is
  not a row.\smallskip

  \noin{\rm(ii)} If $\wh\theta = \emptyset$ and $\theta$ is disjoint from the SW
  diagonal, then $\cC(\theta,p) = 2\,\delta_{p,|\theta|} - \chi(p \geq
  1)\,\delta_{p,|\theta|-1}$.\smallskip

  \noin{\rm(iii)} If $\wh\theta \neq \emptyset$ and $\psi$ is connected to
  $\wh\theta$, then $\cC(\theta,p) = \cC(\wh\theta,p-a) -
  \cC(\wh\theta,p-a+1)$.\smallskip

  \noin{\rm(iv)} If $\wh\theta \neq \emptyset$ and $\psi$ is not connected to
  $\wh\theta$, then $\cC(\theta,p) = 2\,\cC(\wh\theta,p-a) -
  2\,\cC(\wh\theta,p-a+1)$ if $a=1$, and $\cC(\theta,p) = 2\,\cC(\wh\theta,p-a)
  - 3\,\cC(\wh\theta,p-a+1) + \cC(\wh\theta,p-a+2)$ if $a \geq 2$.
\end{prop}

\targetsec{cHq-func}{%
Given a non-empty skew shape $\theta \subset \cP_X^\infty$ with north-east box
$Q$, let $N'_q(\theta) = \max(N'(\theta)-1,0)$ be the number of components of
$\theta$ that do not meet the SW diagonal and do not contain $Q$, and let
$\theta'_q = \theta' \cup Q$ be the result of removing all south-east corners
except $Q$ (in case $Q$ is a south-east corner). For $p \in \Z$ we define
\begin{equation}\label{eqn:Hq}%
  \cH_q(\theta,p) = \sum_{\theta'_q \subset \varphi \subset \theta}
  (-1)^{|\theta|-|\varphi|}\, h(N'_q(\varphi), R(\varphi)-p) \,.
\end{equation}
}

\begin{remark}\label{remark:HCHq}%
  Assume that $\theta \subset \wh\cP_X$ is a skew shape containing at most one
  box from the NE diagonal of $\wh\cP_X$, for example a rim. Then the constant
  $\cH(\theta,p)$ defined by equation \eqn{Hdef} is given by
  \[
    \cH(\theta,p) = \begin{cases}
      \cC(\theta,p) & \text{if $\theta$ is disjoint from the NE diagonal,} \\
      \cH_q(\theta,p) & \text{if $\theta$ contains one box on the NE diagonal.}
    \end{cases}
  \]
  If $\theta \subset \wh\cP_X$ contains two or more boxes from the NE diagonal,
  then $\theta$ is not a skew shape in $\cP_X^\infty$ and $\cH_q(\theta,p)$ is
  not defined. Our next result together with \Proposition{lgkpieri} provides a
  recursive definition of the constants $\cH_q(\theta,p)$.
\end{remark}

\begin{prop}\label{prop:lgopieri}%
  Let $\theta \subset \cP_X^\infty$ be any non-empty skew shape and let $p \in \Z$.
  If $\theta$ is not a rim, then $\cH_q(\theta,p) = 0$. If $\theta$ is a rim
  with north-east arm $\psi = \theta \ssm \wh\theta$ of size $a$, then
  $\cH_q(\theta,p)$ is determined by the following rules.\smallskip

  \noin{\rm(i$''$)} If $\wh\theta = \emptyset$, then $\cH_q(\theta,p) = \chi(p
  \leq |\theta|)$ if $\theta$ is a row, and $\cH_q(\theta,p) =
  \delta_{p,|\theta|}$ if $\theta$ is not a row.\smallskip

  \noin{\rm(iii$''$)} If $\wh\theta \neq \emptyset$ and $\psi$ is connected to
  $\wh\theta$, then $\cH_q(\theta,p) = \cH_q(\wh\theta,p-a)$ if $\psi$ is a row
  or $p \geq |\theta|$, and $\cH_q(\theta,p) = \cC(\wh\theta,p-a) -
  \cH_q(\wh\theta, p-a) + \cH_q(\wh\theta,p-a+1)$ if $\psi$ is a column and $p <
  |\theta|$.\smallskip

  \noin{\rm(iv$''$)} If $\wh\theta \neq \emptyset$ and $\psi$ is not connected
  to $\wh\theta$, then $\cH_q(\theta,p) = \cC(\wh\theta,p-a)$ if $\psi$ is a
  row, and $\cH_q(\theta,p) = \cC(\wh\theta,p-a) - \cC(\wh\theta,p-a+1)$ if
  $\psi$ is not a row.
\end{prop}
\begin{proof}
  If $\theta$ is not a rim, then let $B \in \theta$ be a south-east corner such
  that $\theta$ contains a box strictly north and strictly west of $B$. For any
  skew shape $\varphi$ with $\theta'_q \subset \varphi \subset \theta \ssm B$ we
  have $h(N'_q(\varphi), R(\varphi)-p) = h(N'_q(\varphi \cup B), R(\varphi
  \cup B)-p)$, which implies that $\cH_q(\theta, p) = 0$. We can therefore
  assume that $\theta$ is a non-empty rim. If $\theta = \psi$ is a row, then
  $\cH_q(\theta,p) = h(0,|\theta|-p) = \chi(p \leq |\theta|)$. If $\theta =
  \psi$ is not a row, and $B$ is the bottom box of $\theta$, then
  $\cH_q(\theta,p) = h(0, |\theta|-p) - h(0, |\theta\ssm B|-p) =
  \delta_{p,|\theta|}$.

  Assume that $\wh\theta \neq \emptyset$ and $\psi$ is a row connected to
  $\wh\theta$. Then the skew shapes occurring in \eqn{Hq} have the form $\varphi
  \cup \psi$, where $\wh\theta' = (\wh\theta)' \subset \varphi \subset
  \wh\theta$. Since $h(N'_q(\varphi \cup \psi), |\varphi \cup \psi|-p) =
  h(N'_q(\varphi), |\varphi|-p+a)$, we obtain $\cH_q(\theta,p) =
  \cH_q(\wh\theta, p-a)$.

  Assume that $\wh\theta \neq \emptyset$ and $\psi$ is a column connected to
  $\wh\theta$. If $p \geq |\theta|$, then since $h(N'_q(\varphi), |\varphi|-p)$
  is non-zero only when $p \leq |\varphi|$, we obtain
  \[
    \cH_q(\theta, p) = h(N'_q(\theta), |\theta|-p)
    = h(N'_q(\wh\theta), |\wh\theta|-p+a) = \cH_q(\wh\theta, p-a) \,.
  \]
  Assume that $p < |\theta|$ and let $B$ be the north-east box of $\wh\theta$.
  Then
  \[
    \cH_q(\theta,p) - \cC(\wh\theta,p-a)
    + \cH_q(\wh\theta,p-a) - \cH_q(\wh\theta,p-a+1)
  \]
  is equal to the sum over all skew shapes $\varphi$, with
  $\wh\theta' \subset \varphi \subset \wh\theta \ssm B$, of
  $(-1)^{|\wh\theta|-|\varphi|}$ times
  \begin{equation}\label{eqn:Hqrec}%
    \begin{split}
    & h(N'_q(\varphi \cup \psi), |\varphi\cup\psi|-p)
    - h(N'_q(\varphi\cup B\cup\psi), |\varphi\cup B\cup\psi|-p) \\
    & - h(N'(\varphi), |\varphi|-p+a)
    + h(N'(\varphi\cup B), |\varphi\cup B|-p+a) \\
    & - h(N'_q(\varphi\cup B), |\varphi\cup B|-p+a)
    + h(N'_q(\varphi\cup B), |\varphi\cup B|-p+a-1) \,.
    \end{split}
  \end{equation}
  Using that
  \[
    N'_q(\varphi \cup \psi) = N'(\varphi \cup B) = N'(\varphi)
    \text{ \ \ and \ \ }
    N'_q(\varphi \cup B \cup \psi) = N'_q(\varphi \cup B) = N'_q(\varphi) \,,
  \]
  it follows that \eqn{Hqrec} is equal to
  \[
    h(N'(\varphi), |\varphi|-p+a+1)
    + h(N'_q(\varphi), |\varphi|-p+a)
    - 2\,h(N'_q(\varphi), |\varphi|-p+a+1) \,.
  \]
  If $N'(\varphi) > 0$, then this expression is zero by identity \eqn{hbinom}.
  Otherwise we have $N'(\varphi) = N'_q(\varphi) = 0$, which implies $\varphi =
  \wh\theta\ssm B$, so the expression is zero because $|\varphi|-p+a \geq 0$.

  We finally assume that $\wh\theta \neq \emptyset$ and $\psi$ is not connected
  to $\wh\theta$. If $\psi$ is a row, then
  \[
    \begin{split}
    \cH_q(\theta,p) \
    &= \
    \sum_{\wh\theta' \subset \varphi \subset \wh\theta}
    (-1)^{|\theta|-|\varphi\cup\psi|}\,
    h(N'_q(\varphi\cup\psi), |\varphi\cup\psi|-p) \\
    &= \
    \sum_{\wh\theta' \subset \varphi \subset \wh\theta}
    (-1)^{|\wh\theta|-|\varphi|}\,
    h(N'(\varphi), |\varphi|-p+a)
    \ = \
    \cC(\wh\theta, p-a) \,.
    \end{split}
  \]
  If $\psi$ is not a row, $B$ is the bottom box of $\psi$, and $\psi' = \psi \ssm B$, then
  \[
    \begin{split}
    & \cH_q(\theta,p) \\
    & \ =
    \sum_{\wh\theta' \subset \varphi \subset \wh\theta}
    (-1)^{|\theta|-|\varphi \cup \psi|} \left(
      h(N'_q(\varphi \cup \psi), |\varphi \cup \psi|-p)
      - h(N'_q(\varphi \cup \psi'), |\varphi \cup \psi'|-p)
      \right) \\
    & \ =
    \sum_{\wh\theta' \subset \varphi \subset \wh\theta}
    (-1)^{|\wh\theta|-|\varphi|} \left(
    h(N'(\varphi), |\varphi|-p+a)
    - h(N'(\varphi), |\varphi|-p+a-1)
    \right) \\
    & \ = \
    \cC(\wh\theta, p-a) - \cC(\wh\theta, p-a+1) \,.
  \end{split}
  \]
  This completes the proof.
\end{proof}

\begin{example}
  For any skew shape $\theta = \tableau{6}{&{}\\&{}\\{}&{}} \subset
  \cP_X^\infty$ and $p \leq 2$, we obtain $\wh\theta = \tableau{6}{{}&{}}$ and
  \[
    \cH_q(\theta,p) = \cC(\wh\theta,p-2) - \cH_q(\wh\theta,p-2) +
    \cH_q(\wh\theta,p-1) = 0 - 1 + 1 = 0 \,.
  \]
  This illustrates that negative values of $p$ must be allowed in
  \Proposition{lgopieri} to obtain correct recursive identities without
  including additional special cases.
\end{example}

\begin{defn}\label{defn:lgqkpieri}%
  Given a non-empty skew shape $\theta \subset \cP_X^\infty$ and $p \in \Z$, define
  an integer $\cN_q(\theta,p)$ as follows. If $\theta$ is not a rim, then
  $\cN_q(\theta,p) = 0$. If $\theta$ is a rim with north-east arm $\psi = \theta
  \ssm \wh\theta$ of size $a$, then $\cN_q(\theta,p)$ is determined by the
  following rules.\smallskip

  \noin{\rm(i$'$)} If $\wh\theta = \emptyset$ and $\theta$ meets the SW
  diagonal, then $\cN_q(\theta,p) = \delta_{p,|\theta|}$.\smallskip

  \noin{\rm(ii$'$)} If $\wh\theta = \emptyset$ and $\theta$ is disjoint from the
  SW diagonal, then $\cN_q(\theta,p) = \delta_{p,|\theta|}$ if $\theta$ is a
  column, and $\cN_q(\theta,p) = \delta_{p,|\theta|} - \delta_{p,|\theta|-1}$ if
  $\theta$ is not a column.\smallskip

  \noin{\rm(iii$'$)} If $\wh\theta \neq \emptyset$ and $\psi$ is connected to
  $\wh\theta$, then $\cN_q(\theta,p) = \cN_q(\wh\theta,p-a)$ if $\psi$ is a
  column, and $\cN_q(\theta,p) = \cN_q(\wh\theta,p-a) - \cC(\wh\theta,p-a+1)$ if
  $\psi$ is a row.\smallskip

  \noin{\rm(iv$'$)} If $\wh\theta \neq \emptyset$ and $\psi$ is not connected to
  $\wh\theta$, then $\cN_q(\theta,p) = \cC(\wh\theta, p-a) - \cC(\wh\theta,
  p-a+1)$ if $\psi$ is a column, and $\cN_q(\theta,p) = \cC(\wh\theta, p-a) -
  2\, \cC(\wh\theta, p-a+1) + \cC(\wh\theta, p-a+2)$ if $\psi$ is not a column.
\end{defn}

Recall from \Definition{qklgtab} that $|\cN(\theta,p)|$ is the number of
QKLG-tableaux of shape $\theta \subset \wh\cP_X$ with content $\{1,2,\dots,p\}$.

\begin{lemma}\label{lemma:qklg-count}%
  Let $\theta \subset \wh\cP_X$ be a rim meeting the NE diagonal of $\wh\cP_X$
  and let $p \in \Z$. Then $\cN_q(\theta,p) = \cN(\theta,p)$.
\end{lemma}
\begin{proof}
  Let $\psi = \theta \ssm \wh\theta$ be the north-east arm and set $a = |\psi|$.
  The pictures in this proof will be drawn for the case $a=4$. Assume first that
  $\wh\theta = \emptyset$. If $\theta$ meets the SW diagonal of $\wh\cP_X$, then
  there exists only one QKLG-tableau of shape $\theta$, which is one of the
  following cases:
  \[
    \tableau{12}{{1}&{2}&{3}&{}{a}}
    \text{ \ \ \ \ \ or \ \ \ \ \ }
    \tableau{12}{{1'}\\{2'}\\{3'}\\{a}}
  \]
  If $\theta$ is disjoint from the SW diagonal, then there is a unique QKLG-tableau of shape $\theta$ when $\theta$ is a column or a single box, and exactly two QKLG-tableaux of shape $\theta$ when $\theta$ is a row with at least two boxes:
  \[
    \tableau{12}{{1'}\\{2'}\\{3'}\\{a'}}
    \text{ \ \ \ \ \ or \ \ \ \ \ }
    \tableau{12}{{1'}&{2}&{3}&{a}}
    \text{ \ \ and \ \ }
    \tableau{12}{{1'}&{1}&{2}&{b}} \text{\ , \ \ where $b=a-1$.}
  \]
  This accounts for cases (i$'$) and (ii$'$) of \Definition{lgqkpieri}.

  Assume next that $\wh\theta \neq \emptyset$ and that $\psi$ is connected to
  $\wh\theta$. Then any QKLG-tableau of shape $\theta$ and content
  $\{1,\dots,p\}$ must assign the following labels to the boxes of $\psi$ (with
  $b_i = p-a+i$):
  \[
    \tableau{12}{&{1'}\\&{2'}\\&{3'}\\&{a'}\\{}&{p}}
    \text{ \ \ \ \ \ or \ \ \ \ \ }
    \tableau{12}{{1'}&{b_1}&{b_2}&{b_3}&{p}\\{}}
  \]
  The pictures also show two of the boxes from $\wh\theta$. If $\psi$ is a
  column, then $a'$ must be an unrepeated quantum box, so the labels of
  $\wh\theta$ can be any QKLG-tableau with content $\{a+1,a+2,\dots,p\}$ (with
  $p$ considered on the NE diagonal). If $\psi$ is a row, then the labels of
  $\wh\theta$ must have content either $\{1,2,\dots,p-a\}$ or
  $\{1,2,\dots,p-a+1\}$. In the first case $b_1$ is an unrepeated quantum box,
  so $1'$ is also a quantum box, and the labels of $\wh\theta$ can be any
  QKLG-tableau with content $\{1,2,\dots,p-a\}$ (with $1'$ considered on the NE
  diagonal). In the second case $b_1$ is repeated, $1'$ is not a quantum box, so
  the labels of $\wh\theta$ can be any KLG-tableau with content
  $\{1,2,\dots,p-a+1\}$. This accounts for case (iii$'$) of
  \Definition{lgqkpieri}.

  Finally, assume that $\wh\theta \neq \emptyset$ and $\psi$ is not connected to
  $\wh\theta$. Then any QKLG-tableau of shape $\theta$ and content
  $\{1,\dots,p\}$ must assign the following labels to the boxes of $\psi$ (with
  $b_i = p-a+i$):
  \[
    \tableau{12}{{1'}\\{2'}\\{3'}\\{a'}}
    \text{ \ \ \ \ \ or \ \ \ \ \ }
    \tableau{12}{{1'}&{b_2}&{b_3}&{p}}
  \]
  If $\psi$ is a column or a single box, then the labels of $\wh\theta$ must
  form a KLG-tableau with content $\{a+1,a+2,\dots,p\}$ or $\{a,a+1,\dots,p\}$.
  If $\psi$ is a row with at least two boxes, then the labels of $\wh\theta$
  must form a KLG-tableau with content $\{2,3,\dots,p-a+1\}$,
  $\{1,2,\dots,p-a+1\}$, $\{2,3,\dots,p-a+2\}$, or $\{1,3,\dots,p-a+2\}$. This
  accounts for case (iv$'$) of \Definition{lgqkpieri}.
\end{proof}

\begin{lemma}\label{lemma:rec:basic}%
  Let $\theta \subset \cP_X^\infty$ be a non-empty rim and let $p \in \Z$.\smallskip

  \noin{\rm(a)} \ For $p \leq 0$ we have $\cC(\theta,p) = \cN_q(\theta,p) = 0$,
  and
  \[
    \cH_q(\theta,p) = \begin{cases}
      1 & \text{if $\theta$ is a single row of boxes,} \\
      0 & \text{otherwise.}
    \end{cases}
  \]

  \noin{\rm(b)} \ We have $\cC(\theta,|\theta|) = 2^{N'(\theta)}$ and\,
  $\cN_q(\theta,|\theta|) = \cH_q(\theta,|\theta|) =
  2^{N'_q(\theta)}$.\smallskip

  \noin{\rm(c)} \ For $p > |\theta|$ we have $\cC(\theta,p) = \cN_q(\theta,p) =
  \cH_q(\theta,p) = 0$.
\end{lemma}
\begin{proof}
  These identities follow from the recursive definitions by induction on
  $|\theta|$.
\end{proof}

\begin{lemma}\label{lemma:rec:nncc}%
  Let $\theta \subset \cP_X^\infty$ be a non-empty rim, such that the north-east arm
  $\psi = \theta\ssm\wh\theta$ is not a disconnected single box, and let $p <
  |\theta|$. Then,
  \[
    2 \cN_q(\theta,p) - \cN_q(\theta,p+1) \ = \
    \begin{cases}
      \cC(\theta,p) - \cC(\theta,p+1) & \text{if $\psi$ is a row,} \\
      \cC(\theta,p) & \text{if $\psi$ is a column.}
    \end{cases}
  \]
\end{lemma}
\begin{proof}
  Assume that $\wh\theta = \emptyset$. If $\psi$ meets the SW diagonal or is a
  column, then both sides of the identity are equal to $-\delta_{p+1,|\theta|}$,
  and otherwise both sides are equal to $-3\,\delta_{p+1,|\theta|} +
  \delta_{p+2,|\theta|}$.

  Assume next that $\wh\theta \neq \emptyset$ and $\psi$ is connected to
  $\wh\theta$. Set $a = |\psi|$. If $\psi$ is a row, then
  \begin{multline*}
    2\, \cN_q(\theta,p) - \cN_q(\theta,p+1) - \cC(\theta,p) + \cC(\theta,p+1) \\
    = \ 2\, \cN_q(\wh\theta,p-a) - \cN_q(\wh\theta,p-a+1) - \cC(\wh\theta,p-a)
    \,,
  \end{multline*}
  which vanishes by induction on $|\theta|$, since the north-east arm of
  $\wh\theta$ is a column. If $\psi$ is a column, then
  \begin{multline*}
    2\, \cN_q(\theta,p) - \cN_q(\theta,p+1) - \cC(\theta,p) \\
    = \ 2\, \cN_q(\wh\theta,p-a) - \cN_q(\wh\theta,p-a+1) - \cC(\wh\theta,p-a)
    + \cC(\wh\theta,p-a+1)
  \end{multline*}
  which vanishes by induction on $|\theta|$, since the north-east arm of
  $\wh\theta$ is a row.

  Finally we assume that $\wh\theta \neq \emptyset$ and $\psi$ is not connected
  to $\wh\theta$. If $\psi$ is a column, then both sides are equal to
  \[
    2\, \cC(\wh\theta, p-a) - 3\, \cC(\wh\theta, p-a+1)
    + \cC(\wh\theta, p-a+2) \,,
  \]
  and if $\psi$ is a row, then both sides are equal to
  \[
    2\, \cC(\wh\theta, p-a) - 5\, \cC(\wh\theta, p-a+1)
    + 4\, \cC(\wh\theta, p-a+2) - \cC(\wh\theta, p-a+3) \,.
  \]
  The identity follows from this.
\end{proof}

\begin{lemma}\label{lemma:rec:hhcn}%
  Let $\theta \subset \cP_X^\infty$ be a non-empty rim and let $p < |\theta|$.
  Then,
  \[
    \cH_q(\theta,p) - \cH_q(\theta,p+1) \ = \ \cC(\theta,p) - \cN_q(\theta,p)
    \,.
  \]
\end{lemma}
\begin{proof}
  Let $\psi = \theta \ssm \wh\theta$ be the north-east arm of $\theta$ and set
  $a = |\psi|$. Assume first that $\wh\theta = \emptyset$. If $\psi$ is a row,
  then both sides of the identity are zero, and otherwise both sides are equal
  to $-\delta_{p+1,|\theta|}$.

  Assume next that $\wh\theta \neq \emptyset$ and $\psi$ is connected to
  $\wh\theta$. If $\psi$ is a row, then it follows by induction on $|\theta|$
  that
  \begin{multline*}
    \cH_q(\theta,p) - \cH_q(\theta,p+1)
    = \cH_q(\wh\theta,p-a) - \cH_q(\wh\theta,p-a+1) \\
    = \cC(\wh\theta,p-a) - \cN_q(\wh\theta,p-a)
    = \cC(\theta,p) - \cN_q(\theta,p) \,.
  \end{multline*}
  If $\psi$ is a column and $p \leq |\theta|-2$, then the recursive definitions
  and induction on $|\theta|$ yield
  \[
    \begin{split}
      & \cH_q(\theta,p) - \cH_q(\theta,p+1) - \cC(\theta,p) + \cN_q(\theta,p) \\
      & \ \ = \ -\cH_q(\wh\theta,p-a)
      + 2\,\cH_q(\wh\theta,p-a+1) - \cH_q(\wh\theta,p-a+2)
      + \cN_q(\wh\theta,p-a) \\
      & \ \ = \ 2\,\cN_q(\wh\theta,p-a) - \cN_q(\wh\theta,p-a+1)
      - \cC(\wh\theta,p-a) + \cC(\wh\theta,p-a+1) \,.
    \end{split}
  \]
  This expression is equal to zero by \Lemma{rec:nncc}, as the north-east arm of
  $\wh\theta$ is a row. If $\psi$ is a column and $p = |\theta|-1$, then
  the recursive definitions and induction on $|\theta|$ gives
  \[
    \begin{split}
      & \cH_q(\theta,p) - \cH_q(\theta,p+1) - \cC(\theta,p) + \cN_q(\theta,p) \\
      & \ \ = \
      \cN_q(\wh\theta, p-a) - \cH_q(\wh\theta, p-a) + \cC(\wh\theta, p-a+1) \\
      & \ \ = \
      2\, \cN_q(\wh\theta, p-a) - \cH_q(\wh\theta, p-a+1) -
      \cC(\wh\theta, p-a) + \cC(\wh\theta, p-a+1) \,.
    \end{split}
  \]
  This expression is equal to zero by \Lemma{rec:basic}(b) and \Lemma{rec:nncc},
  as $p-a+1 = |\wh\theta|$ and the north-east arm of $\wh\theta$ is a row.

  Finally assume that $\wh\theta \neq \emptyset$ and $\psi$ is not connected to
  $\wh\theta$. If $\psi$ is a row, then both sides of the identity are equal to
  $\cC(\wh\theta, p-a) - \cC(\wh\theta, p-a+1)$, and otherwise both sides are
  equal to $\cC(\wh\theta, p-a) - 2\,\cC(\wh\theta, p-a+1) + \cC(\wh\theta,
  p-a+2)$. This proves the identity.
\end{proof}

\begin{prop}\label{prop:rec:odot}%
  Let $\theta \subset \cP_X^\infty$ be a non-empty skew shape with north-east
  arm $\psi = \theta\ssm\wh\theta$, and let $p \in \Z$. Then,
  \[
    \cH_q(\theta,p) - \cN_q(\theta,p) = \begin{cases}
      \sum_{\wh\theta \subset \varphi \subsetneq \theta} \cC(\varphi,p)
      & \text{if $\psi$ is a row,} \\
      0 & \text{otherwise,}
    \end{cases}
  \]
  where the sum is over all proper lower order ideals $\varphi$ of $\theta$ that
  contain $\wh\theta$.
\end{prop}
\begin{proof}
  We may assume that $\theta$ is a rim, since otherwise $\wh\theta$ is also not
  a rim, and both sides of the identity vanish. Set $a = |\psi|$. Assume first
  that $\psi$ is not a row. If $\wh\theta = \emptyset$, then $\cH_q(\theta,p) =
  \delta_{p,|\theta|} = \cN_q(\theta,p)$. If $\wh\theta \neq \emptyset$ and
  $\psi$ is connected to $\wh\theta$, then $\cH_q(\theta,p) = \cN_q(\theta,p)$
  for $p \geq |\theta|$ by \Lemma{rec:basic}(b,c), and for $p < |\theta|$ we
  have
  \[
    \cH_q(\theta,p) - \cN_q(\theta,p) =
    \cC(\wh\theta,p-a) - \cH_q(\wh\theta,p-a) + \cH_q(\wh\theta,p-a+1)
    - \cN_q(\wh\theta,p-a) \,,
  \]
  which is equal to zero by \Lemma{rec:hhcn}. Finally, if $\wh\theta \neq
  \emptyset$ and $\psi$ is not connected to $\wh\theta$, then $\cH_q(\theta,p) =
  \cC(\wh\theta,p-a) - \cC(\wh\theta,p-a+1) = \cN_q(\theta,p)$.

  Assume that $\psi$ is a row. For $0 \leq i \leq a-1$, we let $\varphi_i$ be
  the union of $\wh\theta$ with the leftmost $i$ boxes of $\psi$. Then
  $\varphi_0, \varphi_1, \dots, \varphi_{a-1}$ are the proper lower order ideals
  $\varphi$ in $\theta$ that contain $\wh\theta$. If $\wh\theta =
  \emptyset$ and $\psi$ meets the SW diagonal, then
  \[
    \cH_q(\theta,p) - \cN_q(\theta,p)
    = \chi(p \leq |\theta|) - \delta_{p,|\theta|}
    = \chi(p \leq 0) + \sum_{i=1}^{a-1} \delta_{p,i}
    = \sum_{i=0}^{a-1} \cC(\varphi_i, p) \,.
  \]
  If $\theta$ is a single box not on the SW diagonal, then $\cH_q(\theta,p) -
  \cN_q(\theta,p) = \chi(p \leq 0) = \cC(\varphi_0, p)$. If $\wh\theta =
  \emptyset$, $|\theta| \geq 2$, and $\psi$ does not meet the SW diagonal, then
  \begin{multline*}
    \cH_q(\theta,p) - \cN_q(\theta,p)
    \ = \ \chi(p < |\theta|) + \delta_{p,a-1} \\
    = \ \chi(p \leq 0) + 2\,\delta_{p,1} +
    \sum_{i=2}^{a-1} (2\,\delta_{p,i} - \delta_{p,i-1})
    \ = \ \sum_{i=0}^{a-1} \cC(\varphi_i, p) \,.
  \end{multline*}
  If $\wh\theta \neq \emptyset$ and $\psi$ is connected to $\wh\theta$, then
  since the north-east arm of $\wh\theta$ is not a row, we obtain by induction
  on $|\theta|$ that
  \begin{multline*}
    \cH_q(\theta,p) - \cN_q(\theta,p)
    \,=\, \cH_q(\wh\theta,p-a) - \cN_q(\wh\theta,p-a) + \cC(\wh\theta,p-a+1) \\
    =\, \cC(\wh\theta, p-a+1)
    \,=\, \cC(\wh\theta,p) +
    \sum_{i=1}^{a-1} (\cC(\wh\theta,p-i) - \cC(\wh\theta,p-i+1))
    \,=\, \sum_{i=0}^{a-1} \cC(\varphi_i, p) \,.
  \end{multline*}
  If $\wh\theta \neq \emptyset$ and $\psi$ is a single box that is not connected
  to $\wh\theta$, then $\cH_q(\theta,p) - \cN_q(\theta,p) = \cC(\wh\theta,p)$
  follows from the definitions. Finally, if $\wh\theta \neq \emptyset$, $\psi$
  is not connected to $\wh\theta$, and $a \geq 2$, we obtain
  \[
    \begin{split}
      & \cH_q(\theta,p) - \cN_q(\theta,p) \ = \
      2\,\cC(\wh\theta,p-a+1) - \cC(\wh\theta,p-a+2) \\
      & \ \ = \
      \cC(\wh\theta,p) +
      \left(2\,\cC(\wh\theta,p-1) - 2\,\cC(\wh\theta,p)\right) \\
      & \ \ \ \ \ \ \ \ \
      + \sum_{i=2}^{a-1} \left(
      2\,\cC(\wh\theta,p-i) - 3\,\cC(\wh\theta,p-i+1) + \cC(\wh\theta,p-i+2)
      \right) \\
      & \ \ = \
      \sum_{i=0}^{a-1} \cC(\varphi_i,p) \,.
    \end{split}
  \]
  The identity follows from these observations.
\end{proof}

We finally prove that the Pieri structure constants $\wh\cN(\theta,p)$ of
$\QK(X)$ are signed counts of QKLG-tableaux.

\begin{cor}\label{cor:lgqkpieri.proof}%
  Let $\theta \subset \wh\cP_X$ be a skew shape and $1 \leq p \leq n$. Then
  $\wh\cN(\theta,p) = \cN(\theta,p)$.
\end{cor}
\begin{proof}
  If $\theta$ is disjoint from the NE diagonal of $\wh\cP_X$, then
  $\wh\cN(\theta,p) = \cC(\theta,p) = \cN(\theta,p)$ by
  \cite{buch.ravikumar:pieri}. If $\theta$ contains two or more boxes from the
  NE diagonal, then $\cN(\theta,p) = 0$ by definition (since $\theta$ is not a
  rim), and since $\dmax(p) = 1$, it follows from
  \cite[Thm.~8.3]{buch.chaput.ea:positivity} that $\wh\cN(\theta,p) = 0$. Assume
  that $\theta$ contains exactly one box from the NE diagonal of $\wh\cP_X$.
  Then $\theta^-$ equals $\wh\theta$ if the north-east arm of $\theta$ is a row,
  and $\theta^- = \theta$ otherwise. \Lemma{qklg-count} shows that
  $\cN(\theta,p) = \cN_q(\theta,p)$, and \Proposition{rec:odot} and the
  definition \eqn{Nhat} show that $\cN_q(\theta,p) = \wh\cN(\theta,p)$, noting
  that the condition $\theta^- \subset \varphi \subsetneq \theta$ implies that
  $\cH(\varphi,p) = \cC(\varphi,p)$ by \Remark{HCHq}.
\end{proof}


\section*{Index of symbols}

We list the most important symbols used in this paper in approximate order of
appearance after the introduction. In electronic versions of this paper, each
symbol is clickable and linked to its definition.\smallskip

\indexsec{flagvar}{%
\slink{group}{$G$}%
\slink{group}{$B$}%
\slink{group}{$T$}%
\slink{group}{$B^-$}%
\slink{group}{$W$}%
\slink{group}{$\Phi$}%
\slink{group}{$\Phi^+$}%
\slink{group}{$\Delta$}%
\slink{flagvar}{$P_X$}%
\slink{schub}{$W_X$}%
\slink{schub}{$W^X$}%
\slink{schub}{$X_w$}%
\slink{schub}{$X^w$}%
\slink{factor}{$u^X$}%
\slink{factor}{$u_X$}%
\slink{factor}{$w_0^X$}%
\slink{factor}{$w_{0,X}$}%
\slink{factor}{$u^\vee$}%
}

\indexsec{ss:comin}{%
\slink{gamma}{$\ga$}%
\slink{comin}{$u \cap v$}%
\slink{comin}{$u \cup v$}%
\slink{comin}{$\cP_X$}%
\slink{comin}{$I(w)$}%
\slink{comin}{$\delta$}%
\slink{comin}{$\delta(\al)$}%
\slink{schubla}{$X_\la$}%
\slink{schubla}{$X^\la$}%
}

\indexsec{curve-nbhd}{%
\slink{nbhd}{$M_d$}%
\slink{nbhd}{$\Mb_{0,3}(X,d)$}%
\slink{nbhd}{$\ev_i$}%
\slink{nbhd}{$\Gamma_d(X_u,X^v)$}%
\slink{nbhd}{$\Gamma_d(X_u)$}%
\slink{nbhd}{$u(d)$}%
\slink{nbhd}{$v(-d)$}%
\slink{nbhd}{$z_d$}%
\slink{dist}{$\dist(x,y)$}%
\slink{dist}{$\Gamma_d(x,y)$}%
\slink{dist}{$Y_d$}%
\slink{dist}{$Z_d$}%
\slink{dist}{$p_d$}%
\slink{dist}{$q_d$}%
\slink{dist}{$Y_d(X_u,X^v)$}%
\slink{dist}{$Z_d(X_u,X^v)$}%
\slink{dist}{$\ka_d$}%
\slink{dist}{$w_{0,Y_d}^{Z_d}$}%
}

\indexsec{qcohom}{%
\slink{qh}{$\QH(X)$}%
\slink{qh}{$[X_u] \star [X^v]$}%
\slink{qh}{$\QH(X)_q$}%
\slink{qh}{$\cB$}%
}

\indexsec{qktheory}{%
\slink{ktheory}{$K(X)$}%
\slink{ktheory}{$\cO_u$}%
\slink{ktheory}{$\cO^u$}%
\slink{ktheory}{$\cO_\la$}%
\slink{ktheory}{$\cO^\la$}%
\slink{qk}{$\QK(X)$}%
\slink{qk}{$\cO_u \odot \cO^v$}%
\slink{qk}{$\cO_u \star \cO^v$}%
\slink{qk}{$\QK(X)_q$}%
\slink{qk}{$\cB'$}%
}

\indexsec{seidel}{%
\slink{dminmax}{$\dmin(u,v)$}%
\slink{dminmax}{$\dmax(u,v)$}%
\slink{dminmax}{$\dmax(u)$}%
\slink{Wcomin}{$W^\comin$}%
\slink{qunits}{$\QH(X)_q^\times$}%
\slink{qunits}{$\QK(X)_q^\times$}%
\sref{example:qhquadric}{$\pP$}%
\sref{remark:seidel_qh}{$\om_\be^\vee$}%
}

\indexsec{qposet}{%
\slink{qbruhat}{$q^e[X^v] \leq q^d[X^u]$}%
\slink{qposet}{$\wh\cP_X$}%
\slink{qposet}{$I(q^d[X^u])$}%
\slink{qposet}{$\cO^\la$}%
\slink{embed}{$\xi(\al)$}%
\slink{embed}{$\tau(\al)$}%
\slink{embed}{$\tau(\cP_X)$}%
\slink{shift}{$\sigma\star\la$}%
\slink{shift}{$\la[d]$}%
}

\indexsec{gr:poset}{%
\slink{grass}{$\Gr(m,n)$}%
\slink{rectpart}{$(b)^a$}%
\slink{cylinder}{$(i',j') \leq (i,j)$}%
\slink{cylinder}{$\Z^2/\Z(m,m-n)$}%
}

\indexsec{gr:pieri}{%
\slink{binom}{$r(\theta)$}%
\slink{binom}{$\cA(\theta,p)$}%
}

\indexsec{og:poset}{%
\slink{og}{$\OG(n,2n)$}%
\slink{og:poset}{$\wb\cP_X$}%
}

\indexsec{og:pieri}{%
\sref{defn:kog}{$\cB(\theta,p)$}%
}

\indexsec{lg:poset}{%
\slink{lg}{$\LG(n,2n)$}%
\slink{lg:poset}{$\wb\cP_X$}%
}

\indexsec{lg:pieri}{%
\sref{defn:klgtab}{$\cC(\theta,p)$}%
\sref{defn:qklgtab}{$\cN(\theta,p)$}%
}

\indexsec{sgrass}{%
\slink{sg}{$E$}%
\slink{sg}{$(e_i,e_j)$}%
\slink{sg}{$\SG(m,2n)$}%
\slink{sg:schub}{$[a,b]$}%
\slink{sg:schub}{$E_\sP$}%
\slink{sg:schub}{$Y_\sP$}%
\slink{sg:schub}{$Y^\sP$}%
\slink{sg:bruhat}{$\sQ \leq \sP$}%
\slink{sg:length}{$\ell(\sP)$}%
\slink{sg:length}{$\sP^\vee$}%
}

\indexsec{sg:rich}{%
\slink{sgrich}{$Y_\sP^\sQ$}%
\slink{sflag}{$Z_{\sP',\sP}$}%
\slink{sflag}{$Z^{\sQ',\sQ}$}%
\slink{sflag}{$Z_{\sP',\sP}^{\sQ',\sQ}$}%
\slink{sp:weyl}{$s_i$}%
\slink{w-hat}{$\wh w$}%
\slink{w-prime}{$w'$}%
\slink{projrich}{$\Pi_\sigma^\tau(M)$}%
\slink{projrich}{$\leq_M$}%
\slink{projrich}{$\leq_L$}%
\slink{minbruhat}{$\min(\tau,\tau s_i)$}%
}

\indexsec{matspace}{%
\slink{MPQ}{$M_\sP^\sQ$}%
\slink{openrich}{$\oY_\sP^\sQ$}%
}

\indexsec{movable}{%
\slink{constraints}{$A[k]$}%
}

\indexsec{ss:incidence}{%
\slink{perp}{$S$}%
\slink{perp}{$f$}%
\slink{perp}{$g$}%
\slink{perp}{$S_\sP^\sQ$}%
\slink{cut}{$(a,b)$}%
\slink{complint}{$\bP_\sP^\sQ$}%
\slink{M-hat}{$\wh M_\sP^\sQ$}%
\slink{S-prime}{$S'$}%
\slink{S-prime}{$S''$}%
}

\indexsec{incprojrich}{%
\slink{ZSF}{$Y$}%
\slink{ZSF}{$Z$}%
\slink{ZSF}{$p$}%
\slink{ZSF}{$q$}%
\slink{ZSF}{$\wh X$}%
\slink{ZSF}{$\eta$}%
\slink{ZSF}{$\pi$}%
}

\indexsec{ss:gwpieri}{%
\slink{qskew}{$R(\la[d]/\mu)$}%
\slink{qskew}{$N(\la[d]/\mu)$}%
\slink{hfunc}{$h(a,b)$}%
}

\indexsec{multspec}{%
\slink{theta-circ}{$\theta^\circ$}%
\slink{cH-func}{$\cH(\theta,p)$}%
\slink{theta-minus}{$\theta^-$}%
\slink{cN-hat}{$\wh\cN(\theta,p)$}%
}

\indexsec{combin}{%
\slink{cPXinfty}{$\cP_X^\infty$}%
\slink{cPXinfty}{$N'(\theta)$}%
\slink{cPXinfty}{$\theta'$}%
\slink{arm}{$\psi$}%
\slink{arm}{$\wh\theta$}%
\slink{arm}{$\chi$}%
\slink{cHq-func}{$N'_q(\theta)$}%
\slink{cHq-func}{$\theta'_q$}%
\slink{cHq-func}{$\cH_q(\theta,p)$}%
\sref{defn:lgqkpieri}{$\cN_q(\theta,p)$}%
}

\ifdefined\qkpieribib
\bibliography{\qkpieribib}
\else

\fi
\bibliographystyle{halpha}

\end{document}